\documentclass[11pt,oneside]{article}
\usepackage{setspace,graphicx,amssymb,amsmath,latexsym,amsfonts,amscd,amsthm,multirow,ctable,mathdots,caption,array}
\usepackage{fancyhdr,tabularx,cite,mathrsfs}
\usepackage[headings]{fullpage}
\usepackage{stmaryrd}
\usepackage{rotating}
\usepackage{hyperref}

\newcommand\blfootnote[1]{%
  \begingroup
  \renewcommand\thefootnote{}\footnote{#1}%
  \addtocounter{footnote}{-1}%
  \endgroup
}

\numberwithin{equation}{section}
\theoremstyle{plain}
\newtheorem{theorem}{Theorem}[section]

\newtheorem{lemma}[theorem]{Lemma}
  
\newtheorem*{question}{Question}
\theoremstyle{definition}

\newtheorem*{example}{Example}
\newtheorem*{conjecture}{Conjecture}
\newtheorem{thm}{Theorem}[section]

\theoremstyle{remark}
\newtheorem*{remark}{Remark}

\usepackage{setspace,graphicx,amssymb,amsmath,latexsym,amsfonts,amscd,amsthm,multirow,ctable,mathdots,caption,array}
\usepackage{fancyhdr,tabularx,cite,mathrsfs}
\usepackage[headings]{fullpage}
\usepackage{stmaryrd}
\usepackage{rotating}

\usepackage{hyperref}

\newcommand{\bea}{\begin{eqnarray}} 
\newcommand{\eea}{\end{eqnarray}} 
\newcommand{\bee}{\begin{eqnarray*}} 
\newcommand{\eee}{\end{eqnarray*}} 
\newcommand{\al}{\begin{align*}} 
\newcommand{\eal}{\end{align*}} 
\newcommand{\be}{\begin{equation}} 
\newcommand{\ee}{\end{equation}} 
 
\newcommand{\bem}{\begin{pmatrix}} 
\newcommand{\eem}{\end{pmatrix}}

\def\c{\gamma}

\def\h{\eta} 
 
\def\im{\mathrm{Im}}

\def\m{\mu} 
\def\n{\nu}

\def\t{\tau} 
\def\th{\theta}

\def\D{\Delta}

\def\L{\Lambda}

\newcolumntype{R}{ >{$}r <{$}}
\newcolumntype{C}{ >{$}c <{$}}
\newcolumntype{L}{ >{$}l <{$}}
\newcolumntype{F}{>{\centering\arraybackslash}m{1.5cm}}

\def\ll{\ell}

\newcommand{\gt}[1]{\mathfrak{#1}}

\newcommand{\comment}[1]{}

\newcommand{\RR}{{\mathbb R}}
\newcommand{\CC}{{\mathbb C}}
\newcommand{\ZZ}{{\mathbb Z}}
\newcommand{\QQ}{{\mathbb Q}}
\newcommand{\HH}{{\mathbb H}}

\newcommand{\Aut}{\operatorname{Aut}}

\def\jac{\operatorname{j}}

\newcommand{\tr}{\operatorname{{tr}}}

\newcommand{\Sym}{{\textsl{Sym}}}

\newcommand{\Dih}{{\textsl{Dih}}}

\newcommand{\sgn}{\operatorname{sgn}}

\newcommand{\PSL}{\operatorname{\textsl{PSL}}}    
\newcommand{\SL}{\operatorname{\textsl{SL}}}      
\newcommand{\PGL}{\operatorname{\textsl{PGL}}}    
\newcommand{\AGL}{{\textsl{AGL}}}    
\newcommand{\GL}{{\textsl{GL}}}      

\newcommand{\G}{\Gamma}	
\newcommand{\g}{\gamma}	

\newcommand{\MM}{\mathbb{M}}	

\DeclareMathOperator{\lcm}{lcm}

\newcommand{\Z}{\mathbb Z}

\newcommand{\Vnat}{V^{\natural}}

\newcommand{\Lee}{\Lambda}
\renewcommand{\t}{\tau}
\usepackage{ctable}

\def\rad{\operatorname{r}}

\begin{document}

\title{Proof of the Umbral Moonshine Conjecture}


\author{John F. R. Duncan, Michael J. Griffin and Ken Ono\blfootnote{The authors thank the NSF for its support. The first author also thanks the Simons Foundation (\#316779), and the third author thanks the A. G. Candler Fund. The authors thank Miranda Cheng, Jeff Harvey, Michael Somos and the anonymous referee for helpful comments and corrections.}}

\date{2015 November 25}

\newcommand{\Addresses}{{
  \bigskip
  \footnotesize

  J. F. R.~Duncan,  
  \textsc{Dept. of Mathematics \& Computer Science, Emory University,
  Atlanta, GA 30322}\par\nopagebreak
  \textit{E-mail address}, J. F. R. Duncan: \texttt{john.duncan@emory.edu}

  \medskip

  M. J. Griffin, \textsc{Dept. of Mathematics, Princeton University,
   Princeton, NJ 08544}\par\nopagebreak
  \textit{E-mail address}, M. J. Griffin: \texttt{mjg4@princeton.edu}

  \medskip

  K.~Ono, \textsc{Dept. of Mathematics \& Computer Science, Emory University,
  Atlanta, GA 30322}\par\nopagebreak
  \textit{E-mail address}, K.~Ono: \texttt{ono@mathcs.emory.edu}

}}
\maketitle

\begin{abstract} 
The {Umbral Moonshine Conjectures} 
assert that
there are infinite-dimensional graded modules, for prescribed finite groups, whose
McKay-Thompson series are certain distinguished mock modular forms.
Gannon has proved this for the 
special case involving the largest sporadic simple Mathieu group.
Here we establish the existence of the umbral moonshine modules in the remaining 22 cases.
\end{abstract}

\clearpage

\section{Introduction and Statement of Results}\label{sec:intro}

{\it Monstrous moonshine}  relates distinguished modular functions to the representation theory of the Monster, $\MM$,
the largest sporadic simple group. This theory was inspired by the famous observations of
McKay and Thompson in the late 1970s  \cite{MR554399, Tho_NmrlgyMonsEllModFn} that 
\begin{displaymath}
\begin{split}
         196884&=1+196883,\\
	21493760&=1+196883+21296876.
	\end{split}
\end{displaymath}
The left hand sides here are familiar as  coefficients of Klein's modular function (note $q:=e^{2\pi i \tau}$),
\begin{displaymath}
J(\tau)=\sum_{n=-1}^{\infty}c(n)q^n:=
	j(\tau)-744
	=q^{-1}+196884q+21493760q^2+\dots. 
\end{displaymath}
 The sums on the right hand sides involve the first three numbers arising as dimensions of irreducible representations of $\MM$, 
 \begin{displaymath}
\begin{split}
1, \ 
196883, \
21296876, \
842609326,\ \dots, \
258823477531055064045234375.
\end{split}
\end{displaymath}

Thompson conjectured  that 
there is a 
graded infinite-dimensional $\MM$-module
$$\Vnat=\bigoplus_{n=-1}^{\infty}\Vnat_n,
$$
satisfying $\dim(\Vnat_n)=c(n)$. For $g\in \MM$,
he also suggested \cite{Tho_FinGpsModFns} to consider the graded-trace functions
\begin{displaymath}
	T_g(\tau):=\sum_{n=-1}^{\infty}\tr(g|\Vnat_n)q^n,
\end{displaymath}
now known as the {\em McKay-Thompson series},
that arise from
the conjectured $\MM$-module $\Vnat$.
Using the
character table for $\MM$, it was observed \cite{MR554399,Tho_FinGpsModFns} that the first few coefficients of each $T_g(\tau)$ coincide with those 
of a generator for the function field of a discrete group $\Gamma_g<\SL_2(\RR)$, leading
Conway and Norton \cite{MR554399} to their famous {\em Monstrous Moonshine Conjecture}: This is the claim that
for each $g\in\MM$ there is a specific {\it genus zero}
group $\Gamma_g$ such that $T_g(\tau)$ is the unique normalized {\it hauptmodul} for $\Gamma_g$, i.e., the unique $\Gamma_g$-invariant holomorphic function on $\HH$ which satisfies
$T_g(\tau)=q^{-1}+O(q)$ as $\Im(\tau)\to \infty$, and remains bounded near any cusp not equivalent to the infinite one.

In a series of ground-breaking works, Borcherds introduced vertex algebras \cite{Bor_PNAS}, and generalized Kac--Moody Lie algebras \cite{Bor_GKM,MR1120425}, and 
used these notions to prove \cite{borcherds_monstrous} the Monstrous Moonshine Conjecture of Conway and Norton.
He confirmed the conjecture for the module $\Vnat$ constructed
by Frenkel, Lepowsky, and Meurman 
\cite{FLMPNAS,  FLMBerk, FLM} in the early 1980s.
These results provide much more than the predictions of monstrous moonshine.
The $\MM$-module $\Vnat$ is a vertex operator algebra, one whose automorphism group is precisely $\MM$.
The construction of Frenkel, Lepowsky and Meurman can be regarded as one of the first examples of an {\it orbifold conformal field theory}. (Cf. \cite{MR968697}.)
Here the orbifold in question is the quotient  $\left(\RR^{24}/\Lee_{24}\right)/(\Z/2\Z)$,
of the $24$-dimensional torus $\Lee_{24}\otimes_{\Z}\RR/\Lee_{24}\simeq \RR^{24}/\Lee_{24}$ by the Kummer involution $x\mapsto -x$, where $\Lee_{24}$ denotes the Leech lattice. 

We refer to \cite{DuncanGriffinOno, FLM,MR2201600, MR2257727} for more on monstrous moonshine.

In 2010,
Eguchi, Ooguri, and Tachikawa reignited  moonshine with their observation \cite{Eguchi2010} that dimensions of some representations of
$M_{24}$, the largest sporadic simple Mathieu group (cf. e.g. \cite{ATLAS,MR1662447}), are multiplicities of superconformal algebra characters in the K3 elliptic genus. This observation suggested a manifestation of moonshine
for $M_{24}$: Namely, there should be an infinite-dimensional graded $M_{24}$-module whose McKay-Thompson series
are holomorphic parts of {\it harmonic Maass forms}, the so-called {\it mock modular forms}. (See \cite{Ono_unearthing, zagier_mock, zwegers} for introductory accounts of the theory of mock modular forms.) 

Following the work of Cheng \cite{MR2793423}, Eguchi and Hikami \cite{Eguchi2010a}, and Gaberdiel, Hohenegger, and Volpato \cite{Gaberdiel2010a, Gaberdiel2010}, 
Gannon established the existence of this infinite-dimensional graded $M_{24}$-module in \cite{Gannon:2012ck}.

It is natural to seek a general mathematical and physical setting for these results.
Here we consider the mathematical setting, which develops from the 
close relationship between the monster group $\MM$ 
and the Leech lattice $\Lambda_{24}$. 
Recall (cf. e.g. \cite{MR1662447}) that the Leech lattice is 
even, unimodular, and positive-definite of rank 24. 
It turns out that $M_{24}$ is closely related to 
another such lattice. Such observations led Cheng, Duncan and Harvey to further instances of moonshine within the setting of even unimodular positive-definite lattices of rank $24$. In this way they arrived at the {\it Umbral Moonshine Conjectures}  (cf. \S5 of \cite{UM}, \S6 of \cite{MUM}, and \S2 of \cite{mumcor}), predicting the existence of $22$ further, graded infinite-dimensional modules, relating certain finite groups to distinguished mock modular forms. 

To explain this prediction in more detail we recall Niemeier's result \cite{Niemeier} that there are
24 (up to isomorphism) even unimodular positive-definite lattices of rank $24$.
The Leech lattice is the unique one with no root vectors (i.e. lattice vectors with norm-square 2), while the other 23 have root systems with full rank, 24. These  {\em Niemeier root systems} are unions of simple simply-laced root systems with the same Coxeter numbers, and are given explicitly as
\begin{gather}\label{eqn:um-NX1}
	\begin{split}
A_1^{24},\;A_2^{12},\;A_3^{8},&\;A_4^6,\;A_6^4,\;A_{12}^2,\\
A_5^4D_4,\;A_7^2D_5^2,\;A_8^3,\;A_9^2D_6,\;&A_{11}D_7E_6,\;A_{15}D_9,\;A_{17}E_7,\;A_{24},\\
D_4^6,\;D_6^4,\;D_8^3,\;D_{10}E_7^2,\;&D_{12}^2,\;D_{16}E_8,\;D_{24},E_6^4,\;E_8^3,
	\end{split}
\end{gather}
in terms of the standard ADE notation. (Cf. e.g. \cite{MR1662447} or \cite{MR0323842} for more on root systems.)

For each Niemeier root system $X$ let $N^X$ denote the corresponding unimodular lattice, let $W^X$ denote the (normal) subgroup of $\Aut(N^X)$ generated by reflections in roots, and define the {\em umbral group} of $X$ by setting 
\begin{gather}
G^X:=\Aut(N^X) /W^X.
\end{gather}
(See \S\ref{sec:gps:cnst} for explicit descriptions of the groups $G^X$.)

Let $m^X$ denote the Coxeter number of any simple component of $X$. An association of distinguished $2m^X$-vector-valued mock modular forms $H^X_g(\tau)=(H^X_{g,r}(\tau))$ 
to elements $g\in G^X$ is described and analyzed in \cite{UM,MUM,mumcor}. 

For $X=A_1^{24}$ we have $G^X\simeq M_{24}$ and $m^X=2$, and the functions $H^X_{g,1}(\tau)$ are precisely the mock modular forms assigned to elements $g\in M_{24}$ in the works \cite{MR2793423,Eguchi2010a,Gaberdiel2010a, Gaberdiel2010} mentioned above. Generalizing the $M_{24}$ moonshine initiated by Eguchi, Ooguri and Tachikawa, we have the following conjecture
 of Cheng, Duncan and Harvey (cf. \S2 of \cite{mumcor} or \S9.3 of \cite{DuncanGriffinOno}).
 
\begin{conjecture}[Umbral Moonshine Modules]
Let $X$ be a Niemeier root system $X$ and set $m:=m^X$.
There is a naturally defined bi-graded infinite-dimensional $G^X$-module 
\begin{gather}\label{eqn:um-KX}
\check{K}^X=\bigoplus_{r\in I^X}\bigoplus_{\substack{D\in\Z,\; D\leq 0,\\D=r^2\pmod{4m}}}\check{K}^X_{r,-D/4m}
\end{gather} 
such that the vector-valued mock modular form $H^X_g=(H^X_{g,r})$ is a McKay-Thompson series for $\check{K}^X$ related\footnote{In the statement of Conjecture 6.1 of \cite{MUM} the function $H^X_{g,r}$ in (\ref{eqn:um-HXgrKXrd}) is replaced with $3H^X_{g,r}$ in the case that $X=A_8^3$. This is now known to be an error, arising from a misspecification of some of the functions $H^X_{g}$ for $X=A_8^3$. Our treatment of the case $X=A_8^3$ in this work reflects the corrected specification of the corresponding $H^X_g$ which is described and discussed in detail in \cite{mumcor}.} to the graded trace of $g$ on $\check{K}^X$ by
\begin{gather}\label{eqn:um-HXgrKXrd}
	H^X_{g,r}(\tau)=-2q^{-1/4m}\delta_{r,1}
	+\sum_{\substack{D\in\Z,\; D\leq 0,\\D=r^2\pmod{4m}}}\tr(g|\check{K}^X_{r,-D/4m})q^{-D/4m}
\end{gather}
for $r\in I^X$.
\end{conjecture}

In (\ref{eqn:um-KX}) and (\ref{eqn:um-HXgrKXrd}) 
the set $I^X\subset\Z/2m\Z$ is defined in the following way. If $X$ has an A-type component then $I^X:=\{1,2,3,\ldots,m-1\}$. If $X$ has no A-type component but does have a D-type component then $m=2\mod 4$, and $I^X:=\{1,3,5,\ldots,m/2\}$. The remaining cases are $X=E_6^4$ and $X=E_8^3$. In the former of these, $I^X:=\{1,4,5\}$, and in the latter case $I^X:=\{1,7\}$.

\begin{remark}
The functions $H^X_g(\tau)$ are defined explicitly in \S\ref{sec:exp}. An alternative description in terms of Rademacher sums is given in \S\ref{sec:mts:rad}. 
\end{remark}

Here we prove the following theorem.

\begin{theorem}\label{MainThm}
The umbral moonshine modules exist.
\end{theorem}

\noindent
{\it Two remarks.}

\smallskip

\noindent 1)
Theorem~\ref{MainThm} for $X=A_1^{24}$ is the main result of Gannon's work
\cite{Gannon:2012ck}.
\smallskip

\noindent 
2) The vector-valued mock modular forms $H^X=(H^X_{g,r})$ have ``minimal"
 {\it principal parts}. This minimality is analogous to the fact that the
original McKay-Thompson series $T_g(\tau)$ for the Monster are hauptmoduln, and plays an important role in our proof.

\begin{example}  Many of Ramanujan's mock theta functions \cite{MR947735}
are components of the
vector-valued umbral McKay-Thompson series $H^X_g=(H^X_{g,r})$. 
For example, consider the root system $X=A_2^{12}$, whose umbral group is a double cover $2.M_{12}$ of the sporadic simple Mathieu group $M_{12}$.
In terms of Ramanujan's 3rd order mock theta functions
\begin{displaymath}
\begin{split}
f(q)&=1+\sum_{n=1}^{\infty}\frac{q^{n^2}}{(1+q)^2(1+q^2)^2\cdots (1+q^n)^2},\\
\phi(q)&=1+\sum_{n=1}^{\infty}\frac{q^{n^2}}{(1+q^2)(1+q^4)\cdots (1+q^{2n})},\\
\chi(q)&=1+\sum_{n=1}^{\infty}\frac{q^{n^2}}{(1-q+q^2)(1-q^2+q^4)\cdots (1-q^n+q^{2n})}\\
\omega(q)&=\sum_{n=0}^{\infty}\frac{q^{2n(n+1)}}{(1-q)^2(1-q^3)^2\cdots (1-q^{2n+1})^2},\\
\rho(q)&=\sum_{n=0}^{\infty}\frac{q^{2n(n+1)}}{(1+q+q^2)(1+q^3+q^6)\cdots (1+q^{2n+1}+q^{4n+2})},
\end{split}
\end{displaymath}
we have that
\begin{displaymath}
\begin{split}
H^{X}_{2B,1}(\tau)=H^X_{2C,1}(\tau)=H^X_{4C,1}(\tau)&=-2q^{-\frac{1}{12}}\cdot f(q^2),\\
H^X_{6C,1}(\tau)=H^X_{6D,1}(\tau)&=-2q^{-\frac{1}{12}}\cdot \chi(q^2),\\
H^X_{8C,1}(\tau)=H^X_{8D,1}(\tau)&=-2q^{-\frac{1}{12}}\cdot \phi(-q^2),\\
H^X_{2B,2}(\tau)=-H_{2C,2}^X(\tau)&=-4q^{\frac{2}{3}}\cdot \omega(-q),\\
H^X_{6C,2}(\tau)=-H^X_{6D,2}(\tau)&=2q^{\frac{2}{3}}\cdot \rho(-q).
\end{split}
\end{displaymath}
See \S5.4 of \cite{MUM} for more coincidences between umbral McKay-Thompson series and mock theta functions identified by Ramanujan almost a hundred years ago.
\end{example}

\smallskip
Our proof of Theorem~\ref{MainThm} involves the explicit determination of each $G^X$-module $\check K^X$ by computing the multiplicity of each
irreducible component for each homogeneous subspace. It guarantees the existence and
uniqueness of a $\check K^X$ which is compatible
with the representation theory of $G^X$ and the Fourier expansions of the vector-valued mock modular forms
$H^X_g(\tau)=(H^X_{g,r}(\tau))$. 

At first glance our methods do not appear to shed light on any deeper algebraic properties of the $\check K^X$, such as might correspond to the vertex operator algebra structure on $V^{\natural}$, or the monster Lie algebra introduced by Borcherds in \cite{borcherds_monstrous}. However, we do determine, and utilize, specific recursion relations for the coefficients of the umbral McKay-Thompson series which are analogous to the replicability properties of monstrous moonshine formulated by Conway and Norton in \S8 of \cite{MR554399} (cf. also \cite{MR1200252}). More specifically, we use recent work \cite{IRR} of Imamo\u{g}lu, Raum and Richter, as generalized \cite{Mertens} by Mertens, to obtain
such recursions. These results are based on the process of {\it holomorphic projection}.

\begin{theorem}\label{Replicable} 
For each $g\in G^X$ and $0<r<m$, the mock modular form $H^X_{g,r}(\tau)$ is replicable in the mock modular sense.
\end{theorem}

A key step in Borcherds' proof \cite{borcherds_monstrous} of the monstrous moonshine conjecture is the reformulation of replicability in Lie theoretic terms. We may speculate that the {\em mock modular replicability} utilized in this work will ultimately admit an analogous algebraic interpretation. Such a result remains an important goal for future work.

In the statement of Theorem~\ref{Replicable}, replicable means that there are explicit recursion relations for the coefficients of the vector-valued mock modular form in question. For example, 
we recall the recurrence formula for
Ramanujan's third order mock theta function $f(q)=\sum_{n=0}^{\infty}c_f(n)q^n$ that was obtained
recently by Imamo\u{g}lu, Raum and Richter \cite{IRR}.
If $n\in \mathbb{Q}$, then let
$$
\sigma_1(n):=\begin{cases} \sum_{d\mid n} d \ \ \ \ \ &{\text {\rm if}}\ n\in \Z,\\
0\ \ \ \ \ &{\text {\rm otherwise}},
\end{cases}
$$        
$$
\sgn^+(n):=\begin{cases} \sgn(n)\ \ \ \ \ &{\text {\rm if}}\ n\neq 0,\\
1 \ \ \ \ \ &{\text {\rm if}}\ n=0,
\end{cases}
$$
and then define
$$
d(N,\widetilde{N}, t, \widetilde{t}):= \sgn^{+}(N)\cdot \sgn^{+} (\tilde{N})\cdot\left(
|N+t|-|\widetilde{N}+\widetilde{t}|\right).
$$
Then for positive integers $n$, we have that
\begin{displaymath}
\begin{split}
           \sum_{\substack{m\in \Z\\ 3m^2+m\leq 2n}} \left (m+\frac{1}{6}\right) &c_f\left(
           n-\frac{3}{2}m^2-\frac{1}{2}m\right)\\
           &=\frac{4}{3}\sigma(n)-\frac{16}{3}\sigma\left(\frac{n}{2}\right)-2
           \sum_{\substack{a,b\in \Z\\ 2n=ab}}d\left(N,\widetilde{N},\frac{1}{6},\frac{1}{6}\right),
\end{split}
\end{displaymath}
where    $N:=\frac{1}{6}(-3a+b-1)$ and $\widetilde{N}:=\frac{1}{6}(3a+b-1)$, and the sum is over integers $a, b$ for which $N, \widetilde{N}\in \Z$.
This is easily seen to be a recurrence relation for the coefficients $c_f(n)$. The replicability formulas for all of the
$H_{g,r}^{X}(\tau)$ are similar (although some of these relations are slightly more complicated and involve the coefficients of weight 2 cusp forms).

It is important to emphasize that, despite the progress 
which is represented by our main results, Theorems \ref{MainThm} and \ref{Replicable}, the following important question remains open in general.
\begin{question}
Is there a ``natural"
construction of $\check K^X$? Is $\check K^X$  equipped with a deeper algebra structure
as in the case of the monster module $\Vnat$ of Frenkel, Lepowsky and Meurman?
\end{question}
We remark that this question has been answered positively, recently, in one special case: A vertex operator algebra structure underlying the umbral moonshine module $\check K^X$ for $X=E_8^3$ has been described explicitly in \cite{mod3e8}. See also \cite{umvan2,umvan4}, where the problem of constructing algebraic structures that illuminate the umbral moonshine observations is addressed from a different point of view.

The proof of Theorem~\ref{MainThm} is not difficult. It is essentially a collection of tedious calculations.
We use the theory of mock modular forms and the character table for each $G^X$ (cf. \S\ref{sec:gps:chars}) to
solve for the multiplicities of the irreducible $G^X$-module constituents of each homogeneous subspace in the alleged $G^X$-module
$\check K^X$. To prove Theorem \ref{MainThm} it suffices to prove that these multiplicities are non-negative integers. To prove Theorem~\ref{Replicable}
we apply recent work \cite{Mertens} of Mertens  on the holomorphic projection of weight $\frac{1}{2}$ mock modular forms, which
generalizes earlier work  \cite{IRR} of  Imamo\u{g}lu, Raum and Richter.

In \S\ref{HMF} we recall the facts about mock modular forms that we require, and we prove Theorem~\ref{Replicable}.
We prove Theorem~\ref{MainThm} in \S\ref{Proofs}. The appendices furnish all the data that our method requires. In particular, the umbral groups $G^X$ are described in detail in \S\ref{sec:gps}, and explicit definitions for the mock modular forms $H^X_g(\tau)$ are given in \S\ref{sec:mts}.

\section{Harmonic Maass forms and Mock modular forms}\label{HMF}

Here we recall some very basic facts about harmonic Maass forms as developed by Bruinier and Funke
\cite{BruFun_TwoGmtThtLfts} (see also \cite{Ono_unearthing}).

We begin by briefly recalling the definition of a \emph{harmonic Maass form} of weight $k\in \frac{1}{2}\Z$ and multiplier $\nu$ (a generalization of the notion of a Nebentypus). If $\tau=x+iy$ with $x$ and $y$ real, we define the weight $k$ hyperbolic Laplacian by
\begin{equation}
\Delta_k := -y^2\left( \frac{\partial^2}{\partial x^2} +
\frac{\partial^2}{\partial y^2}\right) + iky\left(
\frac{\partial}{\partial x}+i \frac{\partial}{\partial y}\right).
\end{equation}

Suppose $\Gamma$ is a subgroup of finite index in $\SL_2(\Z)$ and $k\in \frac{1}{2}\Z$. Then a function $F(\tau)$ which is real-analytic on the upper half of the complex plane is a {\it harmonic Maass form}  of weight $k$ on $\Gamma$ with multiplier $\nu$ if: 

\begin{enumerate}
\item[(a)]  The function $F(\tau)$ satisfies the weight $k$ modular transformation, 
$$F(\tau)|_k\gamma=\nu(\gamma)F(\tau)$$ for every matrix
$\gamma=\begin{pmatrix}a&b\\c&d\end{pmatrix}\in \Gamma,$ where $F(\tau)|_k\gamma:=F(\gamma \tau)(c\tau+d)^{-k},$ and if $k\in \Z+\frac12,$ the square root is taken to be the principal branch.
\item[(b)] We have that $\Delta_kF(\tau)=0,$
\item[(c)] There is a polynomial $P_{F}(q^{-1})$ and a constant $c>0$ such that $F(\tau)-P_{F}(e^{-2\pi i \tau})=O(e^{-cy})$ as $\tau\to i\infty$.
Analogous conditions are required at each cusp of $\Gamma$.
\end{enumerate}

We denote the $\CC$-vector space of harmonic Maass forms of a given weight $k$, group $\Gamma$ and multiplier $\nu$ by $H_k(\Gamma,\nu).$ If no multiplier is specified, we will take 
\[\nu_0(\gamma):=\left(\left(\frac{c}{d}\right)\sqrt{\left(\frac{-1}{d}\right)}^{-1}\right)^{2k},\]
where $\left(\frac{*}{d}\right)$ is the Kronecker symbol.

\subsection{Main properties}

The  Fourier expansion of a harmonic Maass form $F$ (see Proposition 3.2 of \cite{BruFun_TwoGmtThtLfts}) splits into two components. As before, we let $q:=e^{2\pi i \tau}$.
\begin{lemma}\label{HMFparts}
If $F(\tau)$ is a harmonic Maass form of weight $2-k$ for $\Gamma$ where $\frac{3}{2}\leq k\in \frac12\Z$, then

\begin{displaymath}
F(\tau)=F^+(\tau)+F^{-}(\tau),
\end{displaymath}
where $F^+$ is the holomorphic part of $F$, given by
$$
F^+(\tau)=\sum_{n\gg -\infty} c_{F}^+(n) q^{n}
$$
where the sum admits only finitely many non-zero terms with $n<0$, and $F^{-}$ is the nonholomorphic part, given by
$$
F^{-}(\tau) = \sum_{n<0} c_{F}^-(n)
\Gamma(k-1,4\pi y|n|) q^{n}.
$$
Here $\Gamma(s,z)$ is the upper incomplete gamma function.
\end{lemma}

The holomorphic part of a harmonic Maass form is called a \emph{mock modular form.} We denote the space of harmonic Maass forms of weight $2-k$ for $\Gamma$ and multiplier $\nu$ by $H_{k}(\Gamma,\nu).$ Similarly, we denote the corresponding subspace of holomorphic modular forms by $M_k(\Gamma,\nu),$ and the space of cusp forms by $ S_k(\Gamma,{\nu})$.
 The differential operator $\xi_{w}:=2iy^w\overline{\frac{\partial}{\partial \overline \tau}}$ (see  \cite{BruFun_TwoGmtThtLfts})  defines a surjective map 
$$\xi_{2-k}:H_{2-k}(\Gamma,\nu)\to S_k(\Gamma,\overline{\nu})$$
onto the space of weight $k$ cusp forms for the same group but conjugate multiplier.
 The \emph{shadow} of a Maass form $f(\tau)\in H_{2-k}(\Gamma,\nu)$ is the cusp form $g(\tau)\in S_k(\Gamma,\overline{\nu})$ (defined, for now, only up to scale) such that $\xi_{2-k}f(\tau)=\frac{g}{||g||}$, where $||\bullet||$ denotes the usual Petersson norm.

\subsection{Holomorphic projection of weight $\frac{1}{2}$ mock modular forms}

As noted above, the modular transformations of a weight $\frac{1}{2}$ harmonic Maass form may be simplified by multiplying by its shadow to obtain a weight $2$ nonholomorphic modular form. One can use the theory of holomorphic projections to obtain explicit identities relating these nonholomorphic modular forms to classical quasimodular forms.
In this way, we may essentially reduce many questions about the coefficients of weight $\frac{1}{2}$ mock modular forms to questions about weight 2 holomorphic modular forms. The following theorem is a special case of a more general theorem due to Mertens (cf. Theorem 6.3 of \cite{Mertens}). See also \cite{IRR}.

\begin{theorem}[Mertens]\label{Mertens}
Suppose $g(\tau)$ and $h(\tau)$ are both theta functions of weight $\frac32$ contained in $S_{\frac 32}(\Gamma,\nu_g)$ and $S_{\frac 32}(\Gamma,\nu_h)$ respectively, with Fourier expansions
\begin{displaymath}
\begin{split}
g(\tau):&=\sum_{i=1}^{s}\sum_{n\in\Z}n\chi_i(n)q^{n^2},\\
h(\tau)&:=\sum_{j=1}^{t}\sum_{n\in\Z}n\psi_j(n)q^{n^2},
\end{split}
\end{displaymath}
where each $\chi_i$ and $\psi_i$ is a Dirichlet character. 
Moreover, suppose $h(\tau)$ is the shadow of a weight $\frac{1}{2}$ harmonic Maass form $f(\tau) \in H_{\frac 12}(\Gamma,\overline {\nu_h}).$ 
Define the function
\[D^{f,g}(\tau):=2 \ \sum_{r=1}^{\infty}\sum_{\chi_i,\psi_j}\sum_{\substack{m,n\in \Z^+\\m^2-n^2=r}}\chi_i(m)\overline{\psi_j(n)}(m-n)q^r.\]
If $f(\tau)g(\tau)$ has no singularity at any cusp, then $f^{+}(\tau)g(\tau)+D^{f,g}(\tau)$ is a weight $2$ quasimodular form.
In other words, it lies in the space $\CC E_2(\tau)\oplus M_2(\Gamma,\nu_g\overline{\nu_h})$, where $E_2(\tau)$ is the quasimodular Eisenstein series $E_2(\tau):=1-24\displaystyle\sum_{n\geq1}\frac{nq^n}{1-q^n}.$
\end{theorem}

\medskip
\noindent
{\it Two Remarks.}

\noindent
1) These identities give recurrence relations for the weight $\frac{1}{2}$ mock modular form $f^+$ in terms of the
weight 2 quasimodular form which equals $f^+(\tau)g(\tau)+D^{f,g}(\tau)$. The example after Theorem~\ref{Replicable}
for Ramanujan's third order mock theta function $f$ is an explicit example of such a relation.

\smallskip
\noindent
2) Theorem~\ref{Mertens} extends to vector-valued mock modular forms in a natural way.

\medskip

\begin{proof}[Proof of Theorem~\ref{Replicable}]
Fix a Niemeier lattice and its root system $X$, and let $M=m^X$ denote its Coxeter number. Each $H_{g,r}^{X}(\tau)$ is the holomorphic part of a weight $\frac{1}{2}$ harmonic Maass form $\widehat{H}_{g,r}^{X}(\tau).$ To simplify the exposition in the following section, we will emphasize the case that the root system $X$ is of pure A-type. 
If the root system $X$ is of pure A-type, the shadow function $S_{g,r}^{X}(\tau)$ is given by $\hat \chi_{g,r}^{X_A} S_{M,r}(\tau)$ (see \S\ref{sec:mts:shd}), where  
$$
S_{M,r}(\tau)=\sum_{\substack{n\in \Z\\n\equiv r \pmod{2M}}}n ~q^{\frac{n^2}{4M}},
$$
and $\hat\chi_{g,r}^{X_A}$= $\chi_{g}^{X_A}$ or $\bar\chi_{g}^{X_A}$ depending on the parity of $r$ is the twisted Euler character given in the appropriate table in \S\ref{sec:chars:eul}, a 
character of $G^X.$ (If $X$ is not of pure A-type, then the shadow function $S_{g,r}^{X}(\tau)$ is a linear combination of similar functions as described in \S\ref{sec:mts:shd}.)

Given $X$ and $g$, the symbol $n_g|h_g$ given in the corresponding table in \S\ref{sec:chars:eul} defines the modularity for the vector-valued function $(\widehat{H}_{g,r}^{X}(\tau))$. In particular, if the shadow $(S_{g,r}^{X}(\tau))$ is nonzero, and if for $\gamma\in \Gamma_0(n_g)$ we have that 
$$(S_{g,r}^{X}(\tau))|_{3/2}\gamma=\sigma_{g,\gamma}(S_{g,r}^{X}(\tau)),$$
then 
$$(\widehat{H}_{g,r}^{X}(\tau))|_{1/2}\gamma=\overline{\sigma_{g,\gamma}}(\widehat{H}_{g,r}^{X}(\tau)).$$
Here, for $\gamma\in \Gamma_0(n_g)$, we have $\sigma_{g,\gamma}=\nu_{g}(\gamma)\sigma_{e,\gamma}$ where $\nu_{g}(\gamma)$ is a multiplier which is trivial on $\Gamma_0(n_gh_g)$. 
This identity holds even in the case that the shadow $S_{g,r}^X$ vanishes. 

The vector-valued function $(H_{g,r}^{X}(\tau))$ has poles only at the infinite cusp of $\Gamma_0(n_g)$, and only at the component $H_{g,r}^{X}(\tau)$ where $r=1$ if $X$ has pure A-type, or at components where $r^2\equiv 1 \pmod{4M}$ otherwise. These poles may only have order $\frac{1}{4M}.$ This implies that the function $(\widehat{H}_{g,r}^{X}(\tau)S_{g,r}^{X}(\tau))$ has no pole at any cusp, and is therefore a candidate for an application of Theorem \ref{Mertens}.

The modular transformation of $S_{M,r}(\tau)$ implies that 
$$(\sigma_{e,S})^2=(\sigma_{e,T})^{4M}=\mathbf{I}$$
 where $S=\begin{pmatrix}0&-1\\1&0\end{pmatrix}$, $T=\begin{pmatrix}1&1\\0&1\end{pmatrix}$, and $\mathbf{I}$ is the identity matrix. Therefore $S_{M,r}^{X}(\tau)$, viewed as a scalar-valued modular function, is modular on $\Gamma(4M),$ and so $(\widehat{H}_{g,r}^{X}(\tau)S_{g,r}^{X}(\tau))$ is a weight $2$ nonholomorphic scalar-valued modular form for the group $\Gamma(4M)\cap \Gamma_0(n_g)$ with trivial multiplier.

Applying Theorem~\ref{Mertens}, we obtain a function $F_{g,r}^{X}(\tau)$---call it the holomorphic projection of $\widehat{H_{g,r}^{X}}(\tau) S_{e,r}^{X}(\tau)$---which is a weight 2 quasimodular form on $\Gamma(4M)\cap \Gamma_0(n_g).$ In the case that $S_{g,r}^{X}(\tau)$ is zero, we substitute $S_{e,r}^{X}(\tau)$ in its place to obtain a function $\widetilde F_{g,r}^{X}(\tau)=H_{g,r}^{X}(\tau) S_{e,r}^{X}(\tau)$ which is a weight $2$ holomorphic scalar-valued modular form for the group $\Gamma(4M)\cap \Gamma_0(n_g)$ with multiplier $\nu_{g}$ (alternatively, modular for the group $\Gamma(4M)\cap \Gamma_0(n_gh_g)$ with trivial multiplier).

The function $F_{g,r}^{X}(\tau)$
may be determined explicitly as the sum of Eisenstein series and cusp forms on $\Gamma(4M)\cap \Gamma_0(n_gh_g)$ using the standard arguments from the theory of holomorphic modular forms (i.e. the ``first few" coefficients determine such a form).
Therefore, we have the identity
\begin{equation}\label{defF}
F_{g,r}^X(\tau)=H_{g,r}^X(\tau)\cdot S_{g,r}^X(\tau)+D_{g,r}^X(\tau),
\end{equation}
where 
the function $D_{g,r}^X(\tau)$ is the correction term arising in Theorem \ref{Mertens}. If $X$ has pure A-type, then

\begin{equation}
D_{g,r}^X(\tau)=(\hat\chi_{g,r}^{X_A})^2\sum_{N=1}^{\infty}\sum_{\substack{m,n\in \Z_+\\m^2-n^2=N}}\phi_r(m)\phi_r(n)(m-n)q^\frac{N}{4M},
\end{equation}
where $$\phi_r(\ell)=\begin{cases}\pm1 &\text{ if } \ell\equiv \pm r \pmod {2M}\\0&\text{ otherwise.}\end{cases}$$

Suppose $H_{g,r}^{X}( \tau)=\displaystyle\sum_{n=0}^{\infty}A^X_{g,r}(n) q^{n-\frac{D}{4M}}$ where $0<D<4M$ and $D\equiv r^2 \pmod {4M},$ and $F_{g,r}^{X}(\tau)=\displaystyle\sum_{N=0}^{\infty}B_{g,r}^X(n)q^n.$ Then by Theorem \ref{Mertens}, we find that

\begin{equation}\label{Recursion}
\begin{split}
B^X_{g,r}(N)=&\hat\chi_{g,r}^{X_A}\sum_{\substack{m\in \Z\\m\equiv r \pmod{2M}}}
m\cdot A^X_{g,r}\left(N+\frac{D-m^2}{4M}\right)+(\hat\chi_{g,r}^{X_A})^2\sum_{\substack{m,n\in \Z^+\\m^2-n^2=N}}\phi_r(m)\phi_r(n)(m-n).
\end{split}
\end{equation}
\noindent

The function $F_{g,r}^{X}(\tau)$ may be found in the following manner. Using the explicit prescriptions for ${H_{g,r}^{X}}( \tau)$ given in \S\ref{sec:exp} and (\ref{defF}) above, we may calculate the first several coefficients of each component. The Eisenstein component is determined by the constant terms at cusps. Since $D_{g,r}^X(\tau)$ (and the corresponding correction terms at other cusps) has no constant term, these are the same as the constant terms of $\widehat{H_{g,r}^{X}}( \tau)S_{g,r}^{X}(\tau),$ which are determined by the poles of $\widehat{H_{g,r}^{X}}$. Call this Eisenstein component $E_{g,r}^{X}(\tau).$ The cuspidal component can be found by matching the initial coefficients of $F_{g,r}^{X}(\tau)-E_{g,r}^{X}(\tau).$

Once the coefficients $B^X_{g,r}(n)$ are known, equation (\ref{Recursion}) provides a recursion relation which may be used to calculate the coefficients of $H^{X}_{g,r}(\tau).$ If the shadows $S_{g,r}^{X}(\tau)$ are zero, then we may apply a similar procedure in order to determine $\widetilde F_{g,r}^{X}(\tau)$. For example, suppose $\widetilde F_{g,r}^{X}(\tau)=\displaystyle\sum_{N=0}^{\infty}\widetilde  B_{g,r}^X(n)q^n,$ and $X$ has pure A-type. Then we find that the coefficients $\widetilde B^X_{g,r}(N)$ satisfy

\begin{equation}\label{Recursion2}
\begin{split}
\widetilde B^X_{g,r}(N)=\hat\chi_{g,r}^{X_A}&\sum_{\substack{m\in \Z\\m\equiv r \pmod{2M}}}
m\cdot A^X_{g,r}\left(N+\frac{D-m^2}{4M}\right)
\end{split}
\end{equation}
Proceeding in this way we obtain the claimed results.
\end{proof}

\section{Proof of Theorem 1.1}\label{Proofs}
Here we prove Theorem~\ref{MainThm}. The idea is as follows. For each Niemeier root system $X$ we begin with the vector-valued mock modular
forms $(H^X_g(\tau))$ for $g\in G^X$. 
We use their $q$-expansions to solve for the $q$-series whose coefficients are the alleged multiplicities of the irreducible
components of the alleged infinite-dimensional $G^X$-module
$$\check K^X=\bigoplus_{r\pmod{2m}}\bigoplus_{\substack{D\in\Z,\;D\leq 0,\\D=r^2\pmod{4m}}}\check K^X_{r,-D/4m}.
$$
These $q$-series turn out to be mock modular forms. The proof requires that we establish that these mock modular forms have non-negative integer coefficients.

\begin{proof}[Proof of Theorem~\ref{MainThm}]
As in the previous section, we fix a root system $X$ and set $M:=m^X$, and we emphasize the case when $X$ is of pure A-type.

The umbral moonshine conjecture asserts that 
\begin{equation}\label{Hg_as_mult}
H_{g,r}^{X}(\tau) =\sum_{n=0}^\infty\sum_{\chi}m_{\chi,r}^{X}(n)\chi(g)q^{n-\frac{r^2}{4M}}
\end{equation}
where the second sum is over the irreducible characters of $G^X$. Here we have rewritten the traces of the graded components $\check{K}^X_{r,n-{r^2}/{4M}}$ in \ref{eqn:um-HXgrKXrd} in terms of the values of the irreducible characters of $G^X$, where the $m_{\chi,r}^{X}(n)$ are the corresponding multiplicities. Naturally, if such a $\check K^X$ exists, these multiplicities must be non-negative integers for $n>0$. Similarly, if the mock modular forms $H_{g,r}^{X}(\tau)$ can be expressed as in \ref{Hg_as_mult} with $m_{\chi,r}^{X}(n)$ non-negative integers, then we may construct the umbral moonshine module $\check K^X$ explicitly with $\check K^X_{r,n-r^2/4M}$ defined as the direct sum of irreducible components with the given multiplicities $m_{\chi,r}^{X}(n).$ 

Let
\begin{equation}\label{H_chi Def}
H_{\chi,r}^{X}(\tau):=\frac{1}{|G^X|}\sum_{g}\overline{\chi(g)}H_{g,r}^{X}(\tau).
\end{equation}
It turns out that the coefficients of $H_{\chi,r}^{X}(\tau)$ are precisely the multiplicities $m_{\chi,r}^{X}(n)$ required so that \ref{Hg_as_mult} holds: if
\begin{equation}\label{Hchi_as_mult}
H_{\chi,r}^{X}(\tau)=\sum_{n=0}^\infty m_{\chi,r}^{X}(n)q^{n-\frac{r^2}{4M}},
\end{equation} 
then 
\begin{equation*}
H_{g,r}^{X}(\tau) =\sum_{n=0}^\infty\sum_{\chi}m_{\chi,r}^{X}(n)\chi(g)q^{n-\frac{r^2}{4M}}.
\end{equation*} 
Thus the umbral moonshine conjecture is true if and only if the Fourier coefficients of $H_{\chi,r}^{X}(\tau)$ are non-negative integers. 

To see this fact, we recall the orthogonality of characters. For irreducible characters $\chi_i$ and $\chi_j,$
\begin{equation}\label{Orthogonality1}
\frac{1}{|G^X|}\sum_{g\in G^X} \overline{\chi_i(g)}\chi_j(g)=\begin{cases}1 & $ if $ \chi_i=\chi_j,\\
0 & $ otherwise.$
\end{cases}
\end{equation}
We also have the relation for $g$ and $h\in G^X$,
\begin{equation}\label{Orthogonality2}
\sum_{\chi} \overline{\chi_i(g)}\chi_i(h)=\begin{cases}|C_{G^X}(g)| & $ if $ g $ and $ h $ are conjugate,$\\
0 & $ otherwise.$
\end{cases}
\end{equation}
Here $|C_{G^X}(g)|$ is the order of the centralizer of $g$ in $G^{X}$. Since the order of the centralizer times the order of the conjugacy class of an element is the order of the group,  (\ref{H_chi Def}) and (\ref{Orthogonality2}) together imply the relation
$$
H_{g,r}^{X}(\tau)=\sum_{\chi}\chi(g)H_{\chi,r}^{X}(\tau),
$$
which in turn implies \ref{Hchi_as_mult}.

We have reduced the theorem to proving that the coefficients of certain weight $\frac12$ mock modular forms
are all non-negative integers.
For holomorphic modular forms we may answer questions of this type by making use of Sturm's theorem \cite{Sturm} (see also Theorem~2.58 of \cite{CBMS}). This theorem provides a bound $B$ associated to a space of modular forms such that if the first $B$ coefficients of a modular form $f(\tau)$ are integral, then all of the coefficients of $f(\tau)$ are integral. This bound reduces many questions about the Fourier coefficients of modular forms to finite calculations. 

Sturm's theorem relies on the finite dimensionality of certain spaces of modular forms, and so it can not be applied directly to spaces of mock modular forms. However, by making use of holomorphic projection we can adapt Sturm's theorem to this setting. 

Let $\widehat{H^{X}_{\chi,r}}(\tau)$ be defined as above.
Recall that the transformation matrix for the vector-valued function $\widehat{H^{X}_{g,r}}(\tau))$ is $\overline{\sigma_{g,\gamma}},$ the conjugate of the transformation matrix for $(S_{e,r}^{X}(\tau))$ when $\gamma\in \Gamma_0(n_gh_g),$ and $\sigma_{g,\gamma}$ is the identity for $\gamma\in \Gamma(4M).$ Therefore if
\[
N_\chi^X:=\lcm\{n_gh_g \mid  g \in G,\chi(g)\neq0\},
\]
then the scalar-valued functions $\widehat{H^{X}_{\chi,r}}(\tau)$ are modular on $\Gamma(4M)\cap\Gamma_0(N_\chi^X).$

Let $$A_{\chi,r}(\tau):=H^{X}_{\chi,r}(\tau)S_{e,1}^{X}(\tau),$$ 
and let $\tilde A_{\chi,r}(\tau)$ be the holomorphic projection of $A_{\chi,r}(\tau).$
Suppose that $H^{X}_{\chi,r}(\tau)$ has integral coefficients up to some bound $B$. 
Formulas for the shadow functions (cf. \S\ref{sec:mts:shd}) show that the leading coefficient of $S_{e,1}^{X}(\tau)$ is $1$ and has integral coefficients. This implies that the function
$$A_{\chi,r}(\tau):=H^{X}_{\chi,r}(\tau)S_{e,1}^{X}(\tau)$$
also has integral coefficients up to the bound $B$. 
The shadow of $H^{X}_{\chi,r}(\tau)$ is given by
\[
{S_{\chi,r}^{X}(\tau)}:=\frac{1}{|G^X|}\sum_{g}\overline{\chi(g)}S_{g,r}^{X}(\tau).
\]
If $X$ is pure A-type, then $S_{g,r}^{X}(\tau)=\chi_{g,r}^{X_A}S_{M,r}(\tau)=(\chi'(g)+\chi''(g))S_{M,r}(\tau)$ for some irreducible characters $\chi'$ and $\chi''$, according to \S\ref{sec:chars:eul} and \S\ref{sec:mts:shd}. Therefore,
\begin{equation*}
S_{\chi,r}^{X}(\tau)=\begin{cases}S_{M,r}(\tau) &$ if $ \chi=\chi'$ or $\chi'',\\ 0 &$ otherwise.$
\end{cases}\end{equation*}
When $X$ is not of pure A-type the shadow is some sum of such functions, but in every case has integer coefficients, and so, applying Theorem \ref{Mertens} to $A_{\chi,r}(\tau),$ we find that $\tilde A_{\chi,r}(\tau)$ also has integer coefficients up to the bound $B$. In particular, since $\tilde A_{\chi,r}(\tau)$ is modular on $\Gamma(4M)\cap\Gamma_0(N_\chi^X)$, then if $B$ is at least the Sturm bound for this group we have that every coefficient of $\tilde A_{\chi,r}(\tau)$ is integral. Since the leading coefficient of $S_{e,1}^{X}(\tau)$ is $1,$ we may reverse this argument and obtain that every coefficient of ${H^{X}_{\chi,r}}(\tau)$ is integral. Therefore, in order to check that $H^{X}_{\chi,r}(\tau)$ has only integer coefficients, it suffices to check up to the Sturm bound for $\Gamma(4M)\cap\Gamma_0(N_\chi).$
These calculations were carried out using the $\tt{sage}$ mathematical software \cite{sage}.

The calculations and argument given above shows that the multiplicities $m_{\chi,r}^X(n)$ are all integers.
To complete the proof, it suffices to check that they are also are non-negative.
The proof of this claim follows easily by modifying step-by-step the argument in
Gannon's proof of non-negativity in the $M_{24}$ case \cite{Gannon:2012ck} (i.e. $X=A_1^{24}$).
Here we describe how this is done.

Expressions for the alleged McKay-Thompson series
 $H^{X}_{g,r}(\tau)$ in terms of {Rademacher sums} and unary theta functions are given
 in \S\ref{sec:mts:rad}. 
 Exact formulas are known for all the coefficients of Rademacher sums because they are defined by
 averaging the special function
 $\rad^{[\alpha]}_{1/2}(\g,\tau)$ (see (\ref{eqn:mts:rad-radreg})) over cosets of a specific modular group modulo $\Gamma_{\infty}$, the subgroup of translations.
 Therefore, Rademacher sums are standard Maass-Poincar\'e series, and as a result we have formulas for each of their coefficients as
convergent infinite sums of Kloosterman-type sums weighted by values of the  $I_{1/2}$ modified Bessel function. (For example,
see \cite{BringmannOno} or \cite{whalen} for the general theory, and \cite{Cheng2011} for the specific case that $X=A_1^{24}$.) More importantly, this means also that
the generating function for the multiplicities $m_{\chi,r}^X(n)$ is a weight $\frac12$ harmonic Maass form, which in turn means
that exact formulas (modulo the unary theta functions) are also available in similar terms.
For positive integers $n$, this then means that (cf. Theorem 1.1 of \cite{BringmannOno})
\begin{equation}\label{Multiplicities}
m_{\chi,r}^X(n)=\sum_{\rho} \sum_{m<0} \frac{a_{\rho}^X(m)}{n^{\frac{1}{4}}}\sum_{c=1}^{\infty} \frac{K^X_{\rho}(m,n,c)}{c}\cdot \mathbb{I}^X\left (\frac{4\pi \sqrt{|nm|}}{c}\right),
\end{equation}
where the sums are over the cusps $\rho$ of the group $\Gamma_0(N_g^X)$, and finitely many explicit negative rational numbers $m$.
The constants $a_{\rho}^X(m)$ are essentially the coefficients which describe the generating function in terms of Maass-Poincar\'e series.
Here $\mathbb{I}$ is a suitable normalization and change of variable for the standard $I_{1/2}$ modified Bessel-function.

The Kloosterman-type sums $K_{\rho}^X(m,n,c)$ are well known to be related to Sali\'e-type sums
(for example see Proposition 5 of \cite{Kohnen}).
These Sali\'e-type sums are of the form
$$
S^X_{\rho}(m,n,c)=\sum_{\substack{x\pmod c\\ x^2\equiv -D(m,n)\pmod{c}}} \epsilon_{\rho}^X(m,n)\cdot e\left (\frac{\beta^X x}{c}\right),
$$
where $\epsilon_{\rho}^X(m,n)$ is a root of unity, $-D(m,n)$ is a discriminant of a positive definite binary quadratic form, and
$\beta^X$ is a nonzero positive rational number.

These Sali\'e sums may then be estimated using the equidistribution of CM points with discriminant $-D(m,n)$.
This process was first introduced by Hooley \cite{Hooley},
 and it was first applied to the coefficients of weight $\frac{1}{2}$ mock modular forms by
 Bringmann and Ono\cite{BringmannOno2006}.
 Gannon explains how to make effective the estimates for sums of this shape in \S4 of \cite{Gannon:2012ck}, thereby reducing the proof of the
  $M_{24}$ case of umbral moonshine to a finite calculation. In particular, in equations (4.6-4.10) of \cite{Gannon:2012ck} Gannon shows how to bound coefficients of the form (\ref{Multiplicities}) in terms of the Selberg--Kloosterman zeta function, which is bounded in turn in his proof of Theorem 3 of \cite{Gannon:2012ck}.
 We 
 follow Gannon's proof  \emph{mutatis mutandis}. We find, for most root systems, that the coefficients of each multiplicity generating function are positive beyond the 390th coefficient. In the worst case, for the root system $X=E_8^3$, we find that the coefficients are positive beyond the 1780th coefficient. Moreover, the coefficients exhibit subexponential growth.
 A finite computer calculation in $\tt{sage}$ has verified the non-negativity of the finitely many remaining coefficients.
\end{proof}

\begin{remark}
It turns out that the estimates required for proving non-negativity are the worst for the $X=E_8^3$ case which required $1780$ coefficients.
  \end{remark}

\newpage

\appendix
\section{The Umbral Groups}\label{sec:gps}

In this section we present the facts about the umbral groups that we have used in establishing the main results of this paper. We recall (from \cite{MUM}) their construction in terms of Niemeier root systems in \S\ref{sec:gps:cnst}, and we reproduce their character tables (appearing also in \cite{MUM}) in \S\ref{sec:gps:chars}. Note that we use the abbreviations $a_n:=\sqrt{-n}$ and $b_n:=(-1+\sqrt{-n})/2$ in the tables of \S\ref{sec:gps:chars}.

The root system description of the umbral groups (cf. \S\ref{sec:gps:cnst}) gives rise to certain characters called {\em twisted Euler characters} which we recall (from \cite{MUM}) in \S\ref{sec:chars:eul}. The data appearing in \S\ref{sec:chars:eul} plays an important role in \S\ref{sec:mts:shd}, where we use it to describe the shadows $S^X_g$ of the umbral McKay-Thompson series $H^X_g$ explicitly.

\subsection{Construction}\label{sec:gps:cnst}

As mentioned in \S\ref{sec:intro}, there are exactly $24$ self-dual even positive-definite lattices of rank $24$ up to isomorphism, according to the classification of Niemeier \cite{Niemeier} (cf. also \cite{MR666350,MR558941}). Such a lattice $L$ is determined up to isomorphism by its {\em root system} $L_2:=\{\alpha\in L\mid \langle\alpha,\alpha\rangle=2\}$. The unique example without roots is the Leech lattice. We refer to the remaining $23$ as the {\em Niemeier lattices}, and we call a root system $X$ a {\em Niemeier root system} if it occurs as the root system of a Niemeier lattice.

The simple components of Niemeier root systems are root systems of ADE type, and it turns out that the simple components of a Niemeier root system $X$ all have the same Coxeter number. Define $m^X$ to be the Coxeter number of any simple component of $X$, and call this the {\em Coxeter number} of $X$.

For $X$ a Niemeier root system write $N^X$ for the corresponding Niemeier lattice. The {\em umbral group} attached to $X$ is defined by setting
\begin{gather}	
G^X:=\Aut(N^X)/W^X
\end{gather}
where $W^X$ is the normal subgroup of $\Aut(N^X)$ generated by reflections in root vectors. 

Observe that $G^X$ acts as permutations on the simple components of $X$. In general this action is not faithful, so define $\overline{G}^X$ to be the quotient of $G^X$ by its kernel. It turns out that the level of the mock modular form $H^X_g$ attached to $g\in G^X$ is given by the order, denoted $n_g$, of the image of $g$ in $\bar{G}^X$. (Cf. \S\ref{sec:chars:eul} for the values $n_g$.)

The Niemeier root systems and their corresponding umbral groups are described in Table \ref{tab:mugs}. The root systems are given in terms of their simple components of ADE type. Here $D_{10}E_7^2$, for example, means the direct sum of one copy of the $D_{10}$ root system and two copies of the $E_7$ root system. The symbol $\ell$ is called the {\em lambency} of $X$, and the Coxeter number $m^X$ appears as the first summand of $\ell$.

\begin{table}[ht]
\captionsetup{font=small}
\begin{center}
\caption{The Umbral Groups}\label{tab:mugs}
\medskip

\begin{tabular}{ccccccccccc}
\multicolumn{1}{c|}{$X$}&$A_1^{24}$&$A_2^{12}$&$A_3^8$&$A_4^6$&$A_5^4D_4$&$A_6^4$&$A_7^2D_5^2$\\
	\cmidrule{1-8}
\multicolumn{1}{c|}{$\ll$}&	2&	3&	4&	5&	6&	7&	8\\
	\cmidrule{1-8}
\multicolumn{1}{c|}{$G^{X}$}&			$M_{24}$&	$2.M_{12}$&	$2.\AGL_3(2)$&	$\GL_2(5)/2$&	$\GL_2(3)$&	$\SL_2(3)$&$\Dih_4$\\
\multicolumn{1}{c|}{$\bar{G}^{X}$}&		$M_{24}$&	$M_{12}$&	$\AGL_3(2)$&		$\PGL_2(5)$&	$\PGL_2(3)$&	$\PSL_2(3)$&$2^2$\\
 \\
\multicolumn{1}{c|}{$X$}&$A_8^3$&$A_9^2D_6$&$A_{11}D_7E_6$&$A_{12}^2$&$A_{15}D_9$&$A_{17}E_7$&$A_{24}$\\
	\cmidrule{1-8}
\multicolumn{1}{c|}{$\ll$}&	9&	10& 12&	13&	16&	18&	25\\
	\cmidrule{1-8}
\multicolumn{1}{c|}{$G^{X}$}&$\Dih_6$&$4$&			$2$&	$4$& $2$&$2$&$2$\\
\multicolumn{1}{c|}{$\bar{G}^{X}$}&$\Sym_3$&$2$&		$1$&$2$& $1$&$1$&$1$\\
\\
\multicolumn{1}{c|}{$X$}&$D_4^{6}$&$D_6^{4}$&$D_8^3$&$D_{10}E_7^2$&$D_{12}^2$&$D_{16}E_8$&$D_{24}$\\
	\cmidrule{1-8}
\multicolumn{1}{c|}{$\ll$}& 6+3&	10+5&	14+7&	18+9&	22+11&	30+15&	46+23\\
	\cmidrule{1-8}
\multicolumn{1}{c|}{$G^{X}$}&			$3.\Sym_6$&	$\Sym_4$&	$\Sym_3$&	$2$&	$2$&	$1$&$1$\\
\multicolumn{1}{c|}{$\bar{G}^{X}$}&		$\Sym_6$&	$\Sym_4$&	$\Sym_3$&	$2$&	$2$&	$1$&$1$\\\\ 
\\
\multicolumn{1}{c|}{$X$}&$E_6^4$&$E_8^3$\\
	\cmidrule{1-3}
\multicolumn{1}{c|}{$\ll$}	&12+4&	30+6,10,15 &
\\
	\cmidrule{1-3}
\multicolumn{1}{c|}{$G^{X}$}&	$\GL_2(3)$&$\Sym_3$&\\
\multicolumn{1}{c|}{$\bar{G}^{X}$}&	$\PGL_2(3)$&$\Sym_3$&
\end{tabular}
\end{center}
\end{table}

In the descriptions of the umbral groups $G^X$, and their permutation group quotients $\bar{G}^X$, we write $M_{24}$ and $M_{12}$ for the sporadic simple groups of Mathieu which act quintuply transitively on $24$ and $12$ points, respectively. (Cf. e.g. \cite{ATLAS}.) We write $\GL_n(q)$ for the general linear group of a vector space of dimension $n$ over a field with $q$ elements, and $\SL_n(q)$ is the subgroup of linear transformations with determinant $1$, \&c. The symbols $\AGL_3(2)$ denote the {\em affine general linear group}, obtained by adjoining translations to $\GL_3(2)$. We write $\Dih_n$ for the dihedral group of order $2n$, and $\Sym_n$ denotes the symmetric group on $n$ symbols. We use $n$ as a shorthand for a cyclic group of order $n$. 

We also use the notational convention of writing $A.B$ to denote the middle term in a short exact sequence $1\to A\to A.B\to B\to 1$. This introduces some ambiguity which is nonetheless easily navigated in practice. For example, $2.M_{12}$ is the unique (up to isomorphism) double cover of $M_{12}$ which is not $2\times M_{12}$. The group $\AGL_3(2)$ naturally embeds in $\GL_4(2)$, which in turn admits a unique (up to isomorphism) double cover $2.\GL_4(2)$ which is not a direct product. The group we denote $2.\AGL_3(2)$ is the 
preimage of $\AGL_3(2)<\GL_4(2)$ in $2.\GL_4(2)$ under the natural projection.

\begin{sidewaystable}
\subsection{Character Tables}\label{sec:gps:chars}
\begin{center}
\caption{Character table of $G^{X}\simeq M_{24}$, $X=A_1^{24}$}\label{tab:chars:irr:2}
\smallskip
\begin{small}
\begin{tabular}{c@{ }|c|@{ }r@{ }r@{ }r@{ }r@{ }r@{ }r@{ }r@{ }r@{ }r@{ }r@{ }r@{ }r@{ }r@{ }r@{ }r@{ }r@{ }r@{ }r@{ }r@{ }r@{ }r@{ }r@{ }r@{ }r@{ }r@{ }r} \toprule
$[g]$	&{FS}&1A	&2A	&2B	&3A	&3B	&4A	&4B	&4C	&5A	&6A	&6B	&7A	&7B	&8A	&10A	&11A	&12A	&12B	&14A	&14B	&15A	&15B	&21A	&21B	&23A	&23B	\\
	\midrule
$[g^{2}]$	&&1A	&1A	&1A	&3A	&3B	&2A	&2A	&2B	&5A	&3A	&3B	&7A	&7B	&4B	&5A	&11A	&6A	&6B	&7A	&7B	&15A	&15B	&21A	&21B	&23A	&23B	\\
$[g^{3}]$	&&1A	&2A	&2B	&1A	&1A	&4A	&4B	&4C	&5A	&2A	&2B	&7B	&7A	&8A	&10A	&11A	&4A	&4C	&14B	&14A	&5A	&5A	&7B	&7A	&23A	&23B	\\
$[g^{5}]$	&&1A	&2A	&2B	&3A	&3B	&4A	&4B	&4C	&1A	&6A	&6B	&7B	&7A	&8A	&2B	&11A	&12A	&12B	&14B	&14A	&3A	&3A	&21B	&21A	&23B	&23A	\\
$[g^{7}]$	&&1A	&2A	&2B	&3A	&3B	&4A	&4B	&4C	&5A	&6A	&6B	&1A	&1A	&8A	&10A	&11A	&12A	&12B	&2A	&2A	&15B	&15A	&3B	&3B	&23B	&23A	\\
$[g^{11}]$	&&1A	&2A	&2B	&3A	&3B	&4A	&4B	&4C	&5A	&6A	&6B	&7A	&7B	&8A	&10A	&1A	&12A	&12B	&14A	&14B	&15B	&15A	&21A	&21B	&23B	&23A	\\
$[g^{23}]$	&&1A	&2A	&2B	&3A	&3B	&4A	&4B	&4C	&5A	&6A	&6B	&7A	&7B	&8A	&10A	&11A	&12A	&12B	&14A	&14B	&15A	&15B	&21A	&21B	&1A	&1A	\\
	\midrule
$\chi_{1}$	&$+$&$1$	&$1$	&$1$	&$1$	&$1$	&$1$	&$1$	&$1$	&$1$	&$1$	&$1$	&$1$	&$1$	&$1$	&$1$	&$1$	&$1$	&$1$	&$1$	&$1$	&$1$	&$1$	&$1$	&$1$	&$1$	&$1$	\\
$\chi_{2}$	&$+$&$23$	&$7$	&$-1$	&$5$	&$-1$	&$-1$	&$3$	&$-1$	&$3$	&$1$	&$-1$	&$2$	&$2$	&$1$	&$-1$	&$1$	&$-1$	&$-1$	&$0$	&$0$	&$0$	&$0$	&$-1$	&$-1$	&$0$	&$0$	\\
$\chi_{3}$	&$\circ$&$45$	&$-3$	&$5$	&$0$	&$3$	&$-3$	&$1$	&$1$	&$0$	&$0$	&$-1$	&$b_7$	&$\overline{b_7}$	&$-1$	&$0$	&$1$	&$0$	&$1$	&$-b_7$	&$-\overline{b_7}$	&$0$	&$0$	&$b_7$	&$\overline{b_7}$	&$-1$	&$-1$	\\
$\chi_{4}$	&$\circ$&$45$	&$-3$	&$5$	&$0$	&$3$	&$-3$	&$1$	&$1$	&$0$	&$0$	&$-1$	&$\overline{b_7}$	&$b_7$	&$-1$	&$0$	&$1$	&$0$	&$1$	&$-\overline{b_7}$	&$-b_7$	&$0$	&$0$	&$\overline{b_7}$	&$b_7$	&$-1$	&$-1$	\\
$\chi_{5}$	&$\circ$&$231$	&$7$	&$-9$	&$-3$	&$0$	&$-1$	&$-1$	&$3$	&$1$	&$1$	&$0$	&$0$	&$0$	&$-1$	&$1$	&$0$	&$-1$	&$0$	&$0$	&$0$	&$b_{15}$	&$\overline{b_{15}}$	&$0$	&$0$	&$1$	&$1$	\\
$\chi_{6}$	&$\circ$&$231$	&$7$	&$-9$	&$-3$	&$0$	&$-1$	&$-1$	&$3$	&$1$	&$1$	&$0$	&$0$	&$0$	&$-1$	&$1$	&$0$	&$-1$	&$0$	&$0$	&$0$	&$\overline{b_{15}}$	&$b_{15}$	&$0$	&$0$	&$1$	&$1$	\\
$\chi_{7}$	&$+$&$252$	&$28$	&$12$	&$9$	&$0$	&$4$	&$4$	&$0$	&$2$	&$1$	&$0$	&$0$	&$0$	&$0$	&$2$	&$-1$	&$1$	&$0$	&$0$	&$0$	&$-1$	&$-1$	&$0$	&$0$	&$-1$	&$-1$	\\
$\chi_{8}$	&$+$&$253$	&$13$	&$-11$	&$10$	&$1$	&$-3$	&$1$	&$1$	&$3$	&$-2$	&$1$	&$1$	&$1$	&$-1$	&$-1$	&$0$	&$0$	&$1$	&$-1$	&$-1$	&$0$	&$0$	&$1$	&$1$	&$0$	&$0$	\\
$\chi_{9}$	&$+$&$483$	&$35$	&$3$	&$6$	&$0$	&$3$	&$3$	&$3$	&$-2$	&$2$	&$0$	&$0$	&$0$	&$-1$	&$-2$	&$-1$	&$0$	&$0$	&$0$	&$0$	&$1$	&$1$	&$0$	&$0$	&$0$	&$0$	\\
$\chi_{10}$&$\circ$&$770$	&$-14$	&$10$	&$5$	&$-7$	&$2$	&$-2$	&$-2$	&$0$	&$1$	&$1$	&$0$	&$0$	&$0$	&$0$	&$0$	&$-1$	&$1$	&$0$	&$0$	&$0$	&$0$	&$0$	&$0$	&$b_{23}$	&$\overline{b_{23}}$	\\
$\chi_{11}$&$\circ$&$770$	&$-14$	&$10$	&$5$	&$-7$	&$2$	&$-2$	&$-2$	&$0$	&$1$	&$1$	&$0$	&$0$	&$0$	&$0$	&$0$	&$-1$	&$1$	&$0$	&$0$	&$0$	&$0$	&$0$	&$0$	&$\overline{b_{23}}$	&$b_{23}$	\\
$\chi_{12}$&$\circ$&$990$	&$-18$	&$-10$	&$0$	&$3$	&$6$	&$2$	&$-2$	&$0$	&$0$	&$-1$	&$b_7$	&$\overline{b_7}$	&$0$	&$0$	&$0$	&$0$	&$1$	&$b_7$	&$\overline{b_7}$	&$0$	&$0$	&$b_7$	&$\overline{b_7}$	&$1$	&$1$	\\
$\chi_{13}$&$\circ$&$990$	&$-18$	&$-10$	&$0$	&$3$	&$6$	&$2$	&$-2$	&$0$	&$0$	&$-1$	&$\overline{b_7}$	&$b_7$	&$0$	&$0$	&$0$	&$0$	&$1$	&$\overline{b_7}$	&$b_7$	&$0$	&$0$	&$\overline{b_7} $	&$b_7$	&$1$	&$1$	\\
$\chi_{14}$&$+$&$1035$	&$27$	&$35$	&$0$	&$6$	&$3$	&$-1$	&$3$	&$0$	&$0$	&$2$	&$-1$	&$-1$	&$1$	&$0$	&$1$	&$0$	&$0$	&$-1$	&$-1$	&$0$	&$0$	&$-1$	&$-1$	&$0$	&$0$	\\
$\chi_{15}$&$\circ$&$1035$	&$-21$	&$-5$	&$0$	&$-3$	&$3$	&$3$	&$-1$	&$0$	&$0$	&$1$	&$2b_7$	&$2\overline{b_7}$	&$-1$	&$0$	&$1$	&$0$	&$-1$	&$0$	&$0$	&$0$	&$0$	&$-b_7$	&$-\overline{b_7} $	&$0$	&$0$	\\
$\chi_{16}$&$\circ$&$1035$	&$-21$	&$-5$	&$0$	&$-3$	&$3$	&$3$	&$-1$	&$0$	&$0$	&$1$	&$2\overline{b_7}$	&$2b_7$	&$-1$	&$0$	&$1$	&$0$	&$-1$	&$0$	&$0$	&$0$	&$0$	&$-\overline{b_7}$	&$-b_7 $	&$0$	&$0$	\\
$\chi_{17}$&$+$&$1265$	&$49$	&$-15$	&$5$	&$8$	&$-7$	&$1$	&$-3$	&$0$	&$1$	&$0$	&$-2$	&$-2$	&$1$	&$0$	&$0$	&$-1$	&$0$	&$0$	&$0$	&$0$	&$0$	&$1$	&$1$	&$0$	&$0$	\\
$\chi_{18}$&$+$&$1771$	&$-21$	&$11$	&$16$	&$7$	&$3$	&$-5$	&$-1$	&$1$	&$0$	&$-1$	&$0$	&$0$	&$-1$	&$1$	&$0$	&$0$	&$-1$	&$0$	&$0$	&$1$	&$1$	&$0$	&$0$	&$0$	&$0$	\\
$\chi_{19}$&$+$&$2024$	&$8$	&$24$	&$-1$	&$8$	&$8$	&$0$	&$0$	&$-1$	&$-1$	&$0$	&$1$	&$1$	&$0$	&$-1$	&$0$	&$-1$	&$0$	&$1$	&$1$	&$-1$	&$-1$	&$1$	&$1$	&$0$	&$0$	\\
$\chi_{20}$&$+$&$2277$	&$21$	&$-19$	&$0$	&$6$	&$-3$	&$1$	&$-3$	&$-3$	&$0$	&$2$	&$2$	&$2$	&$-1$	&$1$	&$0$	&$0$	&$0$	&$0$	&$0$	&$0$	&$0$	&$-1$	&$-1$	&$0$	&$0$	\\
$\chi_{21}$&$+$&$3312$	&$48$	&$16$	&$0$	&$-6$	&$0$	&$0$	&$0$	&$-3$	&$0$	&$-2$	&$1$	&$1$	&$0$	&$1$	&$1$	&$0$	&$0$	&$-1$	&$-1$	&$0$	&$0$	&$1$	&$1$	&$0$	&$0$	\\
$\chi_{22}$&$+$&$3520$	&$64$	&$0$	&$10$	&$-8$	&$0$	&$0$	&$0$	&$0$	&$-2$	&$0$	&$-1$	&$-1$	&$0$	&$0$	&$0$	&$0$	&$0$	&$1$	&$1$	&$0$	&$0$	&$-1$	&$-1$	&$1$	&$1$	\\
$\chi_{23}$&$+$&$5313$	&$49$	&$9$	&$-15$	&$0$	&$1$	&$-3$	&$-3$	&$3$	&$1$	&$0$	&$0$	&$0$	&$-1$	&$-1$	&$0$	&$1$	&$0$	&$0$	&$0$	&$0$	&$0$	&$0$	&$0$	&$0$	&$0$	\\
$\chi_{24}$&$+$&$5544$	&$-56$	&$24$	&$9$	&$0$	&$-8$	&$0$	&$0$	&$-1$	&$1$	&$0$	&$0$	&$0$	&$0$	&$-1$	&$0$	&$1$	&$0$	&$0$	&$0$	&$-1$	&$-1$	&$0$	&$0$	&$1$	&$1$	\\
$\chi_{25}$&$+$&$5796$	&$-28$	&$36$	&$-9$	&$0$	&$-4$	&$4$	&$0$	&$1$	&$-1$	&$0$	&$0$	&$0$	&$0$	&$1$	&$-1$	&$-1$	&$0$	&$0$	&$0$	&$1$	&$1$	&$0$	&$0$	&$0$	&$0$	\\
$\chi_{26}$&$+$&$10395$	&$-21$	&$-45$	&$0$	&$0$	&$3$	&$-1$	&$3$	&$0$	&$0$	&$0$	&$0$	&$0$	&$1$	&$0$	&$0$	&$0$	&$0$	&$0$	&$0$	&$0$	&$0$	&$0$	&$0$	&$-1$	&$-1$	\\\bottomrule
\end{tabular}
\end{small}
\end{center}
\end{sidewaystable}

\begin{sidewaystable}

\begin{center}
\caption{Character table of $G^{X}\simeq 2.M_{12}$, $X=A_2^{12}$}\label{tab:chars:irr:3}

\smallskip

\begin{tabular}{c@{ }|c|@{\;}r@{ }r@{ }r@{ }r@{ }r@{ }r@{ }r@{ }r@{ }r@{ }r@{ }r@{ }r@{ }r@{ }r@{ }r@{ }r@{ }r@{ }r@{ }r@{ }r@{ }r@{ }r@{ }r@{ }r@{ }r@{ }r}\toprule
$[g]$&FS&   	1A&   	2A&   	4A&   	2B&   	2C&   	3A&   	6A&   	3B&   	6B&   	4B&   	4C&   	5A&   	10A&   	12A&   	6C&   	6D&   	8A&   	8B&   	 8C&   8D&   	20A&   	20B&   	11A&   	22A&   	11B&   	22B\\ 
	\midrule
$[g^2]$&&	1A&			1A&			2A&			1A&			1A&			3A&			3A&			3B&			3B&			2B&			2B&			5A&			5A&			6B&			3A&			3A&			4B&			4B&			4C&			4C&			10A&		10A&		11B&		11B&		11A&		11A\\		
$[g^3]$& &  1A&   		2A&   		4A&   		2B&   		2C&   		1A&   		2A&   		1A&   		2A&   		4B& 	  		4C&   		5A&   		10A&   		4A&   		2B&   		2C&   		8A&   		8B&   	 	8C&  		8D&   		20A&   		20B&   		11A&   		22A&   		11B&   		22B\\
$[g^5]$& &  1A&   		2A&   		4A&   		2B&   		2C&   		3A&   		6A&   		3B&   		6B&   		4B& 	  		4C&   		1A&   		2A&   		12A&   		6C&   		6D&   		8B&   		8A&   	 	8D&  		8C&   		4A&   		4A&   		11A&   		22A&   		11B&   		22B\\
$[g^{11}]$&&   1A&   		2A&   		4A&   		2B&   		2C&   		3A&   		6A&   		3B&   		6B&   		4B& 	  		4C&   		5A&   		10A&   		12A&   		6C&   		6D&   		8A&   		8B&   	 	8C&  		8D&   		20B&   		20A&   		1A&   		2A&   		1A&   		2A\\
	\midrule
$\chi_{1}$&$+$&   $1$&   $1$&   $1$&   $1$&   $1$&   $1$&   $1$&   $1$&   $1$&   $
1$&   $1$&   $1$&   $1$&   $1$&   $1$&   $1$&   $1$&   $1$&   $1$&   $1$&   $
1$&   $1$&   $1$&   $1$&   $1$&   $1$\\
$\chi_{2}$&$+$&   $11$&   $11$&   $-1$&   $3$&   $3$&   $2$&   $2$&   $-1$&   $
-1$&   $-1$&   $3$&   $1$&   $1$&   $-1$&   $0$&   $0$&   $-1$&   $-1$&   $
1$&   $1$&   $-1$&   $-1$&   $0$&   $0$&   $0$&   $0$\\
$\chi_{3}$&$+$&   $11$&   $11$&   $-1$&   $3$&   $3$&   $2$&   $2$&   $-1$&   $
-1$&   $3$&   $-1$&   $1$&   $1$&   $-1$&   $0$&   $0$&   $1$&   $1$&   $
-1$&   $-1$&   $-1$&   $-1$&   $0$&   $0$&   $0$&   $0$\\
$\chi_{4}$&$\circ$&   $16$&   $16$&   $4$&   $0$&   $0$&   $-2$&   $-2$&   $1$&   $
1$&   $0$&   $0$&   $1$&   $1$&   $1$&   $0$&   $0$&   $0$&   $0$&   $0$&   $
0$&   $-1$&   $-1$&   $b_{11}$&   $
b_{11}$&   $\overline{b_{11}}$&   $\overline{b_{11}}$\\
$\chi_{5}$&$\circ$&   $16$&   $16$&   $4$&   $0$&   $0$&   $-2$&   $-2$&   $1$&   $
1$&   $0$&   $0$&   $1$&   $1$&   $1$&   $0$&   $0$&   $0$&   $0$&   $0$&   $
0$&   $-1$&   $-1$&   $\overline{b_{11}}$&   $
\overline{b_{11}}$&   $b_{11}$&   $b_{11}$\\
$\chi_{6}$&$+$&   $45$&   $45$&   $5$&   $-3$&   $-3$&   $0$&   $0$&   $3$&   $
3$&   $1$&   $1$&   $0$&   $0$&   $-1$&   $0$&   $0$&   $-1$&   $-1$&   $
-1$&   $-1$&   $0$&   $0$&   $1$&   $1$&   $1$&   $1$\\
$\chi_{7}$&$+$&   $54$&   $54$&   $6$&   $6$&   $6$&   $0$&   $0$&   $0$&   $
0$&   $2$&   $2$&   $-1$&   $-1$&   $0$&   $0$&   $0$&   $0$&   $0$&   $
0$&   $0$&   $1$&   $1$&   $-1$&   $-1$&   $-1$&   $-1$\\
$\chi_{8}$&$+$&   $55$&   $55$&   $-5$&   $7$&   $7$&   $1$&   $1$&   $1$&   $
1$&   $-1$&   $-1$&   $0$&   $0$&   $1$&   $1$&   $1$&   $-1$&   $-1$&   $
-1$&   $-1$&   $0$&   $0$&   $0$&   $0$&   $0$&   $0$\\
$\chi_{9}$&$+$&   $55$&   $55$&   $-5$&   $-1$&   $-1$&   $1$&   $1$&   $1$&   $
1$&   $3$&   $-1$&   $0$&   $0$&   $1$&   $-1$&   $-1$&   $-1$&   $-1$&   $
1$&   $1$&   $0$&   $0$&   $0$&   $0$&   $0$&   $0$\\
$\chi_{10}$&$+$&   $55$&   $55$&   $-5$&   $-1$&   $-1$&   $1$&   $1$&   $1$&   $
1$&   $-1$&   $3$&   $0$&   $0$&   $1$&   $-1$&   $-1$&   $1$&   $1$&   $
-1$&   $-1$&   $0$&   $0$&   $0$&   $0$&   $0$&   $0$\\
$\chi_{11}$&$+$&   $66$&   $66$&   $6$&   $2$&   $2$&   $3$&   $3$&   $0$&   $
0$&   $-2$&   $-2$&   $1$&   $1$&   $0$&   $-1$&   $-1$&   $0$&   $0$&   $
0$&   $0$&   $1$&   $1$&   $0$&   $0$&   $0$&   $0$\\
$\chi_{12}$&$+$&   $99$&   $99$&   $-1$&   $3$&   $3$&   $0$&   $0$&   $3$&   $
3$&   $-1$&   $-1$&   $-1$&   $-1$&   $-1$&   $0$&   $0$&   $1$&   $1$&   $
1$&   $1$&   $-1$&   $-1$&   $0$&   $0$&   $0$&   $0$\\
$\chi_{13}$&$+$&   $120$&   $120$&   $0$&   $-8$&   $-8$&   $3$&   $3$&   $0$&   $
0$&   $0$&   $0$&   $0$&   $0$&   $0$&   $1$&   $1$&   $0$&   $0$&   $0$&   $
0$&   $0$&   $0$&   $-1$&   $-1$&   $-1$&   $-1$\\
$\chi_{14}$&$+$&   $144$&   $144$&   $4$&   $0$&   $0$&   $0$&   $0$&   $-3$&   $
-3$&   $0$&   $0$&   $-1$&   $-1$&   $1$&   $0$&   $0$&   $0$&   $0$&   $
0$&   $0$&   $-1$&   $-1$&   $1$&   $1$&   $1$&   $1$\\
$\chi_{15}$&$+$&   $176$&   $176$&   $-4$&   $0$&   $0$&   $-4$&   $-4$&   $
-1$&   $-1$&   $0$&   $0$&   $1$&   $1$&   $-1$&   $0$&   $0$&   $0$&   $
0$&   $0$&   $0$&   $1$&   $1$&   $0$&   $0$&   $0$&   $0$\\
$\chi_{16}$&$\circ$&   $10$&   $-10$&   $0$&   $-2$&   $2$&   $1$&   $-1$&   $-2$&   $
2$&   $0$&   $0$&   $0$&   $0$&   $0$&   $1$&   $-1$&   $a_{2}$&   $
\overline{a_{2}}$&   $a_{2}$&   $\overline{a_{2}}$&   $0$&   $0$&   $-1$&   $
1$&   $-1$&   $1$\\
$\chi_{17}$&$\circ$&   $10$&   $-10$&   $0$&   $-2$&   $2$&   $1$&   $-1$&   $-2$&   $
2$&   $0$&   $0$&   $0$&   $0$&   $0$&   $1$&   $-1$&   $\overline{a_{2}}$&   $
a_{2}$&   $\overline{a_{2}}$&   $a_{2}$&   $0$&   $0$&   $-1$&   $
1$&   $-1$&   $1$\\
$\chi_{18}$&$+$&   $12$&   $-12$&   $0$&   $4$&   $-4$&   $3$&   $-3$&   $0$&   $
0$&   $0$&   $0$&   $2$&   $-2$&   $0$&   $1$&   $-1$&   $0$&   $0$&   $
0$&   $0$&   $0$&   $0$&   $1$&   $-1$&   $1$&   $-1$\\
$\chi_{19}$&$-$&   $32$&   $-32$&   $0$&   $0$&   $0$&   $-4$&   $4$&   $2$&   $
-2$&   $0$&   $0$&   $2$&   $-2$&   $0$&   $0$&   $0$&   $0$&   $0$&   $
0$&   $0$&   $0$&   $0$&   $-1$&   $1$&   $-1$&   $1$\\
$\chi_{20}$&$\circ$&   $44$&   $-44$&   $0$&   $4$&   $-4$&   $-1$&   $1$&   $2$&   $
-2$&   $0$&   $0$&   $-1$&   $1$&   $0$&   $1$&   $-1$&   $0$&   $0$&   $
0$&   $0$&   $a_{5}$&   $\overline{a_{5}}$&   $0$&   $0$&   $0$&   $0$\\
$\chi_{21}$&$\circ$&   $44$&   $-44$&   $0$&   $4$&   $-4$&   $-1$&   $1$&   $2$&   $
-2$&   $0$&   $0$&   $-1$&   $1$&   $0$&   $1$&   $-1$&   $0$&   $0$&   $
0$&   $0$&   $\overline{a_{5}}$&   $a_{5}$&   $0$&   $0$&   $0$&   $0$\\
$\chi_{22}$&$\circ$&   $110$&   $-110$&   $0$&   $-6$&   $6$&   $2$&   $-2$&   $
2$&   $-2$&   $0$&   $0$&   $0$&   $0$&   $0$&   $0$&   $0$&   $
a_{2}$&   $\overline{a_{2}}$&   $\overline{a_{2}}$&   $a_{2}$&   $0$&   $
0$&   $0$&   $0$&   $0$&   $0$\\
$\chi_{23}$&$\circ$&   $110$&   $-110$&   $0$&   $-6$&   $6$&   $2$&   $-2$&   $
2$&   $-2$&   $0$&   $0$&   $0$&   $0$&   $0$&   $0$&   $0$&   $
\overline{a_{2}}$&   $a_{2}$&   $a_{2}$&   $\overline{a_{2}}$&   $0$&   $
0$&   $0$&   $0$&   $0$&   $0$\\
$\chi_{24}$&$+$&   $120$&   $-120$&   $0$&   $8$&   $-8$&   $3$&   $-3$&   $
0$&   $0$&   $0$&   $0$&   $0$&   $0$&   $0$&   $-1$&   $1$&   $0$&   $0$&   $
0$&   $0$&   $0$&   $0$&   $-1$&   $1$&   $-1$&   $1$\\
$\chi_{25}$&$\circ$&   $160$&   $-160$&   $0$&   $0$&   $0$&   $-2$&   $2$&   $
-2$&   $2$&   $0$&   $0$&   $0$&   $0$&   $0$&   $0$&   $0$&   $0$&   $0$&   $
0$&   $0$&   $0$&   $0$&   $-b_{11}$&   $
b_{11}$&   $-\overline{b_{11}}$&   $\overline{b_{11}}$\\
$\chi_{26}$&$\circ$&   $160$&   $-160$&   $0$&   $0$&   $0$&   $-2$&   $2$&   $
-2$&   $2$&   $0$&   $0$&   $0$&   $0$&   $0$&   $0$&   $0$&   $0$&   $0$&   $
0$&   $0$&   $0$&   $0$&   $-\overline{b_{11}}$&   $
\overline{b_{11}}$&   $-b_{11}$&   $b_{11}$\\ \bottomrule
\end{tabular}
\end{center}

\end{sidewaystable}

\begin{table}
\begin{center}
\caption{Character table of ${G}^{X}\simeq 2.\AGL_3(2)$, $X=A_3^8$}\label{tab:chars:irr:4}

\smallskip
\begin{small}
\begin{tabular}{c|c|rrrrrrrrrrrrrrrrrrrrrrrrrrr} \toprule
$[g]$&FS&	   1A&   2A&   2B&   4A&   4B&   2C&   3A&   6A&   6B&   6C&   8A&   4C&   7A&   14A&   7B&   14B
	\\ \midrule
$[g^2]$&&   	1A&   1A&   	1A&   	2A&			2B&			1A&   	3A&   	3A&   		3A&   	3A&   	4A&   	2C&   	7A&		7A&   7B&   7B\\ 
$[g^3]$&&   	1A&   2A&   	2B&   	4A&			4B&			2C&   	1A&   	2A&   		2B&   	2B&   	8A&   	4C&   	7B&   	14B&   7A&   14A\\ 
$[g^7]$&&   	1A&   2A&   	2B&   	4A& 			4B&  		2C&   	3A&   	6A&   		6B&   	6C&   	8A&   	4C&   	1A&   	2A&   1A&   2A\\ 
	\midrule	
$\chi_{1}$&$+$&   $1$&   $1$&   $1$&   $1$&   $1$&   $1$&   $1$&   $1$&   $1$&   $
1$&   $1$&   $1$&   $1$&   $1$&   $1$&   $1$\\
$\chi_{2}$&$\circ$&   $3$&   $3$&   $3$&   $-1$&   $-1$&   $-1$&   $0$&   $0$&   $
0$&   $0$&   $1$&   $1$&   $b_{7}$&   $b_{7}$&   $\overline{b_{7}}$&   $\overline{b_{7}}$\\
$\chi_{3}$&$\circ$&   $3$&   $3$&   $3$&   $-1$&   $-1$&   $-1$&   $0$&   $0$&   $
0$&   $0$&   $1$&   $1$&   $\overline{b_{7}}$&   $\overline{b_{7}}$&   $b_{7}$&   $b_{7}$\\
$\chi_{4}$&$+$&   $6$&   $6$&   $6$&   $2$&   $2$&   $2$&   $0$&   $0$&   $0$&   $
0$&   $0$&   $0$&   $-1$&   $-1$&   $-1$&   $-1$\\
$\chi_{5}$&$+$&   $7$&   $7$&   $7$&   $-1$&   $-1$&   $-1$&   $1$&   $1$&   $
1$&   $1$&   $-1$&   $-1$&   $0$&   $0$&   $0$&   $0$\\
$\chi_{6}$&$+$&   $8$&   $8$&   $8$&   $0$&   $0$&   $0$&   $-1$&   $-1$&   $
-1$&   $-1$&   $0$&   $0$&   $1$&   $1$&   $1$&   $1$\\
$\chi_{7}$&$+$&   $7$&   $7$&   $-1$&   $3$&   $-1$&   $-1$&   $1$&   $1$&   $
-1$&   $-1$&   $1$&   $-1$&   $0$&   $0$&   $0$&   $0$\\
$\chi_{8}$&$+$&   $7$&   $7$&   $-1$&   $-1$&   $-1$&   $3$&   $1$&   $1$&   $
-1$&   $-1$&   $-1$&   $1$&   $0$&   $0$&   $0$&   $0$\\
$\chi_{9}$&$+$&   $14$&   $14$&   $-2$&   $2$&   $-2$&   $2$&   $-1$&   $-1$&   $
1$&   $1$&   $0$&   $0$&   $0$&   $0$&   $0$&   $0$\\
$\chi_{10}$&$+$&   $21$&   $21$&   $-3$&   $1$&   $1$&   $-3$&   $0$&   $0$&   $
0$&   $0$&   $-1$&   $1$&   $0$&   $0$&   $0$&   $0$\\
$\chi_{11}$&$+$&   $21$&   $21$&   $-3$&   $-3$&   $1$&   $1$&   $0$&   $0$&   $
0$&   $0$&   $1$&   $-1$&   $0$&   $0$&   $0$&   $0$\\
$\chi_{12}$&$+$&   $8$&   $-8$&   $0$&   $0$&   $0$&   $0$&   $2$&   $-2$&   $
0$&   $0$&   $0$&   $0$&   $1$&   $-1$&   $1$&   $-1$\\
$\chi_{13}$&$\circ$&   $8$&   $-8$&   $0$&   $0$&   $0$&   $0$&   $-1$&   $1$&   $a_{3}$&   $\overline{a_{3}}$&   $0$&   $0$&   $1$&   $-1$&   $1$&   $-1$\\
$\chi_{14}$&$\circ$&   $8$&   $-8$&   $0$&   $0$&   $0$&   $0$&   $-1$&   $1$&   $\overline{a_{3}}$&   $a_{3}$&   $0$&   $0$&   $1$&   $-1$&   $1$&   $-1$\\
$\chi_{15}$&$\circ$&   $24$&   $-24$&   $0$&   $0$&   $0$&   $0$&   $0$&   $0$&   $0$&   $0$&   $0$&   $0$&   $\overline{b_{7}}$&   $-\overline{b_{7}}$&   $b_{7}$&   $-b_{7}$\\
$\chi_{16}$&$\circ$&   $24$&   $-24$&   $0$&   $0$&   $0$&   $0$&   $0$&   $0$&   $0$&   $0$&   $0$&   $0$&   $b_{7}$&   $-b_{7}$&   $\overline{b_{7}}$&   $-\overline{b_{7}}$\\\bottomrule
\end{tabular}
\end{small}
\end{center}

\begin{center}
\caption{Character table of ${G}^{X}\simeq \GL_2(5)/2$, $X=A_4^6$}\label{tab:chars:irr:5}

\smallskip
\begin{small}
\begin{tabular}{c|c|rrrrrrrrrrrrrrrrrrrrrrrrrrr}\toprule
$[g]$&FS&   1A&   2A&   2B&   2C&   3A&   6A&   5A&   10A&   4A&   4B&   4C&   4D&   12A&   12B\\
	\midrule
$[g^2]$&&	1A&	1A&	1A&	1A&	3A&	3A&	5A&	5A&		2A&	2A&	2C&	2C&	6A&	6A\\
$[g^3]$&&	1A&	2A&	2B&	2C&	1A&	2A&	5A&	10A&	4B&	4A&	4D&	4C&	4B&	4A\\
$[g^5]$&&	1A&	2A&	2B&	2C&	3A&	6A&	1A&	2A&		4A&	4B&	4C&	4D&	12A&	12B\\
	\midrule
${\chi}_{1}$&$+$&   $1$&   $1$&   $1$&   $1$&   $1$&   $1$&   $1$&   $1$&   $1$&   $1$&   $1$&   $1$&   $1$&   $1$\\
${\chi}_{2}$&$+$&   $1$&   $1$&   $1$&   $1$&   $1$&   $1$&   $1$&   $1$&   $
-1$&   $-1$&   $-1$&   $-1$&   $-1$&   $-1$\\
${\chi}_{3}$&$+$&   $4$&   $4$&   $0$&   $0$&   $1$&   $1$&   $-1$&   $-1$&   $
2$&   $2$&   $0$&   $0$&   $-1$&   $-1$\\
${\chi}_{4}$&$+$&   $4$&   $4$&   $0$&   $0$&   $1$&   $1$&   $-1$&   $-1$&   $
-2$&   $-2$&   $0$&   $0$&   $1$&   $1$\\
${\chi}_{5}$&$+$&   $5$&   $5$&   $1$&   $1$&   $-1$&   $-1$&   $0$&   $0$&   $
1$&   $1$&   $-1$&   $-1$&   $1$&   $1$\\
${\chi}_{6}$&$+$&   $5$&   $5$&   $1$&   $1$&   $-1$&   $-1$&   $0$&   $0$&   $-1$&   $-1$&   $1$&   $1$&   $-1$&   $-1$\\
${\chi}_{7}$&$+$&   $6$&   $6$&   $-2$&   $-2$&   $0$&   $0$&   $1$&   $1$&   $
0$&   $0$&   $0$&   $0$&   $0$&   $0$\\
${\chi}_{8}$&$\circ$&   $1$&   $-1$&   $1$&   $-1$&   $1$&   $-1$&   $1$&   $-1$&   $
a_1$&   $-a_1$&   $a_1$&   $-a_1$&   $a_1$&   $-a_1$\\
${\chi}_{9}$&$\circ$&   $1$&   $-1$&   $1$&   $-1$&   $1$&   $-1$&   $1$&   $-1$&   $
-a_1$&   $a_1$&   $-a_1$&   $a_1$&   $-a_1$&   $a_1$\\
${\chi}_{10}$&$\circ$&   $4$&   $-4$&   $0$&   $0$&   $1$&   $-1$&   $-1$&   $1$&   $
2a_1$&   $-2a_1$&   $0$&   $0$&   $-a_1$&   $a_1$\\
${\chi}_{11}$&$\circ$&   $4$&   $-4$&   $0$&   $0$&   $1$&   $-1$&   $-1$&   $1$&   $
-2a_1$&   $2a_1$&   $0$&   $0$&   $a_1$&   $-a_1$\\
${\chi}_{12}$&$\circ$&   $5$&   $-5$&   $1$&   $-1$&   $-1$&   $1$&   $0$&   $0$&   $
a_1$&   $-a_1$&   $-a_1$&   $a_1$&   $a_1$&   $-a_1$\\
${\chi}_{13}$&$\circ$&   $5$&   $-5$&   $1$&   $-1$&   $-1$&   $1$&   $0$&   $0$&   $
-a_1$&   $a_1$&   $a_1$&   $-a_1$&   $-a_1$&   $a_1$\\
${\chi}_{14}$&$+$&   $6$&   $-6$&   $-2$&   $2$&   $0$&   $0$&   $1$&   $-1$&   $
0$&   $0$&   $0$&   $0$&   $0$&   $0$\\\bottomrule
\end{tabular}
\end{small}
\end{center}
\end{table}

\begin{table}
\begin{center}
\caption{Character table of ${G}^{X}\simeq 
	\GL_2(3)$, $X\in\{A_5^4D_4, E_6^4\}$}\label{tab:chars:irr:6}
\smallskip
\begin{tabular}{c|c|rrrrrrrr}\toprule
$[g]$	&FS&   1A&   2A&   2B&   4A&   3A&   6A&   8A&   8B\\
	\midrule
$[g^2]$ 	&&1A&	1A&	1A&	2A&	3A&	3A&	4A&	4A\\
$[g^3]$	&&1A&	2A&	2B&	4A&	1A&	2A&	8A&	8B\\
	\midrule
${\chi}_{1}$&$+$&   $1$&   $1$&   $1$&   $1$&   $1$&   $1$&   $1$&   $1$\\
${\chi}_{2}$&$+$&   $1$&   $1$&   $-1$&   $1$&   $1$&   $1$&   $-1$&   $-1$\\
${\chi}_{3}$&$+$&   $2$&   $2$&   $0$&   $2$&   $-1$&   $-1$&   $0$&   $0$\\
${\chi}_{4}$&$+$&   $3$&   $3$&   $-1$&   $-1$&   $0$&   $0$&   $1$&   $1$\\
${\chi}_{5}$&$+$&   $3$&   $3$&   $1$&   $-1$&   $0$&   $0$&   $-1$&   $-1$\\
${\chi}_{6}$&$\circ$&  $2$&   $-2$&   $0$&   $0$&   $-1$&   $1$&   $a_2$&   $\overline{a_2}$\\
${\chi}_{7}$&$\circ$&   $2$&   $-2$&   $0$&   $0$&   $-1$&   $1$&   $\overline{a_2}$&   $a_2$\\
${\chi}_{8}$&$+$&   $4$&   $-4$&   $0$&   $0$&   $1$&   $-1$&   $0$&   $0$\\
	\bottomrule
\end{tabular}
\end{center}
\end{table}

\begin{table}
\begin{center}
\caption{Character table of ${G}^{X}\simeq 3.\Sym_6$, $X=D_4^6$}\label{tab:chars:irr:6+3}
\smallskip
\begin{small}
\begin{tabular}{c|c|rrrrrrrrrrrrrrrrrrrrrrrrrrr}\toprule
$[g]$&FS&   1A&   3A&   2A&   6A&   3B&   3C&   4A&   12A&   5A&   15A&   15B&   2B&   2C&   4B&	6B&	6C\\
	\midrule
$[g^2]$&&	1A&	3A&	1A&3A&	3B&3C&2A&6A&5A&15A&15B&1A&1A&2A&3B&3C\\
$[g^3]$&&	1A&	1A&	2A&2A&	1A&1A&4A&4A&5A&5A&5A&2B&2C&4B&2B&2C\\
$[g^5]$&&	1A&	3A&	2A&	6A&3B&3C&4A&12A&1A&3A&3A&2B&2C&4B&6B&6C\\
	\midrule
${\chi}_{1}$&$+$&   $1$&   $1$&   $1$&   $1$&   $1$&   $1$&   $1$&   $1$&   $1$&   $1$&   $1$&   $1$&   $1$&   $1$&   $1$&   $1$\\
${\chi}_{2}$&$+$&   $1$&   $1$&   $1$&   $1$&   $1$&   $1$&   $1$&   $1$&   $1$&   $1$&   $1$&   $-1$&   $-1$&   $-1$&   $-1$&   $-1$\\
${\chi}_{3}$&$+$&   $5$&   $5$&   $1$&   $1$&   $2$&   $-1$&   $-1$&   $-1$&   $0$&   $0$&   $0$&   $3$&   $-1$&   $1$&   $0$&   $-1$\\
${\chi}_{4}$&$+$&   $5$&   $5$&   $1$&   $1$&   $2$&   $-1$&   $-1$&   $-1$&   $0$&   $0$&   $0$&   $-3$&   $1$&   $-1$&   $0$&   $1$\\
${\chi}_{5}$&$+$&   $5$&   $5$&   $1$&   $1$&   $-1$&   $2$&   $-1$&   $-1$&   $0$&   $0$&   $0$&   $-1$&   $3$&   $1$&   $-1$&   $0$\\
${\chi}_{6}$&$+$&   $5$&   $5$&   $1$&   $1$&   $-1$&   $2$&   $-1$&   $-1$&   $0$&   $0$&   $0$&   $1$&   $-3$&   $-1$&   $1$&   $0$\\
${\chi}_{7}$&$+$&   $16$&   $16$&   $0$&   $0$&   $-2$&   $-2$&   $0$&   $0$&   $1$&   $1$&   $1$&   $0$&   $0$&   $0$&   $0$&   $0$\\
${\chi}_{8}$&$+$&   $9$&   $9$&   $1$&   $1$&   $0$&   $0$&   $1$&   $1$&   $-1$&   $-1$&   $-1$&   $3$&   $3$&   $-1$&   $0$&   $0$\\
${\chi}_{9}$&$+$&   $9$&   $9$&   $1$&   $1$&   $0$&   $0$&   $1$&   $1$&   $-1$&   $-1$&   $-1$&   $-3$&   $-3$&   $1$&   $0$&   $0$\\
${\chi}_{10}$&$+$&   $10$&   $10$&   $-2$&   $-2$&   $1$&   $1$&   $0$&   $0$&   $0$&   $0$&   $0$&   $2$&   $-2$&   $0$&   $-1$&   $1$\\
${\chi}_{11}$&$+$&   $10$&   $10$&   $-2$&   $-2$&   $1$&   $1$&   $0$&   $0$&   $0$&   $0$&   $0$&   $-2$&   $2$&   $0$&   $1$&   $-1$\\
${\chi}_{12}$&$\circ$&   $6$&   $-3$&   $-2$&   $1$&   $0$&   $0$&   $2$&   $-1$&   $1$&   $b_{15}$&   $\overline{b_{15}}$&   $0$&   $0$&   $0$&   $0$&   $0$\\
${\chi}_{13}$&$\circ$&   $6$&   $-3$&   $-2$&   $1$&   $0$&   $0$&   $2$&   $-1$&   $1$&   $\overline{b_{15}}$&   $b_{15}$&   $0$&   $0$&   $0$&   $0$&   $0$\\
${\chi}_{14}$&$+$&   $12$&   $-6$&   $4$&   $-2$&   $0$&   $0$&   $0$&   $0$&   $2$&   $-1$&   $-1$&   $0$&   $0$&   $0$&   $0$&   $0$\\
${\chi}_{15}$&$+$&   $18$&   $-9$&   $2$&   $-1$&   $0$&   $0$&   $2$&   $-1$&   $-2$&   $1$&   $1$&   $0$&   $0$&   $0$&   $0$&   $0$\\
${\chi}_{16}$&$+$&   $30$&   $-15$&   $-2$&   $1$&   $0$&   $0$&   $-2$&   $1$&   $0$&   $0$&   $0$&   $0$&   $0$&   $0$&   $0$&   $0$\\\bottomrule
\end{tabular}
\end{small}
\end{center}
\end{table}

\begin{table}
\begin{center}
\caption{Character table of ${G}^{X}\simeq\SL_2(3)$, $X=A_6^4$}\label{tab:chars:irr:7}
\smallskip
\begin{tabular}{c|c|rrrrrrr}\toprule
$[g]$	&FS&   1A&   2A&   4A&   3A&   6A&   3B&   6B\\
	\midrule
$[g^2]$ 	&&1A	&	1A&		2A&		3B&		3A&		3A&		3B\\
$[g^3]$	&&1A	&	2A&		4A&		1A&		2A&		1A&		2A\\
	\midrule
$\chi_1$&$+$&   $1$&   $1$&   $1$&   $1$&   $1$&   $1$&   $1$\\
$\chi_2$&$\circ$&   $1$&   $1$&   $1$&   ${b_3}$&   $\overline{b_3}$&   $\overline{b_3}$&   ${b_3}$\\
$\chi_3$&$\circ$&   $1$&   $1$&   $1$&   $\overline{b_3}$&   ${b_3}$&   ${b_3}$&   $\overline{b_3}$\\
$\chi_4$&$+$&   $3$&   $3$&   $-1$&   $0$&   $0$&   $0$&   $0$\\
$\chi_5$&$-$&   $2$&   $-2$&   $0$&   $-1$&   $1$&   $-1$&   $1$\\
$\chi_6$&$\circ$&   $2$&   $-2$&   $0$&   $-\overline{b_3}$&   ${b_3}$&   $-{b_3}$&   $\overline{b_3}$\\
$\chi_7$&$\circ$&   $2$&   $-2$&   $0$&   $-{b_3}$&   $\overline{b_3}$&   $-\overline{b_3}$&   ${b_3}$\\\bottomrule
\end{tabular}
\end{center}
\end{table}

\begin{table}
\begin{center}
\caption{Character table of ${G}^{X}\simeq \Dih_4$, $X=A_7^2D_5^2$}\label{tab:chars:irr:8}
\smallskip
\begin{tabular}{c|c|rrrrr}\toprule
$[g]$&FS&   1A&   2A&   2B&   2C&   4A\\
	\midrule
$[g^2]$&&	1A&	1A&	1A&	1A&	2A\\
	\midrule
$\chi_1$&$+$&   $1$&   $1$&   $1$&   $1$&   $1$\\
$\chi_2$&$+$&   $1$&   $1$&   $-1$&   $-1$&   $1$\\
$\chi_3$&$+$&   $1$&   $1$&   $-1$&   $1$&   $-1$\\
$\chi_4$&$+$&   $1$&   $1$&   $1$&   $-1$&   $-1$\\
$\chi_5$&$+$&   $2$&   $-2$&   $0$&   $0$&   $0$\\\bottomrule
\end{tabular}
\end{center}
\end{table}

\begin{table}
\begin{center}
\caption{Character table of ${G}^{X}\simeq \Dih_6$, $X=A_8^3$}\label{tab:chars:irr:9}
\smallskip
\begin{tabular}{c|c|rrrrrr}\toprule
$[g]$&FS&   1A&   2A&   2B&   2C&   3A&    6A\\
	\midrule
$[g^2]$&&	1A&	1A&	1A&	1A&	3A&	3A\\
$[g^3]$&&	1A&	2A&	2B&	2C&	1A&	2A\\
	\midrule
$\chi_1$&$+$&   $1$&   $1$&   $1$&   $1$&   $1$&   $1$\\
$\chi_2$&$+$&   $1$&   $1$&   $-1$&   $-1$&   $1$&   $1$\\
$\chi_3$&$+$&   $2$&   $2$&   $0$&   $0$&   $-1$&   $-1$\\
$\chi_4$&$+$&   $1$&   $-1$&   $-1$&   $1$&   $1$&   $-1$\\
$\chi_5$&$+$&   $1$&   $-1$&   $1$&   $-1$&   $1$&   $-1$\\
$\chi_6$&$+$&   $2$&   $-2$&   $0$&   $0$&   $-1$&   $1$\\\bottomrule
\end{tabular}
\end{center}
\end{table}

\begin{table}
\begin{center}
\caption{Character table of ${G}^{X}\simeq 4$, for $X\in\{A_9^2D_{6},A_{12}^2\}$}\label{tab:chars:irr:13}
\smallskip
\begin{tabular}{c|c|rrrr}\toprule
$[g]$&FS&   1A&   2A&   4A&   4B\\
	\midrule
$[g^2]$&&	1A&	1A&	2A&	2A\\
	\midrule
$\chi_1$&$+$&   $1$&   $1$&   $1$&   $1$\\
$\chi_2$&$+$&   $1$&   $1$&   $-1$&   $-1$\\
$\chi_3$&$\circ$&   $1$&   $-1$&   $a_1$&   $\overline{a_1}$\\
$\chi_4$&$\circ$&   $1$&   $-1$&   $\overline{a_1}$&   $a_1$\\\bottomrule
\end{tabular}
\end{center}
\end{table}

\begin{table}
\begin{center}
\caption{Character table of ${G}^{X}\simeq 	\PGL_2(3)\simeq \Sym_4$, $X=D_{6}^4$}\label{tab:chars:irr:10+5}
\smallskip
\begin{tabular}{c|c|rrrrrrrr}\toprule
$[g]$	&FS&   1A& 2A&   3A&  2B&    4A\\
	\midrule
$[g^2]$ 	&&1A&	1A&	3A&	1A&2A\\
$[g^3]$	&&1A&	2A&	1A&	2B&	4A\\
	\midrule
${\chi}_{1}$&$+$&   $1$&   $1$&   $1$&  $1$&   $1$\\
${\chi}_{2}$&$+$&   $1$&   $1$&   $1$&   $-1$&   $-1$\\
${\chi}_{3}$&$+$&   $2$&   $2$&   $-1$&   $0$&   $0$\\
${\chi}_{4}$&$+$&   $3$&   $-1$&   $0$&   $1$&   $-1$\\
${\chi}_{5}$&$+$&   $3$&   $-1$&   $0$&   $-1$&   $1$\\
	\bottomrule
\end{tabular}
\end{center}
\end{table}

\begin{table}
\centering
\caption{Character table of ${G}^{X}\simeq 2$, for $X\in\{A_{11}D_7E_6, A_{15}D_9, A_{17}E_7, A_{24}, D_{10}E_7^2, D_{12}^2\}$}\label{tab:chars:irr:25}
\smallskip
\begin{tabular}{c|c|rr}\toprule
$[g]$&FS&   1A&   2A\\
	\midrule
$[g^2]$&&	1A&	1A\\
	\midrule
$\chi_1$&$+$&   $1$&   $1$\\
$\chi_2$&$+$&   $1$&   $-1$\\\bottomrule
\end{tabular}
\end{table}

\begin{table}
\begin{center}
\caption{Character table of ${G}^{X}\simeq \Sym_3$, $X\in\{D_8^3, E_8^3\}$}\label{tab:chars:irr:14+7}
\smallskip
\begin{tabular}{c|c|rrrrrr}\toprule
$[g]$&FS&   1A&   2A&   3A\\
	\midrule
$[g^2]$&&	1A&	1A&	3A\\
$[g^3]$&&	1A&	2A&	1A\\
	\midrule
$\chi_1$&$+$&   $1$&   $1$&   $1$\\
$\chi_2$&$+$&   $1$&   $-1$&  $1$\\
$\chi_3$&$+$&   $2$&   $0$&   $-1$\\
\bottomrule
\end{tabular}
\end{center}
\end{table}

\clearpage

\subsection{Twisted Euler Characters}\label{sec:chars:eul}

In this section we reproduce certain characters---the {\em twisted Euler characters}---which are attached to each group $G^X$, via its action on the root system $X$. (Their construction is described in detail in \S2.4 of \cite{MUM}.)

To interpret the tables, write $X_A$ for the (possibly empty) union of type A components of $X$, and interpret $X_D$ and $X_E$ similarly, so that if $m=m^X$ Then $X=A_{m-1}^d$ for some $d$, and $X=X_A\cup X_D\cup X_E$, for example. Then $g\mapsto\bar{\chi}^{X_A}_g$ denotes the character of the permutation representation attached to the action of $\bar{G}^X$ on the simple components of $X_A$. The characters $g\mapsto\bar\chi^{X_D}_g$ and $g\mapsto \bar\chi^{X_E}_g$ are defined similarly. The characters $\chi^{X_A}_g$, $\chi^{X_D}_g$, $\chi^{X_E}_g$ and $\check\chi^{X_D}_g$ incorporate outer automorphisms of simple root systems induced by the action $G^X$ on $X$. We refer to \S2.4 of \cite{MUM} for full details of the construction. For the purposes of this work, it suffices to have the explicit descriptions in the tables in this section. The twisted Euler characters presented here will be used to specify the umbral shadow functions in \S\ref{sec:mts:shd}.

The twisted Euler character tables also attach integers $n_g$ and $h_g$ to each $g\in G^X$. By definition, $n_g$ is the order of the image of $g\in G^X$ in $\bar{G}^X$ (cf. \S\ref{sec:gps:cnst}). 
The integer $h_g$ may be defined by setting $h_g:=N_g/n_g$ where $N_g$ is the product of the shortest and longest cycle lengths appearing in the cycle shape attached to $g$ by the action of $G^X$ on a (suitable) set of simple roots for $X$.

\begin{table}[ht]
\begin{center}
\caption{Twisted Euler characters at $\ll=2$, $X=A_1^{24}$}\label{tab:chars:eul:2}
\smallskip
\begin{tabular}{l@{ }|@{ }r@{ }r@{ }r@{ }r@{ }r@{ }r@{ }r@{ }r@{ }r@{ }r@{ }r@{ }r@{ }r@{ }r@{ }r}
\toprule
$[g]$	&1A	&2A	&2B	&3A	&3B	&4A	&4B	&4C	&5A	&6A	&6B	\\
	\midrule
$n_g|h_g$&$1|1$&$2|1$&${2|2}$&$3|1$&$3|3$&$4|2$&$4|1$&${4|4}$&$5|1$&$6|1$&$6|6$&\\
	\midrule
$\bar{\chi}^{X_A}_{g}$&     
	$24$&   $8$&   		$0$&   		$6$&   		$0$&   		$0$&   		$4$&  		$0$&   		$4$&   		$2$&   		$0$&  \\
	\midrule
\midrule
$[g]$	& 7AB	&8A	&10A	&11A&12A	&12B	&14AB	&15AB	&21AB	&23AB	\\
	\midrule
$n_g|h_g$&$7|1$&$8|1$&$10|2$&$11|1$&$12|2$&$12|12$&$14|1$&$15|1$&$21|3$&$23|1$\\
	\midrule
$\bar{\chi}^{X_A}_{g}$&
	$3$&   		$2$&   		$0$&   		$2$&   		$0$&   		$0$& 		$1$&   		$1$&   		$0$&   		$1$\\
\bottomrule
\end{tabular}
\end{center}
\end{table}

\begin{table}[ht]
\begin{center}
\caption{Twisted Euler characters at $\ll=3$, $X=A_2^{12}$}\label{tab:chars:eul:3}
\smallskip
\begin{tabular}{l@{ }|@{\;}r@{\,}r@{\,}r@{\,}r@{\,}r@{\,}r@{\,}r@{\,}r@{\,}r@{\,}r@{\,}r@{\,}r@{\,}r@{\,}r@{\,}r@{\,}r@{\,}r@{\,}r@{\,}r@{\,}r@{\,}r@{\,}r@{\,}r@{\,}r@{\,}r@{\,}r}\toprule
$[g]$&   	1A&   		2A&   		4A&   		2B&   		2C&   		3A&   		6A&   		3B&   		6B&   		4B& 	  		4C&   		5A&   		10A&   		12A&   		6C&   		6D&   		8AB&   	 	8CD&   		20AB&   		11AB&   		22AB\\ 
	\midrule
$n_g|h_g$&$1|1$&$1|4$&${2|8}$&$2|1$&$2|2$&$3|1$&$3|4$&${3|3}$&${3|12}$&$4|2$&$4|1$&$5|1$&$5|4$&$6|24$&$6|1$&$6|2$&$8|4$&$8|1$&${10|8}$&$11|1$&$11|4$\\
	\midrule
$\bar{\chi}^{X_A}_{g}$&   $12$&   $12$&   		$0$&   		$4$&   		$4$&   		$3$&   		$3$&  		$0$&   		$0$&   		$0$&   		$4$&   		$2$&   		$2$&   		$0$&   		$1$&   		$1$&   		$0$&   		$2$&   		$0$&   		$1$&   		$1$\\
$\chi^{X_A}_{g}$&   $12$&   $-12$&   		$0$&   		$4$&   		$-4$&   		$3$&   		$-3$&   		$0$&   		$0$&   		$0$&  		$0$&   		$2$&   		$-2$&   		$0$&   		$1$&   		$-1$&   		$0$&   		$0$&   		$0$&   		$1$&   		$-1$\\
	\bottomrule
\end{tabular}
\smallskip
\end{center}
\end{table}

\begin{table}
\begin{center}
\caption{Twisted Euler characters at $\ll=4$, $X=A_3^8$}\label{tab:chars:eul:4}
\smallskip
\begin{tabular}{l@{\, }|@{\;}r@{\, }r@{\, }r@{\, }r@{\, }r@{\, }r@{\, }r@{\, }r@{\, }r@{\, }r@{\, }r@{\, }r@{\, }r}\toprule
$[g]$&   		1A&   2A&   	2B&   	4A&			4B&			2C&   	3A&   	6A&   		6BC&   	8A&   	4C&   	7AB&   	14AB\\ 
	\midrule
$n_g|h_g$&$1|1$& $1|2$&	$2|2$&	$2|4$&			${4|{4}}$&	$2|1$& 	$3|1$& 	$3|2$&		$6|2$&	${4|{8}}$&		$4|1$&  	$7|1$&	$7|2$\\	
	\midrule
$\bar{\chi}^{X_A}_g$&   $8$&$8$&	$0$& 	$0$& 		$0$&		$4$&  	$2$& 	$2$&  		$0$& 	$0$& 	$2$& 	$1$& 	$1$\\
$\chi^{X_A}_g$&   $8$&$-8$&	$0$&	$0$& 		$0$&		$0$&  	$2$& 	$-2$& 		$0$& 	$0$& 	$0$& 	$1$& 	$-1$\\
\bottomrule
\end{tabular}
\smallskip
\end{center}

\end{table}

\begin{table}
\begin{center}
\caption{Twisted Euler characters at $\ll=5$, $X=A_4^6$}\label{tab:chars:eul:5}
\smallskip
\begin{tabular}{l@{\, }|@{\;}r@{\, }r@{\, }r@{\, }r@{\, }r@{\, }r@{\, }r@{\, }r@{\, }r@{\, }r@{\, }r@{\, }r@{\, }r}\toprule
$[g]$&   		1A&		2A&   	2B&   	2C&			3A&			6A&   	5A&   	10A&   		4AB&   	4CD&	12AB\\ 
	\midrule
$n_g|h_g$&		$1|1$&	$1|4$&	$2|2$&	$2|1$&		$3|3$&		$3|12$&	$5|1$&	$5|4$&		$2|8$&		$4|1$&	$6|24$	\\	
	\midrule
$\bar{\chi}^{X_A}_{g}$&   $6$&	$6$&	$2$& 	$2$& 		$0$&		$0$&  	$1$& 	$1$&  		$0$& 	$2$& 	$0$ 	\\
$\chi^{X_A}_{g}$&   $6$&	$-6$&	$-2$&	$2$& 		$0$&		$0$&  	$1$& 	$-1$& 		$0$&	$0$& 	$0$ 	\\
\bottomrule
\end{tabular}
\smallskip
\end{center}
\end{table}

\begin{table}
\begin{center}
\caption{Twisted Euler characters at $\ll=6$, $X=A_5^4D_4$}\label{tab:chars:eul:6}
\smallskip
\begin{tabular}{l|rrrrrrr}\toprule
$[g]$&   		1A&		2A&   	2B&   	4A&			3A&			6A&   	8AB\\ 
	\midrule
$n_g|h_g$&		$1|1$&	$1|2$&	$2|1$&	$2|2$&		$3|1$&		$3|2$&	$4|2$\\	
	\midrule
$\bar{\chi}^{X_A}_g$&
			$4$&	$4$&	$2$&	$0$&	$1$&	$1$&	$0$\\
$\chi^{X_A}_g$&
			$4$&	$-4$&	$0$&	$0$&	$1$&	$-1$&	$0$\\
			\midrule
$\bar{\chi}^{X_D}_g$&
			$1$&	$1$&	$1$&	$1$&	$1$&	$1$&	$1$\\
${\chi}^{X_D}_g$&
			$1$&	$1$&	$-1$&	$1$&	$1$&	$1$&	$-1$\\
$\check{\chi}^{X_D}_g$&
			$2$&	$2$&	$0$&	$2$&	$-1$&	$-1$&	$0$\\
	\bottomrule
\end{tabular}
\smallskip
\end{center}

\end{table}

\begin{table}
\begin{center}
\caption{Twisted Euler characters at $\ll=6+3$, $X=D_4^6$}\label{tab:chars:eul:6+3}
\begin{tabular}{l@{\;}|@{\;}r@{\;}r@{\;}r@{\;}r@{\;}r@{\;}r@{\;}r@{\;}r@{\;}r@{\;}r@{\;}r@{\;}r@{\;}r@{\;}r@{\;}r}\toprule
$[g]$&   		1A&		3A&   	2A&   	6A&			3B&			6C&   	4A&	12A&	5A&	15AB&	2B&	2C&	4B&	6B&	6C\\ 
	\midrule
$n_g|h_g$&		$1|1$&	$1|3$&	$2|1$&	$2|3$&		$3|1$&		$3|3$&	$4|2$&	$4|6$&	$5|1$&	$5|3$&	$2|1$&	$2|2$&	$4|1$&	$6|1$&	$6|6$\\	
	\midrule
$\bar{\chi}^{X_D}_g$&
			$6$&	$6$&	$2$&	$2$&	$3$&	$0$&	$0$&	$0$&	$1$&	$1$&	$4$&	$0$&	$2$&	$1$&	$0$\\
${\chi}^{X_D}_g$&
			$6$&	$6$&	$2$&	$2$&	$3$&	$0$&	$0$&	$0$&	$1$&	$1$&	$-4$&	$0$&	$-2$&	$-1$&	$0$\\
$\check{\chi}^{X_D}_g$&
			$12$&	$-6$&	$4$&	$-2$&	$0$&	$0$&	$0$&	$0$&	$2$&	$-1$&	$0$&	$0$&	$0$&	$0$&	$0$\\
	\bottomrule
\end{tabular}
\smallskip
\end{center}

\end{table}

\begin{table}
\begin{center}
\caption{Twisted Euler characters at $\ll=7$, $X=A_6^4$}\label{tab:chars:eul:7}
\begin{tabular}{l|rrrrr}\toprule
$[g]$&   1A&   2A&   4A&   3AB&   6AB\\ 
	\midrule
$n_g|h_g$&		$1|1$&	$1|4$&	$2|8$&	$3|1$&	$3|4$\\
	\midrule
$\bar{\chi}^{X_A}_{g}$&	4&	4&	0&	1&	1\\
$\chi^{X_A}_{g}$&	4&	-4&	0&	1&	-1\\
\bottomrule
\end{tabular}
\end{center}

\end{table}

\begin{table}
\begin{center}
\caption{Twisted Euler characters at $\ll=8$, $X=A_7^2D_5^2$}\label{tab:chars:eul:8}
\begin{tabular}{l|rrrrr}\toprule
	$[g]$&	1A&	2A&	2B&2C&4A\\
		\midrule
$n_g|h_g$&		$1|1$&$1|2$&${2|1}$&$2|1$&${2|4}$\\	
		\midrule
$\bar{\chi}^{X_A}_g$&	2&2&0&2&0	\\
$\chi^{X_A}_g$&	2&-2&0&0&0	\\
		\midrule
$\bar{\chi}^{X_D}_g$&	2&2&2&0&0	\\
$\chi^{X_D}_g$&	2&-2&0&0&0	\\
\bottomrule
\end{tabular}
\end{center}

\end{table}

\begin{table}
\begin{center}
\caption{Twisted Euler characters at $\ll=9$, $X=A_8^3$}\label{tab:chars:eul:9}
\begin{tabular}{l|rrrrrr}\toprule
	$[g]$&	1A&	2A&	2B&2C&3A&6A\\
		\midrule
$n_g|h_g$&		$1|1$&$1|4$&${2|1}$&$2|2$&$3|3$&$3|12$\\	
		\midrule
	$\bar{\chi}^{X_A}_{g}$	&3&3&1&1&0&0\\
	$\chi^{X_A}_{g}$		&3&-3&1&-1&0&0\\
	\bottomrule
\end{tabular}
\end{center}
\end{table}

\begin{table}
\begin{center}
\caption{Twisted Euler characters at $\ll=10$, $X=A_9^2D_6$}\label{tab:chars:eul:10}
\begin{tabular}{l|rrr}\toprule
	$[g]$&	1A&	2A&	4AB\\
		\midrule
$n_g|h_g$&		$1|1$&$1|2$&${2|2}$\\	
		\midrule
$\bar{\chi}^{X_A}_g$&2&2&0\\
$\chi^{X_A}_g$&2&-2&0\\
	\midrule
$\bar{\chi}^{X_D}_g$&$1$&$1$&$1$\\
$\chi^{X_D}_g$&$1$&$1$&$-1$\\
	\bottomrule
\end{tabular}
\end{center}

\end{table}

\begin{table}
\begin{center}
\caption{Twisted Euler characters at $\ll=10+5$, $X=D_6^4$}\label{tab:chars:eul:10+5}
\begin{tabular}{l|rrrrr}\toprule
	$[g]$&	1A&	2A&	3A&2B&4A\\
		\midrule
	$n_g|h_g$&$1|1$&$2|2$&$3|1$&$2|1$&$4|4$\\
		\midrule
$\bar{\chi}^{X_D}_g$&$4$&$0$&$1$&$2$&$0$\\
$\chi^{X_D}_g$&$4$&$0$&$1$&$-2$&$0$\\
	\bottomrule
\end{tabular}
\end{center}
\end{table}

\begin{table}
\begin{center}
\caption{Twisted Euler characters at $\ll=12$, $X=A_{11}D_7E_6$}\label{tab:chars:eul:12}
\begin{tabular}{l|rr}\toprule
	$[g]$&	1A&	2A\\
		\midrule
	$n_g|h_g$&		$1|1$&$1|2$\\	
		\midrule
$\bar{\chi}^{X_A}_g$&$1$&$1$\\
${\chi}^{X_A}_g$&$1$&$-1$\\
	\midrule
$\bar{\chi}^{X_D}_g$&$1$&$1$\\
${\chi}^{X_D}_g$&$1$&$-1$\\
	\midrule
$\bar{\chi}^{X_E}_g$&$1$&$1$\\
${\chi}^{X_E}_g$&$1$&$-1$\\
\bottomrule
\end{tabular}
\end{center}
\end{table}

\begin{table}
\begin{center}
\caption{Twisted Euler characters at $\ll=12+4$, $X=E_6^4$}\label{tab:chars:eul:12+4}
\begin{tabular}{l|rrrrrrr}\toprule
$[g]$&   		1A&		2A&   	2B&   	4A&			3A&			6A&   	8AB\\ 
	\midrule
$n_g|h_g$&	$1|1$&$1|2$&$2|1$&$2|4$&$3|1$&$3|2$&$4|8$\\
	\midrule
$\bar{\chi}^{X_E}_g$&
			$4$&	$4$&	$2$&	$0$&	$1$&	$1$&	$0$\\
$\chi^{X_E}_g$&
			$4$&	$-4$&	$0$&	$0$&	$1$&	$-1$&	$0$\\
\bottomrule
\end{tabular}
\smallskip
\end{center}
\end{table}

\begin{table}
\begin{center}
\caption{Twisted Euler characters at $\ll=13$, $X=A_{12}^2$}\label{tab:chars:eul:13}
\begin{tabular}{r|rrr}\toprule
	$[g]$&	1A&	2A&	4AB\\
		\midrule
$n_g|h_g$&		$1|1$&$1|4$&${2|8}$\\	
		\midrule
	$\bar{\chi}^{X_A}_{g}$	&2&2&0\\
	$\chi^{X_A}_{g}$	&2&-2&0\\
\bottomrule
\end{tabular}
\end{center}

\end{table}

\clearpage

\begin{table}
\begin{center}
\caption{Twisted Euler characters at $\ll=14+7$, $X=D_8^3$}\label{tab:chars:eul:14+7}
\begin{tabular}{l|rrr}\toprule
	$[g]$&	1A&	2A&	3A\\
		\midrule
	$n_g|h_g$&$1|1$&$2|1$&$3|3$\\
	\midrule
	$\bar{\chi}^{X_D}_{g}$	&3&1&0\\
	$\chi^{X_D}_{g}$	&3&1&0\\
	\bottomrule
\end{tabular}
\end{center}
\end{table}

\begin{table}
\begin{center}
\caption{Twisted Euler characters at $\ll=16$, $X=A_{15}D_9$}\label{tab:chars:eul:16}
\begin{tabular}{l|rr}\toprule
	$[g]$&	1A&	2A\\
		\midrule
	$n_g|h_g$&$1|1$&$1|2$\\
	\midrule
$\bar{\chi}^{X_A}_g$&$1$&$1$\\
${\chi}^{X_A}_g$&$1$&$-1$\\
	\midrule
$\bar{\chi}^{X_D}_g$&$1$&$1$\\
${\chi}^{X_D}_g$&$1$&$-1$\\
\bottomrule
\end{tabular}
\end{center}
\end{table}

\begin{table}
\begin{center}
\caption{Twisted Euler characters at $\ll=18$, $X=A_{17}E_7$}\label{tab:chars:eul:18}
\begin{tabular}{l|rr}\toprule
	$[g]$&	1A&	2A\\
		\midrule
	$n_g|h_g$&$1|1$&$1|2$\\
	\midrule
$\bar{\chi}^{X_A}_g$&$1$&$1$\\
${\chi}^{X_A}_g$&$1$&$-1$\\
	\midrule
$\bar{\chi}^{X_E}_g$&$1$&$1$\\
\bottomrule
\end{tabular}
\end{center}
\end{table}

\begin{table}
\begin{center}
\caption{Twisted Euler characters at $\ll=18+9$, $X=D_{10}E_7^2$}\label{tab:chars:eul:18+9}
\begin{tabular}{l|rr}\toprule
	$[g]$&	1A&	2A\\
		\midrule
	$n_g|h_g$&$1|1$&$2|1$\\
	\midrule
$\bar{\chi}^{X_D}_g$&$1$&$1$\\
${\chi}^{X_D}_g$&$1$&$-1$\\
	\midrule
$\bar{\chi}^{X_E}_g$&$2$&$0$\\
\bottomrule
\end{tabular}
\end{center}
\end{table}

\begin{table}
\begin{center}
\caption{Twisted Euler characters at $\ll=22+11$, $X=D_{12}^2$}\label{tab:chars:eul:22+11}
\begin{tabular}{l|rr}\toprule
	$[g]$&	1A&	2A\\
		\midrule
		$n_g|h_g$&$1|1$&$2|2$\\
		\midrule
	$\bar{\chi}^{X_D}_{g}$	&2&0\\
	$\chi^{X_D}_{g}$	&2&0\\
\bottomrule
\end{tabular}
\end{center}
\end{table}

\begin{table}
\begin{center}
\caption{Twisted Euler characters at $\ll=25$, $X=A_{24}$}\label{tab:chars:eul:25}
\begin{tabular}{l|rr}\toprule
	$[g]$&	1A&	2A\\
		\midrule
	$n_g|h_g$&	$1|1$&	$1|4$\\
	\midrule
$\bar{\chi}^{X_A}_g$&$1$&$1$\\
${\chi}^{X_A}_g$&$1$&$-1$\\
\bottomrule
\end{tabular}
\end{center}
\end{table}

\begin{table}
\begin{center}
\caption{Twisted Euler characters at $\ll=30+6,10,15$, $X=E_8^3$}\label{tab:chars:eul:30+6,10,15}
\begin{tabular}{l|rrr}\toprule
$[g]$&   		1A&		2A&   	3A\\ 
	\midrule
	$n_g|h_g$&	$1|1$&	$2|1$&	$3|3$\\
	\midrule
$\bar{\chi}^{X_E}_g$&
			$3$&	$1$&	$0$\\
\bottomrule
\end{tabular}
\smallskip
\end{center}
\end{table}

\clearpage

\section{The Umbral McKay-Thompson Series}\label{sec:mts}

In this section we describe the umbral McKay-Thompson series in complete detail. In particular, we present explicit formulas for all the McKay-Thompson series attached to elements of the umbral groups by umbral moonshine in \S\ref{sec:exp}. Most of these expressions appeared first in \cite{UM,MUM}, but some appear for the first time in this work. 

In order to facilitate explicit formulations we recall certain standard functions in \S\ref{sec:mts:specfn}. We then, using the twisted Euler characters of \S\ref{sec:chars:eul}, explicitly describe the shadow functions of umbral moonshine in \S\ref{sec:mts:shd}. The umbral McKay--Thompson series defined in \S\ref{sec:exp} may also be described in terms of Rademacher sums, according to the results of \cite{mumcor}. We present this description in \S\ref{sec:mts:rad}.

\subsection{Special Functions}\label{sec:mts:specfn}
Throughout this section we assume $q:=e^{2\pi i \tau}$, and $u:=e^{2\pi i z},$ where $\tau,z\in \CC$ with $\im ~\tau>0$.
The Dedekind eta function is $\eta(\tau):=q^{1/24}\prod_{n>0}(1-q^n)$, where . Write $\Lambda_M(\t)$ for the function 
\begin{gather*}
\Lambda_M(\t) := M q\frac{{\rm d}}{{\rm d}q} \left(\log \frac{\h(M\t)}{\h(\t)}\right)
	=\frac{M(M-1)}{24}+M\sum_{k>0}\sum_{d|k}d
	\left(q^k-Mq^{Mk}\right),
\end{gather*}
which is a modular form of weight two for $\G_0(N)$ if $M|N$. 

Define the Jacobi theta function $\theta_1(\tau,z)$ by setting
\begin{gather}
	\theta_1(\tau,z)
	:=i q^{1/8}u^{-1/2}\sum_{n\in \ZZ}(-1)^nu^nq^{n(n-1)/2}.
\end{gather}
According to the Jacobi triple product identity we have
\begin{gather}
	\theta_1(\tau,z)
	=-i q^{1/8}
	u^{1/2}\prod_{n>0}(1-u^{-1}q^{n-1})(1-uq^n)(1-q^n).
\end{gather}
The other Jacobi theta functions are
\begin{gather}
\begin{split}
	\th_2(\t,z)
	&:=  q^{1/8} u^{1/2} \prod_{n>0} (1+u^{-1} q^{n-1})(1+u q^n) (1-q^n) ,\\
	\th_3(\t,z)
	&:=  \prod_{n>0}  (1+u^{-1} q^{n-1/2})(1+u q^{n-1/2})(1-q^n) ,\\
\th_4(\t,z) 
	&:=  \prod_{n>0} (1-u^{-1} q^{n-1/2})(1-uq^{n-1/2}) (1-q^n) .
	\end{split}
\end{gather}

Define $\Psi_{1,1}$ and $\Psi_{1,-1/2}$ by setting
\begin{gather}
	\begin{split}
	\Psi_{1,1}(\tau,z)&:=-i\frac{\theta_1(\t,2z)\eta(\t)^3}{\theta_1(\tau,z)^2},\\
	\Psi_{1,-1/2}(\tau,z)&:=-i\frac{\eta(\tau)^3}{\theta_1(\tau,z)}.
	\end{split}
\end{gather}
These are meromorphic Jacobi forms of weight one, with indexes $1$ and $-1/2$, respectively. Here, the term meromorphic refers to the presence of simple poles in the functions $z\mapsto \Psi_{1,*}(\tau,z)$, 
for fixed $\tau\in\HH$, at lattice points $z\in \ZZ\tau+\ZZ$. (Cf. \S8 of \cite{DMZ}.)

From \S5 of \cite{eichler_zagier} we recall the {\em index $m$ theta functions}, for $m\in\ZZ$, defined by  setting
\begin{gather}\label{eqn:specfun-thmr}
	\theta_{m,r}(\tau,z):=\sum_{k\in\ZZ}u^{2mk+r}q^{(2mk+r)^2/4m},
\end{gather}
where $r\in \ZZ$. Evidently, $\theta_{m,r}$ only depends on $r\mod 2m$.
We set $S_{m,r}(\tau):=\frac{1}{2\pi i}\left.\partial_z \theta_{m,r}(\tau,z)\right|_{z=0}$, so that
\begin{gather}\label{eqn:specfun-Smr}
	S_{m,r}(\tau)=\sum_{k\in\ZZ}(2mk+r)q^{(2mk+r)^2/4m}.
\end{gather}

For a $m$ a positive integer 
define 
\begin{gather}\label{eqn:mum0}
	\mu_{m,0}(\tau,z)=\sum_{k\in\ZZ} u^{2km}q^{mk^2}\frac{uq^k+1}{uq^k-1}=\frac{u+1}{u-1}+O(q),
\end{gather}
and observe that we recover $\Psi_{1,1}$ upon specializing (\ref{eqn:mum0}) to $m=1$.
Observe also that
\begin{gather}
	\mu_{m,0}(\tau,z+1/2)=\sum_{k\in\ZZ} u^{2km}q^{mk^2}\frac{uq^k-1}{uq^k+1}=\frac{u-1}{u+1}+O(q).
\end{gather}
Define the even and odd parts of $\mu_{m,0}$ by setting
\begin{gather}
	\mu_{m,0}^k(\tau,z):=\frac{1}{2}(\mu_{m,0}(\tau,z)+(-1)^k\mu_{m,0}(\tau,z+1/2))
\end{gather}
for $k\mod 2$.

For $m,r\in \ZZ+\frac12$ with $m>0$ define {\em half-integral index theta functions}
\begin{gather}
	\theta_{m,r}(\tau,z):=\sum_{k\in\ZZ}e(mk+r/2)u^{2mk+r}q^{(2mk+r)^2/4m},
\end{gather}
and define also $S_{m,r}(\tau):=\frac{1}{2\pi i}\partial_z\theta_{m,r}(\tau,z)|_{z=0}$, so that
\begin{gather}
	S_{m,r}(\tau)=\sum_{k\in\ZZ}e(mk+r/2)(2mk+r)q^{(2mk+r)^2/4m}.	
\end{gather}
As in the integral index case, $\theta_{m,r}$ depends only on $r\mod 2m$. We recover $-\theta_1$ upon specializing $\theta_{m,r}$ to $m=r=1/2$.

For $m\in \ZZ+1/2$, $m>0$, define
\begin{gather}
	\mu_{m,0}(\tau,z):=i\sum_{k\in \ZZ}(-1)^ku^{2mk+1/2}q^{mk^2+k/2}\frac{1}{1-uq^k}=\frac{-iu^{1/2}}{y-1}+O(q).
\end{gather}

Given $\alpha\in \QQ$ 
write $[\alpha]$ for the operator on $q$-series (in rational, possibility negative powers of $q$) that eliminates exponents not contained in $\ZZ+\alpha$, so that if $f=\sum_{\beta\in \QQ}c(\beta)q^\beta$ then 
\begin{gather}\label{eqn:exp:specfun-proj}
	[\alpha]f:=\sum_{n\in\ZZ}c(n+\alpha)q^{n+\alpha}
\end{gather}

\subsection{Shadows}\label{sec:mts:shd}

Let $X$ be a Niemeier root system and let $m=m^X$ be the Coxeter number of $X$. For $g\in G^X$ we define the associated {\em shadow function} $S^X_g=(S^X_{g,r})$ by setting
\begin{gather}
	S^X_g:=S^{X_A}_g+S^{X_D}_g+S^{X_E}_g
\end{gather}
where the $S^{X_A}_g$, \&c., are defined in the following way, in terms of the twisted Euler characters $\chi^{X_A}_g$, \&c. given in \S\ref{sec:chars:eul}, and the unary theta series $S_{m,r}$ (cf. (\ref{eqn:specfun-Smr})). 

Note that if $m=m^X$ then $S^X_{g,r}=S^X_{g,r+2m}=-S^X_{g,-r}$ for all $g\in G^X$, so we need specify the $S^{X_A}_{g,r}$, \&c., only for $0< r< m$.

If $X_A=\emptyset$ then $S^{X_A}_g:=0$. Otherwise, we define $S^{X_A}_{g,r}$ for $0<r<m$ 
by setting
\begin{gather}
	S^{X_A}_{g,r}:=
	\begin{cases}
		\chi^{X_A}_gS_{m,r}&\text{if $r=0\mod 2$,}\\
		\bar\chi^{X_A}_gS_{m,r}&\text{if $r=1\mod 2$.}
	\end{cases}
\end{gather}

If $X_D=\emptyset$ then $S^{X_D}_g:=0$. If $X_D\neq\emptyset$ then $m$ is even and $m\geq 6$. If $m=6$ then set
\begin{gather}
	S^{X_D}_{g,r}:=
	\begin{cases}
		0&\text{if $r=0\mod 2$,}\\
		\bar\chi^{X_D}_gS_{6,r}+\chi^{X_D}_gS_{6,6-r}&\text{if $r=1,5\mod 6$,}\\
		\check\chi^{X_D}_gS_{6,r}&\text{if $r=3\mod 6$.}
	\end{cases}
\end{gather}
If $m>6$ and $m=2\mod 4$ then set
\begin{gather}
	S^{X_D}_{g,r}:=
	\begin{cases}
		0&\text{if $r=0\mod 2$,}\\
		\bar\chi^{X_D}_gS_{m,r}+\chi^{X_D}_gS_{m,m-r}&\text{if $r=1\mod 2$.}
	\end{cases}
\end{gather}
If $m>6$ and $m=0\mod 4$ then set
\begin{gather}
	S^{X_D}_{g,r}:=
	\begin{cases}
		\chi^{X_D}_gS_{m,m-r}&\text{if $r=0\mod 2$,}\\
		\bar\chi^{X_D}_gS_{m,r}&\text{if $r=1\mod 2$.}
	\end{cases}
\end{gather}

If $X_E=\emptyset$ then $S^{X_E}_g:=0$. Otherwise, $m$ is $12$ or $18$ or $30$. In case $m=12$ define $S^{X_E}_{g,r}$ for $0<r<12$ by setting
\begin{gather}
	S^{X_E}_{g,r}=
	\begin{cases}
		\bar\chi^{X_E}_g(S_{12,1}+S_{12,7})&\text{if $r\in\{1,7\}$,}\\
		\bar\chi^{X_E}_g(S_{12,5}+S_{12,11})&\text{if $r\in\{ 5, 11\}$,}\\
		\chi^{X_E}_g(S_{12,4}+S_{12,8})&\text{if $r\in\{ 4, 8\}$,}\\
		0&\text{else.}
	\end{cases}
\end{gather}
In case $m=18$ define $S^{X_E}_{g,r}$ for $0<r<18$ by setting
\begin{gather}
	S^{X_E}_{g,r}=
	\begin{cases}
		\bar\chi^{X_E}_g(S_{18,r}+S_{18,18-r})&\text{if $r\in\{ 1,5,7,11,13,17\}$,}\\
		\bar\chi^{X_E}_gS_{18,9}&\text{if $r\in\{3,15\}$,}\\
		\bar\chi^{X_E}_g(S_{18,3}+S_{18,9}+S_{18,15})&\text{if $r=9$,}\\
		0&\text{else.}
	\end{cases}
\end{gather}
In case $m=30$ define $S^{X_E}_{g,r}$ for $0<r<30$ by setting
\begin{gather}
	S^{X_E}_{g,r}=
	\begin{cases}
		\bar\chi^{X_E}_g(S_{30,1}+S_{30,11}+S_{30,19}+S_{30,29})&\text{if $r\in\{1,11,19,29\}$,}\\
		\bar\chi^{X_E}_g(S_{30,7}+S_{30,13}+S_{30,17}+S_{30,23})&\text{if $r\in\{7,13,17,23\}$,}\\
		0&\text{else.}\\
	\end{cases}
\end{gather}

\subsection{Explicit Prescriptions}\label{sec:exp}

Here we give explicit expressions for all the umbral McKay-Thompson series $H^X_g$. Most of these appeared first in \cite{UM,MUM}. The expressions in \S\S\ref{sec:exp:8A3}, \ref{sec:exp:6A4}, \ref{sec:exp:4A6}, \ref{sec:exp:2A12} are taken from \cite{umvan4}. The expressions in \S\S\ref{sec:exp:4D6}, \ref{sec:exp:3D8}, \ref{sec:exp:2D12}, \ref{sec:exp:D24} are taken from \cite{umvan2}. The expressions for $H^X_g$ with $X=E_8^3$ appeared first in \cite{mod3e8}. The expression for $H^{(6+3)}_{2B,1}$ in \S\ref{sec:exp:6D4}, and the expressions for $H^{(12+4)}_{4A,r}$ and $H^{(12+4)}_{8AB,r}$ in \S\ref{sec:exp:4E6}, appear here for the first time.

The labels for conjugacy classes in $G^X$ are as in \S\ref{sec:gps:chars}. 

\subsubsection{$\ell=2$, $X=A_1^{24}$}\label{sec:exp:24A1}

We have $G^{(2)}=G^X\simeq M_{24}$ and $m^X=2$. So for $g\in M_{24}$, the associated umbral McKay-Thompson series 
$H^{(2)}_g=(H^{(2)}_{g,r})$ is a $4$-vector-valued function, with components indexed by $r\in \ZZ/4\ZZ$, satisfying $H^{(2)}_{g,r}=-H^{(2)}_{g,-r}$, and in particular, $H^{(2)}_{g,r}=0$ for $r=0\mod 2$. So it suffices to specify the $H^{(2)}_{g,1}$ explicitly.

Define $H^{(2)}_g=(H^{(2)}_{g,r})$ for $g=e$ by requiring that 
\begin{gather}
	-2\Psi_{1,1}(\tau,z)\varphi^{(2)}_1(\tau,z)
	=-24\mu_{2,0}(\tau,z)+\sum_{r\mod 4}H^{(2)}_{e,r}(\tau)\theta_{2,r}(\tau,z),
\end{gather}
where
\begin{gather}\label{eqn:exp:vp2}
	\varphi^{(2)}_1(\tau,z):=4\left(\frac{\th_2(\tau,z)^2}{\th_2(\tau,0)^2}+\frac{\th_3(\tau,z)^2}{\th_3(\tau,0)^2}+\frac{\th_4(\tau,z)^2}{\th_4(\tau,0)^2}\right).
\end{gather}
More generally, for $g\in G^{(2)}$ define
\begin{gather}\label{h_g_explicit}
H^{(2)}_{g,1}(\t) := \frac{\bar\chi^{(2)}_g}{24} H^{(2)}_{e,1}(\t) - {F}^{(2)}_g(\t)\frac{ 1}{S_{2,1}(\tau)},
\end{gather}
where $\bar\chi^{(2)}_g$ and $F^{(2)}_g$ are as specified in Table \ref{tab:24A1-FXg}. Note that $\bar\chi^{(2)}_g=\bar\chi^{X_A}_g$, the latter appearing in Table \ref{tab:chars:eul:2}. Also, $S_{2,1}(\tau)=\eta(\tau)^3$. 

\begin{table}[ht]
\centering  
\caption{Character Values and Weight Two Forms for $\ell=2$, $X=A_1^{24}$ \label{tab:24A1-FXg}}
\begin{tabular}{rrl}
$[g]$ & $\bar\chi^{(2)}_g$ & ${F}^{(2)}_g(\tau)$\\
\noalign{\vskip 1mm}
\hline
\noalign{\vskip 1mm}
$1A$ & $24$ & $0 $\\
$2A$ & ${8}$ & $16\Lambda_2(\tau)$\\
$2B $&0& $2\eta(\tau)^8\eta(2\tau)^{-4}$\\
$3A$&6& $6\Lambda_3(\tau)$\\
$3B$&0& $2\eta(\tau)^6\eta(3\tau)^{-2}$\\
$4A$&0& $2\eta(2\tau)^8\eta(4\tau)^{-4}$\\
$4B$&4& $4(-\Lambda_2(\tau)+\Lambda_4(\tau))$\\
$4C$&0&$ 2\eta(\tau)^4\eta(2\tau)^2\eta(4\tau)^{-2}$\\
$5A$ &4&$ 2\Lambda_5(\tau)$\\
$6A$&2&$ 2(-\Lambda_2(\tau)-\Lambda_3(\tau)+\Lambda_6(\tau))$\\
$6B$&0&$ 2\eta(\tau)^2\eta(2\tau)^2\eta(3\tau)^2\eta(6\tau)^{-2}$\\
$7AB$&3&$ \Lambda_7(\tau)$\\
$8A$&2&$ -\Lambda_4(\tau)+\Lambda_8(\tau)$\\
$10A$& 0&$ 2\eta(\tau)^3\eta(2\tau)\eta(5\tau)\eta(10\tau)^{-1}$\\
$11A$& 2&$ 2(\Lambda_{11}(\tau)-11\eta(\tau)^2\eta(11\tau)^2)/5$\\
$12A$& 0&$ 2\eta(\tau)^3\eta(4\tau)^2\eta(6\tau)^3\eta(2\tau)^{-1}\eta(3\tau)^{-1}\eta(12\tau)^{-2}$\\
$12B$& 0&$ 2\eta(\tau)^4\eta(4\tau)\eta(6\tau)\eta(2\tau)^{-1}\eta(12\tau)^{-1}$\\
$14AB$& 1&$ (-\Lambda_2(\tau)-\Lambda_7(\tau)+\Lambda_{14}(\tau)-{14}\eta(\tau)\eta(2\tau)\eta(7\tau)\eta(14\tau))/3$\\
$15AB$& 1&$ (-\Lambda_3(\tau)-\Lambda_5(\tau)+\Lambda_{15}(\tau)-{15}\eta(\tau)\eta(3\tau)\eta(5\tau)\eta(15\tau))/4$\\
$21AB$& 0&$ (7\eta(\tau)^3\eta(7\tau)^3\eta(3\tau)^{-1}\eta(21\tau)^{-1}-\eta(\tau)^6\eta(3\tau)^{-2})/3$\\
$23AB$& 1&$ (\Lambda_{23}(\tau)-{23}f_{23,a}(\tau)-{69}f_{23,b}(\tau))/11$\\
\end{tabular}
\end{table}

The functions $f_{23,a}$ and $f_{23,b}$ in Table \ref{tab:24A1-FXg} are cusp forms of weight two for $\G_0(23)$, defined by
\begin{gather}
\begin{split}
	f_{23,a}(\t)&:= \frac{\eta(\tau)^3\eta(23\tau)^3}{\eta(2\tau)\eta(46\tau)}
	+3\h(\t)^2\h(23\t)^2
	+4\eta(\tau)\eta(2\tau)\eta(23\tau)\eta(46\tau)
	+4\eta(2\tau)^2\eta(46\tau)^2, \\
	 f_{23,b}(\t)&:= \h(\t)^2\h(23\t)^2.
\end{split}
\end{gather}
Note that the definition of $F^{(2)}_{g}$ appearing here for $g\in 23A\cup 23B$ corrects errors in \cite{MR2985326,Cheng2011}.

\subsubsection{$\ell=3$, $X=A_2^{12}$}\label{sec:exp:12A2}

We have $G^{(3)}=G^X\simeq 2.M_{12}$ and $m^X=3$. So for $g\in 2.M_{12}$, the associated umbral McKay-Thompson series 
$H^{(3)}_g=(H^{(3)}_{g,r})$ is a $6$-vector-valued function, with components indexed by $r\in \ZZ/6\ZZ$, satisfying $H^{(3)}_{g,r}=-H^{(3)}_{g,-r}$, and in particular, $H^{(3)}_{g,r}=0$ for $r=0\mod 3$. So it suffices to specify the $H^{(3)}_{g,1}$ and $H^{(3)}_{g,2}$ explicitly.

Define $H^{(3)}_g=(H^{(3)}_{g,r})$ for $g=e$ by requiring that 
\begin{gather}
	-2\Psi_{1,1}(\tau,z)\varphi^{(3)}_1(\tau,z)
	=-12\mu_{3,0}(\tau,z)+\sum_{r\mod 6}H^{(3)}_{e,r}(\tau)\theta_{3,r}(\tau,z),
\end{gather}
where
\begin{gather}\label{eqn:exp:vp3}
	\varphi^{(3)}_1(\tau,z):=2\left(\frac{\th_3(\tau,z)^2}{\th_3(\tau,0)^2}\frac{\th_4(\tau,z)^2}{\th_4(\tau,0)^2}+\frac{\th_4(\tau,z)^2}{\th_4(\tau,0)^2}\frac{\th_2(\tau,z)^2}{\th_2(\tau,0)^2}+\frac{\th_2(\tau,z)^2}{\th_2(\tau,0)^2}\frac{\th_3(\tau,z)^2}{\th_3(\tau,0)^2}\right).
\end{gather}
More generally, for $g\in G^{(3)}$ define
\begin{gather}
	H^{(3)}_{g,1}(\tau):=\frac{\bar\chi^{(3)}_g}{12}H^{(3)}_{e,1}(\t)+\frac{1}{2}\left(F^{(3)}_g+F^{(3)}_{zg}\right)\frac{1}{S_{3,1}(\t)},\\
	H^{(3)}_{g,2}(\tau):=\frac{\chi^{(3)}_g}{12}H^{(3)}_{e,1}(\t)+\frac{1}{2}\left(F^{(3)}_g-F^{(3)}_{zg}\right)\frac{1}{S_{3,2}(\t)},
\end{gather}
where $\chi^{(3)}_g$ and $F^{(3)}_g$ are as specified in Table \ref{tab:12A2-FXg}, and $z$ is the non-trivial central element of $G^{(3)}$. The action of $g\mapsto zg$ on conjugacy classes can be read off Table \ref{tab:12A2-FXg}, for the horizontal lines indicate the sets $[g]\cup [zg]$. 

Note the eta product identities, $S_{3,1}(\tau)=\eta(2\t)^5/\eta(4\t)^2$, and $S_{3,2}(\t)=2\eta(\t)^2\eta(4\t)^2/\eta(2\t)$. Note also that $\bar\chi^{(3)}_g=\bar\chi^{X_A}_g$ and $\chi^{(3)}_g=\chi^{X_A}_g$, the latter appearing in Table \ref{tab:chars:eul:3}.

\begin{table}[ht]
\centering  
\caption{Character Values and Weight Two Forms for $\ell=3$, $X=A_2^{12}$ \label{tab:12A2-FXg}}
\begin{tabular}{rrrl}
$[g]$&$\bar\chi^{(3)}_g$&$\chi^{(3)}_g$& $F^{(3)}_g(\t)$\\
\noalign{\vskip 1mm}
\hline
\noalign{\vskip 1mm}
1A&12&12&0\\  
2A&12&$-12$&0\\  \noalign{\vskip 2pt}\hline\noalign{\vskip 2pt}
4A&0&0&$-2 {\h(\t)^4\h(2\t)^2}/{\h(4\t)^2}$\\   \noalign{\vskip 2pt}\hline\noalign{\vskip 2pt}
2B&4&4&$-16\L_2(\t)$\\   
2C&4&$-4$&$16\L_2(\t)- \frac{16}{3}\L_4(\t)$\\  \noalign{\vskip 2pt}\hline\noalign{\vskip 2pt}
3A&3&3&$-6\L_3(\t)$\\ 
6A&3&$-3$&$ -9\L_2(\t)-2\L_3(\t)+3\L_4(\t)+3\L_6(\t)-\L_{12}(\t)$\\   \noalign{\vskip 2pt}\hline\noalign{\vskip 2pt}
3B&0&0&$ 8\L_3(\t) - 2 \L_9(\t)+2\,{\h^6(\t)}/{\h^2(3\t)}$\\   
6B&0&0&
$ -2{\eta(\tau)^5\eta(3\tau)}/{\eta(2\tau)\eta(6\tau)}$
\\   \noalign{\vskip 2pt}\hline\noalign{\vskip 2pt}
4B&0&0&$-2 {\h(2\t)^8}/{\h(4\t)^4} $\\   \noalign{\vskip 2pt}\hline\noalign{\vskip 2pt}
4C&4&0&$-8\L_4(\t)/3$\\   \noalign{\vskip 2pt}\hline\noalign{\vskip 2pt}
5A&2&2&$-2\L_5(\t)$\\   
10A&2&$-2$&$\sum_{d|20} c_{10A}(d) \L_d(\t)+\tfrac{20}{3}\eta(2\t)^2\eta(10\t)^2$\\   \noalign{\vskip 2pt}\hline\noalign{\vskip 2pt}
12A&0&0&$-2 {\h(\t)\h(2\t)^5\h(3\t) }/{\h(4\t)^2 \h(6\t)} 
$\\   \noalign{\vskip 2pt}\hline\noalign{\vskip 2pt}
6C&1&1&$ 2(\L_2(\t)+\L_3(\t)-\L_6(\t))$\\ 
6D&1&$-1$&$ -5\L_2(\t)-2\L_3(\t)+\tfrac{5}{3}\L_4(\t)+3\L_6(\t)-\L_{12}(\t)$\\   \noalign{\vskip 2pt}\hline\noalign{\vskip 2pt}
8AB&0&0&$ -2 {\h(2\t)^4 \h(4\t)^2}/{\h(8\t)^2}$\\   \noalign{\vskip 2pt}\hline\noalign{\vskip 2pt}
8CD&2&0&$-2\L_2(\t)+\tfrac{5}{3}\L_4(\t)-\L_8(\t)$\\   \noalign{\vskip 2pt}\hline\noalign{\vskip 2pt}
20AB&0&0&$-2 {\h(2\t)^7 \h(5\t)}/{\h(\t) \h(4\t)^2 \h(10\t)}$\\\noalign{\vskip 2pt}\hline\noalign{\vskip 2pt}   
11AB&1&1&$-\frac{2}{5}\L_{11}(\t)-\frac{33}{5}\eta(\t)^2\eta(11\t)^2$\\   
22AB&1&$-1$& $\sum_{d|44} c_{g}(d) \L_d(\t) -\tfrac{11}{5}\sum_{d|4} c'_{g}(d)\eta(d\t)^2\eta(11d\t)^2+\tfrac{22}{3}f_{44}(\t)$
\end{tabular}
\end{table}

The function $f_{44}$ is the unique new cusp form of weight $2$ for $\Gamma_0(44)$, normalized so that $f_{44}(\t)=q+O(q^3)$ as $\Im(\t)\to\infty$. The coefficients $c_g(d)$ and $c'_g(d)$ for $g\in 10A\cup 22A\cup 22B$ are given by
\begin{gather}
	c_{10A}(2)=-5,\,
	c_{10A}(4)=-\frac53,\,
	c_{10A}(5)=-\frac23,\,
	c_{10A}(10)=1,\,
	c_{10A}(20)=-\frac13,\\
	c_{22AB}(2)=-\frac{11}{5},\,
	c_{22AB}(4)=\frac{11}{5},\,
	c_{22AB}(11)=-\frac{2}{15},\,
	c_{22AB}(22)=\frac15,\,
	c_{22AB}(44)=-\frac{1}{15},\\
	c'_{22AB}(1)=1,\,
	c'_{22AB}(2)=4,\,
	c'_{22AB}(4)=8.	
\end{gather}

\subsubsection{$\ell=4$, $X=A_3^{8}$}\label{sec:exp:8A3}

We have $m^X=4$, so the umbral McKay-Thompson series $H^{(4)}_g=(H^{(4)}_{g,r})$ associated to $g\in G^{(4)}$ is an $8$-vector-valued function, with components indexed by $r\in \ZZ/8\ZZ$. 

Define $H^{(4)}_g=(H^{(4)}_{g,r})$ for $g\in G^{(4)}$, $g\notin 4C$, by requiring that
\begin{gather}
	\psi^{(4)}_g(\tau,z)=
	-\chi^{(4)}_g\mu_{4,0}^0(\t,z)
	-\bar\chi^{(4)}_g\mu_{4,0}^1(\t,z)
	+\sum_{r\mod 8}H^{(4)}_{g,r}(\t)\theta_{4,r}(\tau,z),
\end{gather}
where $\chi^{(4)}_{g}:=\chi^{X_A}_g$ and $\bar\chi^{(4)}_{g}:=\bar\chi^{X_A}_g$ (cf. Table \ref{tab:chars:eul:4}), and the $\psi^{(4)}_g$ are meromorphic Jacobi forms of weight $1$ and index $4$ given explicitly in Table \ref{tab:8A3-psiXg}.
\begin{table}[ht]
\centering
\caption{Character Values and Meromorphic Jacobi Forms for $\ell=4$, $X=A_3^{8}$ \label{tab:8A3-psiXg}}
\begin{tabular}{rrrl}
$[g]$&$\chi^{(4)}_g$&$\bar\chi^{(4)}_g$&$\psi^{(4)}_g(\tau,z)$\\
\noalign{\vskip 1mm}
\hline
\noalign{\vskip 1mm}
1A&8&8&$2i{\th_1(\tau,2z)^3}{\th_1(\tau,z)^{-4}}\eta(\tau)^3$
	\\
2A&$-8$&$8$&$2i{\th_1(\tau,2z)^3}{\th_2(\tau,z)^{-4}}\eta(\tau)^3$
	\\
	\noalign{\vskip 2pt}\hline\noalign{\vskip 2pt}
2B&0&0& $-2i\th_1(\tau,2z)^3\th_1(\tau,z)^{-2}\th_2(\tau,z)^{-2}\eta(\tau)^3$
	\\
\noalign{\vskip 2pt}\hline\noalign{\vskip 2pt}
4A&0&0&$-2i{\th_1(\tau,2z)\th_2(\tau,2z)^2}{\th_2(2\tau,2z)^{-2}}\eta(2\tau)^2\eta(\tau)^{-1}$
	\\
\noalign{\vskip 2pt}\hline\noalign{\vskip 2pt}
4B&0&0& $-2i\th_1(2\tau,2z)\th_3(2\tau,2z)^2\th_4(2\tau,2z)\eta(2\tau)^2\eta(\tau)^{-2}\eta(4\tau)^{-2}$
	\\
\noalign{\vskip 2pt}\hline\noalign{\vskip 2pt}
2C&0&4&$2i\th_1(\tau,2z)\th_2(\tau,2z)^2\th_1(\tau,z)^{-2}\th_2(\tau,z)^{-2}\eta(\tau)^3$
	\\
\noalign{\vskip 2pt}\hline\noalign{\vskip 2pt}
3A&2&2&$2i\th_1(3\tau,6z)\th_1(\tau,z)^{-1}\th_1(3\tau,3z)^{-1}\eta(\tau)^3$
	\\
6A&$-2$&2&$-2i\th_1(3\tau,6z)\th_2(\tau,z)^{-1}\th_2(3\tau,3z)^{-1}\eta(\tau)^3$
	\\
\noalign{\vskip 2pt}\hline\noalign{\vskip 2pt}
6BC&0&0& cf. (\ref{eqn:exp:8A3-psiXg})
	\\
\noalign{\vskip 2pt}\hline\noalign{\vskip 2pt}
8A&0&0& $-2i\th_1(\tau,2z)\th_2(2\tau,4z)\th_2(4\tau,4z)^{-1}\eta(\tau)\eta(4\tau)\eta(2\tau)^{-1}$
	\\
\noalign{\vskip 2pt}\hline\noalign{\vskip 2pt}
4C&0&2& $2i\th_1(\tau,2z)\th_2(2\tau,4z)\th_1(2\tau,2z)^{-2}\eta(2\tau)^7\eta(\tau)^{-3}\eta(4\tau)^{-2}$
	\\
\noalign{\vskip 2pt}\hline\noalign{\vskip 2pt}
7AB&1&1& cf. (\ref{eqn:exp:8A3-psiXg})
	\\
14AB&$-1$&1& cf. (\ref{eqn:exp:8A3-psiXg})
\end{tabular}
\end{table}

\begin{gather}\label{eqn:exp:8A3-psiXg}	
	\begin{split}
	\psi^{(4)}_{6BC}&:=
	\left(
	\th_1(\tau,z+\tfrac13)\th_1(\tau,z+\tfrac16)-\th_1(\tau,z-\tfrac13)\th_1(\tau,z-\tfrac16)
	\right)
	\frac{-i\th_1(3\tau,6z)}{\th_1(3\tau,3z)\th_2(3\tau,3z)}
	\eta(3\tau)
	\\
	\psi^{(4)}_{7AB}&:=
	\left(
	\prod_{j=1}^3\th_1(\tau,2z+\tfrac{j^2}{7})\th_1(\tau,z-\tfrac{j^2}{7})
	+\prod_{j=1}^3\th_1(\tau,2z-\tfrac{j^2}{7})\th_1(\tau,z+\tfrac{j^2}{7})
	\right)\frac{-i}{\th_1(7\tau,7z)}\frac{\eta(7\tau)}{\eta(\tau)^4}
	\\
	\psi^{(4)}_{14AB}&:=
	\left(
	\prod_{j=1}^3\th_1(\tau,2z+\tfrac{j^2}{7})\th_2(\tau,z-\tfrac{j^2}{7})
	+\prod_{j=1}^3\th_1(\tau,2z-\tfrac{j^2}{7})\th_2(\tau,z+\tfrac{j^2}{7})
	\right)\frac{i}{\th_2(7\tau,7z)}\frac{\eta(7\tau)}{\eta(\tau)^4}
	\end{split}
\end{gather}

For use later on, note that $\psi^{(4)}_{1A}=-2\Psi_{1,1}\varphi^{(4)}_1$, where
\begin{gather}\label{eqn:exp:vp4}
	\varphi^{(4)}_1(\tau,z):=\frac{\th_1(\tau,2z)^2}{\th_1(\tau,z)^2}.
\end{gather}

\subsubsection{$\ell=5$, $X=A_4^{6}$}\label{sec:exp:6A4}

We have $m^X=5$, so the umbral McKay-Thompson series $H^{(5)}_g=(H^{(5)}_{g,r})$ associated to $g\in G^{(5)}$ is a $10$-vector-valued function, with components indexed by $r\in \ZZ/10\ZZ$. 

Define $H^{(5)}_g=(H^{(5)}_{g,r})$ for $g\in G^{(5)}$, $g\notin 5A\cup 10A$, by requiring that
\begin{gather}
	\psi^{(5)}_g(\tau,z)=
	-\chi^{(5)}_g\mu_{5,0}^0(\t,z)
	-\bar\chi^{(5)}_g\mu_{5,0}^1(\t,z)
	+\sum_{r\mod 10}H^{(5)}_{g,r}(\t)\theta_{5,r}(\tau,z),
\end{gather}
where $\chi^{(5)}_{g}:=\chi^{X_A}_g$ and $\bar\chi^{(5)}_{g}:=\bar\chi^{X_A}_g$ (cf. Table \ref{tab:chars:eul:5}), and the $\psi^{(5)}_g$ are meromorphic Jacobi forms of weight $1$ and index $5$ given explicitly in Table \ref{tab:6A4-psiXg}. 

\begin{table}[ht]
\centering
\caption{Character Values and Meromorphic Jacobi Forms for $\ell=5$, $X=A_4^{6}$ \label{tab:6A4-psiXg}}
\begin{tabular}{rrrl}
$[g]$&$\chi^{(5)}_g$&$\bar\chi^{(5)}_g$&$\psi^{(5)}_g(\tau,z)$\\
\noalign{\vskip 1mm}
\hline
\noalign{\vskip 1mm}
1A&6&6&$2i{\th_1(\tau,2z)\th_1(\tau,3z)}{\th_1(\tau,z)^{-3}}\eta(\tau)^3$
	\\
2A&$-6$&6&$-2i{\th_1(\tau,2z)\th_2(\tau,3z)}{\th_2(\tau,z)^{-3}}\eta(\tau)^3$
	\\
\noalign{\vskip 2pt}\hline\noalign{\vskip 2pt}
2B&$-2$&2&$-2i{\th_1(\tau,2z)\th_1(\tau,3z)}{\th_1(\tau,z)^{-1}\th_2(\tau,z)^{-2}}\eta(\tau)^3$
	\\
2C&$2$&2&$2i{\th_1(\tau,2z)\th_2(\tau,3z)}{\th_1(\tau,z)^{-2}\th_2(\tau,z)^{-1}}\eta(\tau)^3$
	\\
\noalign{\vskip 2pt}\hline\noalign{\vskip 2pt}
3A&0&0&$-2i{\th_1(\tau,2z)\th_1(\tau,3z)}{\th_1(3\tau,3z)^{-1}}\eta(3\tau)$
	\\
6A&0&0&$-2i{\th_1(\tau,2z)\th_2(\tau,3z)}{\th_2(3\tau,3z)^{-1}}\eta(3\tau)$
	\\
\noalign{\vskip 2pt}\hline\noalign{\vskip 2pt}
4AB&0&0& cf. (\ref{eqn:exp:6A4-psiXg}) 
	\\
\noalign{\vskip 2pt}\hline\noalign{\vskip 2pt}
4CD&0&2& cf. (\ref{eqn:exp:6A4-psiXg})
	\\
\noalign{\vskip 2pt}\hline\noalign{\vskip 2pt}
12AB&0&0& cf. (\ref{eqn:exp:6A4-psiXg})
\end{tabular}
\end{table}

\begin{gather}\label{eqn:exp:6A4-psiXg}	
\begin{split}
	\psi^{(5)}_{4AB}(\tau,z)&:=
	-i\th_2(\tau,2z)\frac{\th_1(\tau,z+\frac14)\th_1(\tau,3z+\frac14)-\th_1(\tau,z-\frac14)\th_1(\tau,3z-\frac14)}{\th_2(2\tau,2z)^2}\frac{\eta(2\tau)^2}{\eta(\tau)}
	\\
	\psi^{(5)}_{4CD}(\tau,z)&:=
	-i\th_2(\tau,2z)\frac{\th_1(\tau,z+\frac14)\th_1(\tau,3z-\frac14)+\th_1(\tau,z-\frac14)\th_1(\tau,3z+\frac14)}{\th_1(2\tau,2z)\th_2(2\tau,2z)}\frac{\eta(2\tau)^2}{\eta(\tau)}
	\\
	\psi^{(5)}_{12AB}(\tau,z)&:=
	i\frac{\th_2(\tau,2z)}{\th_2(6\tau,6z)}\left({\th_1(\tau,z+\tfrac{1}{12})\th_1(\tau,z+\tfrac{1}{4})\th_1(\tau,z+\tfrac{5}{12})\th_1(\tau,3z-\tfrac14)}
	\right. 
	\\
	&\quad\qquad\qquad\qquad -
	\left.{\th_1(\tau,z-\tfrac{1}{12})\th_1(\tau,z-\tfrac{1}{4})\th_1(\tau,z-\tfrac{5}{12})\th_1(\tau,3z+\tfrac14)}
	\right)\frac{\eta(6\tau)}{\eta(\tau)^3}
\end{split}
\end{gather}

For $g\in 5A$ 
use the formulas of \S\ref{sec:exp:A24} to define 
\begin{gather}
		H^{(5)}_{5A,r}(\tau)
		:=
			H^{(25)}_{1A,r}(\tau/5)-H^{(25)}_{1A,10-r}(\t/5)+
			H^{(25)}_{1A,10+r}(\tau/5)-H^{(25)}_{1A,20-r}(\t/5)+
			H^{(25)}_{1A,20+r}(\tau/5).
\end{gather}
For $g\in 10A$ set $H^{(5)}_{10A,r}(\tau):=-(-1)^rH^{(5)}_{5A,r}(\t)$.

For use later on we note that $\psi^{(5)}_{1A}=-2\Psi_{1,1}\varphi^{(5)}_1$, where
\begin{gather}\label{eqn:exp:vp5}
	\varphi^{(5)}_1(\tau,z):=\frac{\th_1(\tau,3z)}{\th_1(\tau,z)}. 
\end{gather}

\subsubsection{$\ell=6$, $X=A_5^{4}D_4$}\label{sec:exp:4A5D4}

We have $m^X=6$, so the umbral McKay-Thompson series $H^{(6)}_g=(H^{(6)}_{g,r})$ associated to $g\in G^{(6)}$ is a $12$-vector-valued function with components indexed by $r\in \ZZ/12\ZZ$. We have $H^{(6)}_{g,r}=-H^{(6)}_{g,-r}$, so it suffices to specify the $H^{(6)}_{g,r}$ for $r\in\{1,2,3,4,5\}$. 

To define $H^{(6)}_g=(H^{(6)}_{g,r})$ for $g=e$, first define $h(\tau)=(h_r(\tau))$ by requiring that 
\begin{gather}
	-2\Psi_{1,1}(\tau,z)\varphi^{(6)}_1(\tau,z)
	=-24\mu_{6,0}(\tau,z)+\sum_{r\mod 12}h_{r}(\tau)\theta_{6,r}(\tau,z),
\end{gather}
where
\begin{gather}\label{eqn:exp:vp6}
	\varphi^{(6)}_1(\tau,z):=\varphi^{(2)}_1(\tau,z)\varphi^{(5)}_1(\tau,z)-\varphi^{(3)}_1(\tau,z)\varphi^{(4)}_1(\tau,z).
\end{gather}
(Cf. (\ref{eqn:exp:vp2}), (\ref{eqn:exp:vp3}), (\ref{eqn:exp:vp4}), (\ref{eqn:exp:vp5}).) Now define the $H^{(6)}_{1A,r}$ by setting 
\begin{gather}
	\begin{split}
		H^{(6)}_{1A,1}(\tau)&:=\frac{1}{24}\left(5h_1(\tau)+h_5(\tau)\right),\\
		H^{(6)}_{1A,2}(\tau)&:=\frac{1}{6}h_2(\tau),\\
		H^{(6)}_{1A,3}(\tau)&:=\frac{1}{4}h_3(\tau),\\
		H^{(6)}_{1A,4}(\tau)&:=\frac{1}{6}h_4(\tau),\\
		H^{(6)}_{1A,5}(\tau)&:=\frac{1}{24}\left(h_1(\tau)+5h_5(\tau)\right).		
	\end{split}
\end{gather}
Define $H^{(6)}_{2A,r}$ by requiring
\begin{gather}
	H^{(6)}_{2A,r}(\tau):=-(-1)^rH^{(6)}_{1A,r}(\tau).
\end{gather}

For the remaining $g$, recall (\ref{eqn:exp:specfun-proj}). The $H^{(6)}_{g,r}$ for $g\notin 1A\cup 2A$ are defined as follows for $r=2$ and $r=4$, noting that $H^{(3)}_{g,4}=H^{(3)}_{g,-2}=-H^{(3)}_{g,2}$. 
\begin{gather}
	\begin{split}\label{eqn:exp:4A5D4-r24}
	H^{(6)}_{2B,r}(\t)&:=[-\tfrac{r^2}{24}]H^{(3)}_{4C,r}(\tau/2)
	\\
	H^{(6)}_{4A,r}(\t)&:=[-\tfrac{r^2}{24}]H^{(3)}_{4B,r}(\tau/2)
	\\
	H^{(6)}_{3A,r}(\t)&:=[-\tfrac{r^2}{24}]H^{(3)}_{6C,r}(\tau/2)
	\\
	H^{(6)}_{6A,r}(\t)&:=[-\tfrac{r^2}{24}]H^{(3)}_{6D,r}(\tau/2)
	\\
	H^{(6)}_{8AB,r}(\t)&:=[-\tfrac{r^2}{24}]H^{(3)}_{8CD,r}(\tau/2)
	\end{split}
\end{gather}

For the $H^{(6)}_{g,3}$ we define
\begin{gather}	
	\begin{split}
		H^{(6)}_{2B,3}(\tau), H^{(6)}_{4A,3}(\tau)&:=-[-\tfrac{9}{24}]H^{(2)}_{6A,1}(\tau/3),
		\\ 
		H^{(6)}_{3A,3}(\tau), H^{(6)}_{6A,3}(\tau)&:=0,
		\\ 
		H^{(6)}_{8AB,3}(\tau)&:=-[-\tfrac{9}{24}]H^{(2)}_{12A,1}(\tau/3).
	\end{split}
\end{gather}

Noting that $H^{(2)}_{g,5}=H^{(2)}_{g,1}$ and $H^{(3)}_{g,5}=-H^{(3)}_{g,1}$, the $H^{(6)}_{g,1}$ and $H^{(6)}_{g,5}$ are defined for $o(g)\neq 0\mod 3$ by setting 
\begin{gather}
	\begin{split}
		H^{(6)}_{2B,r}(\tau)&:=[-\tfrac{1}{24}]\frac12\left(H^{(2)}_{6A,r}(\tau/3)+H^{(3)}_{4C,r}(\tau/2)\right)\\
		H^{(6)}_{4A,r}(\tau)&:=[-\tfrac{1}{24}]\frac12\left(H^{(2)}_{6A,r}(\tau/3)+H^{(3)}_{4B,r}(\tau/2)\right)\\
		H^{(6)}_{8AB,r}(\tau)&:=[-\tfrac{1}{24}]\frac12\left(H^{(2)}_{12A,r}(\tau/3)+H^{(3)}_{8CD,r}(\tau/2)\right)
	\end{split}
\end{gather}
It remains to specify the $H^{(6)}_{g,r}$
when $g\in 3A\cup 6A$ and $r$ is $1$ or $5$. These cases are determined by using the formulas of \S\ref{sec:exp:A17E7} to set
\begin{gather}
	\begin{split}
	H^{(6)}_{3A,1}(\tau), H^{(6)}_{6A,1}(\tau)&:=H^{(18)}_{1A,1}(3\tau)-H^{(18)}_{1A,11}(3\tau)+H^{(18)}_{1A,13}(3\tau),\\
	H^{(6)}_{3A,5}(\tau), H^{(6)}_{6A,5}(\tau)&:=H^{(18)}_{1A,5}(3\tau)-H^{(18)}_{1A,7}(3\tau)+H^{(18)}_{1A,17}(3\tau).
	\end{split}
\end{gather}

\subsubsection{$\ell=6+3$, $X=D_4^6$}\label{sec:exp:6D4}

We have $m^X=6$, so the umbral McKay-Thompson series $H^{(6+3)}_g=(H^{(6+3)}_{g,r})$ associated to $g\in G^{(6+3)}$ is a $12$-vector-valued function with components indexed by $r\in \ZZ/12\ZZ$. In addition to the identity $H^{(6+3)}_{g,r}=-H^{(6+3)}_{g,-r}$, we have $H^{(6+3)}_{g,r}=0$ for $r=0\mod 2$. 
Thus it suffices to specify the $H^{(6+3)}_{g,r}$ for $r\in\{1,3,5\}$.

Recall (\ref{eqn:exp:specfun-proj}). For $r=1$, define
\begin{gather}
\begin{split}
	H^{(6+3)}_{1A,1}(\tau), H^{(6+3)}_{3A,1}(\tau)&:=H^{(6)}_{1A,1}(\tau)+H^{(6)}_{1A,5}(\t),\\
	H^{(6+3)}_{2A,1}(\tau), H^{(6+3)}_{6A,1}(\tau)&:=H^{(6)}_{2B,1}(\tau)+H^{(6)}_{2B,5}(\t),\\
	H^{(6+3)}_{3B,1}(\tau)&:=H^{(6)}_{3A,1}(\tau)+H^{(6)}_{3A,5}(\t),\\
	H^{(6+3)}_{3C,1}(\tau)&:=-2 \frac{\eta(\t)^2}{\eta(3\t)},\\
	H^{(6+3)}_{4A,1}(\tau), H^{(6+3)}_{12A,1}(\tau)&:=H^{(6)}_{8AB,1}(\tau)+H^{(6)}_{8AB,5}(\t),\\
	H^{(6+3)}_{5A,1}(\tau), H^{(6+3)}_{15A,1}(\tau)&:=[-\tfrac{1}{24}]H^{(2)}_{15AB,1}(\tau/3),\\
	H^{(6+3)}_{2C,1}(\tau)&:=H^{(6)}_{4A,1}(\tau)-H^{(6)}_{4A,5}(\t),\\
	H^{(6+3)}_{4B,1}(\tau)&:=H^{(6)}_{8AB,1}(\tau)-H^{(6)}_{8AB,5}(\t),\\
	H^{(6+3)}_{6B,1}(\tau)&:=H^{(6)}_{6A,1}(\tau)-H^{(6)}_{6A,5}(\t),\\
	H^{(6+3)}_{6C,1}(\tau)&:=-2 \frac{\eta(2\t)\,\eta(3\t)}{\eta(6\t)}.\\
\end{split}
\end{gather}
Then define $H^{(6+3)}_{2B,1}$ by setting
\begin{gather}
	H^{(6+3)}_{2B,1}(\t):=2H^{(6+3)}_{4B,1}(\t)+2\frac{\eta(\tau)^3}{\eta(2\tau)^2}.
\end{gather}
For $r=3$ set
\begin{gather}
\begin{split}
	H^{(6+3)}_{1A,3}(\tau)&:=2H^{(6)}_{1A,3}(\tau),\\
	H^{(6+3)}_{3A,3}(\tau)&:=-H^{(6)}_{1A,3}(\tau),\\
	H^{(6+3)}_{2A,3}(\tau)&:=2H^{(6)}_{2B,3}(\tau),\\
	H^{(6+3)}_{6A,3}(\tau)&:=-H^{(6)}_{2B,3}(\tau),\\
	H^{(6+3)}_{4A,3}(\tau)&:=2H^{(6)}_{8AB,3}(\tau),\\
	H^{(6+3)}_{12A,3}(\tau)&:=-H^{(6)}_{8AB,3}(\tau),\\
	H^{(6+3)}_{5A,3}(\tau)&:=-2[-\tfrac{9}{24}]H^{(2)}_{15AB,1}(\tau),\\
	H^{(6+3)}_{15A,3}(\tau)&:=[-\tfrac{9}{24}]H^{(2)}_{15AB,1}(\tau),\\
\end{split}
\end{gather}
and
\begin{gather}
	H^{(6+3)}_{3B,3}(\tau),
	H^{(6+3)}_{3C,3}(\tau),
	H^{(6+3)}_{2B,3}(\tau),
	H^{(6+3)}_{2C,3}(\tau),
	H^{(6+3)}_{4B,3}(\tau),
	H^{(6+3)}_{6B,3}(\tau),
	H^{(6+3)}_{6C,3}(\tau):=0.
\end{gather}

For $r=5$ define $H^{(6+3)}_{g,5}(\tau):=H^{(6+3)}_{g,1}(\tau)$ for $[g]\in\{1A, 3A, 2A, 6A, 3B, 3C, 4A, 12A, 5A, 15AB\}$, and set $H^{(6+3)}_{g,5}(\tau):=-H^{(6+3)}_{g,1}(\tau)$ for the remaining cases, $[g]\in \{2B, 2C, 4B, 6B, 6C\}$.

\subsubsection{$\ell=7$, $X=A_6^{4}$}\label{sec:exp:4A6}

We have $m^X=7$, so the umbral McKay-Thompson series $H^{(7)}_g=(H^{(7)}_{g,r})$ associated to $g\in G^{(7)}=G^X\simeq \SL_2(3)$ is a $14$-vector-valued function, with components indexed by $r\in \ZZ/14\ZZ$. 

Define $H^{(7)}_g=(H^{(7)}_{g,r})$ for $g\in G^{(7)}$ by requiring that
\begin{gather}
	\psi^{(7)}_g(\tau,z)=
	-\chi^{(7)}_g\mu_{7,0}^0(\t,z)
	-\bar\chi^{(7)}_g\mu_{7,0}^1(\t,z)
	+\sum_{r\mod 14}H^{(7)}_{g,r}(\t)\theta_{7,r}(\tau,z),
\end{gather}
where $\chi^{(7)}_{g}:=\chi^{X_A}_g$ and $\bar\chi^{(7)}_{g}:=\bar\chi^{X_A}_g$ (cf. Table \ref{tab:chars:eul:7}), and the $\psi^{(7)}_g$ are meromorphic Jacobi forms of weight $1$ and index $7$ given explicitly in Table \ref{tab:4A6-psiXg}.

\begin{table}[ht]
\centering
\caption{Character Values and Meromorphic Jacobi Forms for $\ell=7$, $X=A_6^{4}$ \label{tab:4A6-psiXg}}
\begin{tabular}{rrrl}
$[g]$&$\chi^{(7)}_g$&$\bar\chi^{(7)}_g$&$\psi^{(7)}_g(\tau,z)$\\
\noalign{\vskip 1mm}
\hline
\noalign{\vskip 1mm}
1A&4&4&$2i{\th_1(\tau,4z)}{\th_1(\tau,z)^{-2}}\eta(\tau)^3$
	\\
2A&$-4$&4&$-2i{\th_1(\tau,4z)}{\th_2(\tau,z)^{-2}}\eta(\tau)^3$
	\\
\noalign{\vskip 2pt}\hline\noalign{\vskip 2pt}
4A&$0$&0&$-2i{\th_1(\tau,4z)}{\th_2(2\tau,2z)^{-1}}\eta(2\tau)\eta(\tau)$
	\\
\noalign{\vskip 2pt}\hline\noalign{\vskip 2pt}
3A&1&1& cf. (\ref{eqn:exp:4A6-psiXg})
	\\
6A&$-1$&1& cf. (\ref{eqn:exp:4A6-psiXg})
\end{tabular}
\end{table}

\begin{gather}\label{eqn:exp:4A6-psiXg}	
\begin{split}
	\psi^{(7)}_{3A}(\tau,z)&:=
	-i\frac{\th_1(\tau,4z+\frac13)\th_1(\tau,z-\frac13)+\th_1(\tau,4z-\frac13)\th_1(\tau,z+\frac13)}{\th_1(3\tau,3z)}\eta(3\tau)
	\\
	\psi^{(7)}_{6A}(\tau,z)&:=
	-i\frac{\th_1(\tau,4z+\frac13)\th_1(\tau,z-\frac16)-\th_1(\tau,4z-\frac13)\th_1(\tau,z+\frac16)}{\th_2(3\tau,3z)}\eta(3\tau)
\end{split}
\end{gather}

For use later on we note that $\psi^{(7)}_{1A}=-2\Psi_{1,1}\varphi^{(7)}_1$, where
\begin{gather}\label{eqn:exp:vp7}
	\varphi^{(7)}_1(\tau,z):=\frac{\th_1(\tau,4z)}{\th_1(\tau,2z)}.
\end{gather}

\subsubsection{$\ell=8$, $X=A_7^{2}D_5^2$}\label{sec:exp:2A72D5}

We have $m^X=8$, so the umbral McKay-Thompson series $H^{(8)}_g=(H^{(8)}_{g,r})$ associated to $g\in G^{(8)}$ is a $16$-vector-valued function with components indexed by $r\in \ZZ/16\ZZ$. We have $H^{(8)}_{g,r}=-H^{(8)}_{g,-r}$, so it suffices to specify the $H^{(8)}_{g,r}$ for $r\in\{1,2,3,4,5,6,7\}$. 

To define $H^{(8)}_g=(H^{(8)}_{g,r})$ for $g=e$, first define $h(\tau)=(h_r(\tau))$ by requiring that 
\begin{gather}
	-2\Psi_{1,1}(\tau,z)\left(\varphi^{(8)}_1(\tau,z)+\frac12\varphi^{(8)}_2(\tau,z)\right)
	=-24\mu_{8,0}(\tau,z)+\sum_{r\mod 16}h_{r}(\tau)\theta_{8,r}(\tau,z),
\end{gather}
where
\begin{gather}\label{eqn:exp:vp8}
	\begin{split}
	\varphi^{(8)}_1(\tau,z)&:=\varphi_1^{(3)}(\tau,z)\varphi_1^{(6)}(\tau,z)-5\varphi_1^{(4)}(\tau,z)\varphi_1^{(5)}(\tau,z),\\
	\varphi^{(8)}_2(\tau,z)&:=\varphi_1^{(4)}(\tau,z)\varphi_1^{(5)}(\tau,z)-\varphi_1^{(8)}(\tau,z).
	\end{split}
\end{gather}
(Cf. (\ref{eqn:exp:vp3}), (\ref{eqn:exp:vp4}), (\ref{eqn:exp:vp5}), (\ref{eqn:exp:vp6}).) Now define the $H^{(8)}_{1A,r}$ by setting 
\begin{gather}
		H^{(8)}_{1A,r}(\tau):=\frac{1}{6}h_r(\tau),
\end{gather}
for $r\in \{1,3,4,5,7\}$, and 
\begin{gather}
		H^{(8)}_{1A,2}(\tau),H^{(8)}_{1A,6}(\tau):=\frac{1}{12}\left(h_2(\tau)+h_6(\t)\right).
\end{gather}
Define $H^{(8)}_{2A,r}$ for $1\leq r\leq 7$ by requiring
\begin{gather}
	H^{(8)}_{2A,r}(\tau):=-(-1)^rH^{(8)}_{1A,r}(\tau).
\end{gather}

For the remaining $g$, recall (\ref{eqn:exp:specfun-proj}). The $H^{(8)}_{g,r}$ for $g\in 2B\cup 2C\cup 4A$ are defined as follows for $r\in\{1,3,5,7\}$, noting that $H^{(4)}_{g,7}=H^{(4)}_{g,-1}=-H^{(4)}_{g,1}$, \&c. 
\begin{gather}
	\begin{split}\label{eqn:exp:2A72D5-rel}
	H^{(8)}_{2BC,r}(\t)&:=[-\tfrac{r^2}{32}]H^{(4)}_{4C,r}(\tau/2)
	\\
	H^{(8)}_{4A,r}(\t)&:=[-\tfrac{r^2}{32}]H^{(4)}_{4B,r}(\tau/2)
	\end{split}
\end{gather}
The $H^{(8)}_{2BC,r}$ and $H^{(8)}_{4A,r}$ vanish for $r=0\mod 2$.

\subsubsection{$\ell=9$, $X=A_8^3$}\label{sec:exp:3A8}

We have $m^X=9$, so for $g\in G^{(9)}$ the associated umbral McKay-Thompson series 
$H^{(9)}_g=(H^{(9)}_{g,r})$ is a $18$-vector-valued function, with components indexed by $r\in \ZZ/18\ZZ$, satisfying $H^{(9)}_{g,r}=-H^{(9)}_{g,-r}$, and in particular, $H^{(9)}_{g,r}=0$ for $r=0\mod 9$. So it suffices to specify the $H^{(9)}_{g,r}$ for $r\in \{1,2,3,4,5,6,7,8\}$.

Define $H^{(9)}_g=(H^{(9)}_{g,r})$ for $g=e$ by requiring that 
\begin{gather}
	-\Psi_{1,1}(\tau,z)\varphi_1^{(9)}(\tau,z)
	=-3\mu_{9,0}(\tau,z)+\sum_{r\mod 18}H^{(9)}_{e,r}(\tau)\theta_{9,r}(\tau,z),
\end{gather}
where
\begin{gather}\label{eqn:exp:vp9}
	\varphi_1^{(9)}(\tau,z):=\varphi_1^{(3)}(\tau,z)\varphi_1^{(7)}(\tau,z)-\varphi_1^{(5)}(\tau,z)^2.
\end{gather}
(Cf. (\ref{eqn:exp:vp3}), (\ref{eqn:exp:vp5}), (\ref{eqn:exp:vp7}).)

Recall (\ref{eqn:exp:specfun-proj}). The $H^{(9)}_{2B,r}$ are defined for $r\in\{1,2,4,5,7,8\}$ by setting
\begin{gather}
	\begin{split}\label{eqn:exp:3A8-rel}
	H^{(9)}_{2B,r}(\t)&:=[-\tfrac{r^2}{36}]H^{(3)}_{6C,r}(\tau/3),
	\end{split}
\end{gather}
where we note that $H^{(3)}_{g,4}=H^{(3)}_{g,-2}=-H^{(3)}_{g,2}$, \&c. 
We determine $H^{(9)}_{2B,3}$ and $H^{(9)}_{2B,6}$ by using \S\ref{sec:exp:A17E7} to set
\begin{gather}
	H^{(9)}_{2B,r}(\tau):=H^{(18)}_{1A,r}(2\tau)-H^{(18)}_{1A,18-r}(2\tau)
\end{gather}
for $r\in \{3,6\}$.

The $H^{(9)}_{3A,r}$ are defined by the explicit formulas
\begin{gather}
\begin{split}
	H^{(9)}_{3A,1}(\tau)&:=[-\tfrac{1}{36}]f^{(9)}_1(\tau/3),
	\\
	H^{(9)}_{3A,2}(\tau)&:=[-\tfrac{4}{36}]f^{(9)}_2(\tau/3),
	\\
	H^{(9)}_{3A,3}(\tau)&:=-\theta_{3,3}(\tau,0),
	\\
	H^{(9)}_{3A,4}(\tau)&:=-[-\tfrac{16}{36}]f^{(9)}_2(\tau/3),
	\\
	H^{(9)}_{3A,5}(\tau)&:=-[-\tfrac{25}{36}]f^{(9)}_1(\tau/3), 
	\\
	H^{(9)}_{3A,6}(\tau)&:=\theta_{3,0}(\tau,0),
	\\
	H^{(9)}_{3A,7}(\tau)&:=[-\tfrac{13}{36}]f^{(9)}_1(\tau/3), 
	\\
	H^{(9)}_{3A,8}(\tau)&:=[-\tfrac{28}{36}]f^{(9)}_2(\tau/3), 
\end{split}
\end{gather}
where
\begin{gather}
\begin{split}
f^{(9)}_1(\tau)&:= -2\frac{\h(\t) \h(12\t) \h(18\t)^2 }{\h(6\t) \h(9\t) \h(36\t)},  \\
f^{(9)}_2(\tau)&:= \frac{\h(2\t)^6 \h(12\t) \h(18\t)^2 }{\h(\t) \h(4\t)^4 \h(6\t)\h(9\t)\h(36\t)}-\frac{\h(\t) \h(2\t) \h(3\t)^2 }{\h(4\t)^2 \h(9\t)}.
\end{split}
\end{gather}

Finally, the $H^{(9)}_{g,r}$ are determined for $g\in 2A\cup 2C\cup 6A$ by setting 
\begin{gather}
\begin{split}
H^{(9)}_{2A,r}(\tau)&:=(-1)^{r+1}H^{(9)}_{1A,r}(\tau),\\
H^{(9)}_{2C,r}(\tau)&:=(-1)^{r+1}H^{(9)}_{2B,r}(\tau),\\
H^{(9)}_{6A,r}(\tau)&:=(-1)^{r+1}H^{(9)}_{3A,r}(\tau).
\end{split}
\end{gather}

\subsubsection{$\ell=10$, $X=A_9^2D_6$}\label{sec:exp:2A9D6}

We have $m^X=10$, so the umbral McKay-Thompson series $H^{(10)}_g=(H^{(10)}_{g,r})$ associated to $g\in G^{(10)}$ is a $20$-vector-valued function with components indexed by $r\in \ZZ/20\ZZ$. We have $H^{(10)}_{g,r}=-H^{(10)}_{g,-r}$, so it suffices to specify the $H^{(10)}_{g,r}$ for $1\leq r\leq 9$. 

To define $H^{(10)}_g=(H^{(10)}_{g,r})$ for $g=e$, first define $h(\tau)=(h_r(\tau))$ by requiring that 
\begin{gather}
	-6\Psi_{1,1}(\tau,z)\varphi^{(10)}_1(\tau,z)
	=-24\mu_{10,0}(\tau,z)+\sum_{r\mod 20}h_{r}(\tau)\theta_{10,r}(\tau,z),
\end{gather}
where
\begin{gather}\label{eqn:exp:vp10}
	\varphi^{(10)}_1(\tau,z):=5\varphi_1^{(4)}(\tau,z)\varphi_1^{(7)}(\tau,z)-\varphi_1^{(5)}(\tau,z)\varphi_1^{(6)}(\tau,z).
\end{gather}
(Cf. (\ref{eqn:exp:vp4}), (\ref{eqn:exp:vp5}), (\ref{eqn:exp:vp6}), (\ref{eqn:exp:vp7}).) Now define the $H^{(10)}_{1A,r}$ for $r$ odd by setting 
\begin{gather}
	\begin{split}
		H^{(10)}_{1A,1}(\tau)&:=\frac{1}{24}\left(3h_1(\tau)+h_9(\t)\right),\\
		H^{(10)}_{1A,3}(\tau)&:=\frac{1}{24}\left(3h_3(\tau)+h_7(\t)\right),\\
		H^{(10)}_{1A,5}(\tau)&:=\frac{1}{6}h_5(\tau),\\		
		H^{(10)}_{1A,3}(\tau)&:=\frac{1}{24}\left(h_3(\tau)+3h_7(\t)\right),\\
		H^{(10)}_{1A,9}(\tau)&:=\frac{1}{24}\left(h_1(\tau)+3h_9(\t)\right).
	\end{split}
\end{gather}
For $r=0\mod 2$ set
\begin{gather}
	H^{(10)}_{1A,r}(\tau):=\frac{1}{12}h_r(\tau),
\end{gather}
and define $H^{(10)}_{2A,r}$ for $1\leq r\leq 9$ by requiring
\begin{gather}
	H^{(10)}_{2A,r}(\tau):=-(-1)^rH^{(10)}_{1A,r}(\tau).
\end{gather}

It remains to specify $H^{(10)}_{g,r}$ for $g\in 4A\cup 4B$. For $r=0\mod 2$ set 
\begin{gather}
	H^{(10)}_{4AB,r}(\tau):=0.
\end{gather}
For $r$ odd, recall (\ref{eqn:exp:specfun-proj}), and define 
\begin{gather}\label{eqn:exp:2A9D6-rel}
	H^{(10)}_{4A,r}(\t):=[-\tfrac{r^2}{40}]\frac{1}{2}\left(H^{(2)}_{10A,r}(\tau/5)+H^{(5)}_{4CD,r}(\tau/2)\right).
\end{gather}

\subsubsection{$\ell=10+5$, $X=D_6^4$}\label{sec:exp:4D6}

We have $m^X=10$, so the umbral McKay-Thompson series $H^{(10+5)}_g=(H^{(10+5)}_{g,r})$ associated to $g\in G^{(10+5)}$ is a $20$-vector-valued function with components indexed by $r\in \ZZ/20\ZZ$. We have $H^{(10+5)}_{g,r}=0$ for $r=0\mod 2$, so it suffices to specify the $H^{(10+5)}_{g,r}$ for $r$ odd. Observing that $H^{(10+5)}_{g,r}=-H^{(10+5)}_{g,-r}$ we may determine $H^{(10+5)}_g$ by requiring that
\begin{gather}
	\psi^{(5/2)}_g(\tau,z)
	=-2\chi^{(5/2)}_gi\mu_{5/2,0}(\tau,z)+
	\sum_{\substack{r\in\ZZ+1/2\\  r\mod 5}}
	e(-r/2)H^{(10+5)}_{g,2r}(\tau)\theta_{5/2,r}(\tau,z),
\end{gather}
where $\chi^{(5/2)}_g:=\bar\chi^{X_D}_g$ as in Table \ref{tab:chars:eul:10+5}, and the $\psi^{(5/2)}_g$ are the meromorphic Jacobi forms of weight $1$ and index $5/2$ defined as follows.

\begin{table}[ht]
\centering
\caption{Character Values and Meromorphic Jacobi Forms for $\ell=10+5$, $X=D_{6}^{4}$ \label{tab:4D6-psiXg}}
\begin{tabular}{rrl}
$[g]$&$\bar\chi^{(5/2)}_g$&$\psi^{(5/2)}_g(\tau,z)$\\
\noalign{\vskip 1mm}
\hline
\noalign{\vskip 1mm}
1A&4&$2i{\theta_1(\tau,2z)^2}{\theta_1(\tau,z)^{-3}}\eta(\tau)^3$
	\\
2A&0&$-2i{\theta_1(\tau,2z)^2}{\theta_1(\tau,z)^{-1}\theta_2(\tau,z)^{-2}}\eta(\tau)^3$
	\\
3A&$1$& $2i{\theta_1(3\tau,6z)}{\theta_1(\tau,2z)^{-1}\theta_1(3\tau,3z)^{-1}}\eta(\tau)^3$
	\\
2B&2&$2i{\theta_1(\tau,2z)\theta_2(\tau,2z)}{\theta_1(\tau,z)^{-2}\theta_2(\tau,z)^{-1}}\eta(\tau)^3$
	\\
4A&0&$-2i{\theta_1(\tau,2z)\theta_2(\tau,2z)}{\theta_2(2\tau,2z)^{-1}}\eta(\tau)\eta(2\tau)$
\end{tabular}
\end{table}

\subsubsection{$\ell=12$, $X=A_{11}D_7E_6$}\label{sec:exp:A11D7E6}

We have $m^X=12$, so the umbral McKay-Thompson series $H^{(12)}_g=(H^{(12)}_{g,r})$ associated to $g\in G^{(12)}\simeq \ZZ/2\ZZ$ is a $24$-vector-valued function with components indexed by $r\in \ZZ/24\ZZ$. We have $H^{(12)}_{g,r}=-H^{(12)}_{g,-r}$, so it suffices to specify the $H^{(12)}_{g,r}$ for $1\leq r\leq 11$. 

To define $H^{(12)}_e=(H^{(12)}_{e,r})$, first define $h(\tau)=(h_r(\tau))$ by requiring that 
\begin{gather}
	-2\Psi_{1,1}(\tau,z)\left(\varphi^{(12)}_1(\tau,z)+\varphi_2^{(12)}(\tau,z)\right)
	=-24\mu_{12,0}(\tau,z)+\sum_{r\mod 24}h_{r}(\tau)\theta_{12,r}(\tau,z),
\end{gather}
where
\begin{gather}\label{eqn:exp:vp12}
	\begin{split}
	\varphi^{(12)}_1(\tau,z)&:=3\varphi_1^{(3)}(\tau,z)\varphi_1^{(10)}(\tau,z)-8\varphi_1^{(4)}(\tau,z)\varphi_1^{(9)}(\tau,z)+\varphi_1^{(5)}(\tau,z)\varphi_1^{(8)}(\tau,z),\\
	\varphi^{(12)}_2(\tau,z)&:=4\varphi_1^{(4)}(\tau,z)\varphi_1^{(9)}(\tau,z)-\varphi_1^{(5)}(\tau,z)\varphi_1^{(8)}(\tau,z)-\varphi_1^{(12)}(\tau,z).
	\end{split}
\end{gather}
(Cf. (\ref{eqn:exp:vp3}), (\ref{eqn:exp:vp4}), (\ref{eqn:exp:vp5}), (\ref{eqn:exp:vp7}), (\ref{eqn:exp:vp8}), (\ref{eqn:exp:vp9}), (\ref{eqn:exp:vp10}).) Now define the $H^{(12)}_{1A,r}$ for $r\neq 0 \mod 3$ by setting 
\begin{gather}
	\begin{split}
		H^{(12)}_{1A,1}(\tau)&:=\frac{1}{24}\left(3h_1(\tau)+h_7(\t)\right),\\
		H^{(12)}_{1A,2}(\tau), H^{(12)}_{1A,10}(\tau)&:=\frac{1}{24}\left(h_2(\tau)+h_{10}(\tau)\right),\\
		H^{(12)}_{1A,4}(\tau),H^{(12)}_{1A,8}(\tau)&:=\frac{1}{12}\left(h_4(\tau)+h_8(\t)\right),\\
		H^{(12)}_{1A,5}(\tau)&:=\frac{1}{24}\left(3h_5(\tau)+h_{11}(\t)\right),\\
		H^{(12)}_{1A,7}(\tau)&:=\frac{1}{24}\left(h_1(\tau)+3h_7(\t)\right),\\
		H^{(12)}_{1A,11}(\tau)&:=\frac{1}{24}\left(h_5(\tau)+3h_{11}(\t)\right).
	\end{split}
\end{gather}
For $r=0\mod 3$ set
\begin{gather}
	H^{(12)}_{1A,r}(\tau):=\frac{1}{12}h_r(\tau),
\end{gather}
and define $H^{(12)}_{2A,r}$ by requiring
\begin{gather}
	H^{(12)}_{2A,r}(\tau):=-(-1)^rH^{(12)}_{1A,r}(\tau).
\end{gather}

\subsubsection{$\ell=12+4$, $X=E_6^4$}\label{sec:exp:4E6}

We have $m^X=12$, so the umbral McKay-Thompson series $H^{(12+4)}_g=(H^{(12+4)}_{g,r})$ associated to $g\in G^{(12+4)}$ is a $24$-vector-valued function with components indexed by $r\in \ZZ/24\ZZ$. In addition to the identity $H^{(12+4)}_{g,r}=-H^{(12+4)}_{g,-r}$, we have $H^{(12+4)}_{g,r}=0$ for $r\in\{2,3,6,9,10\}$, $H^{(12+4)}_{g,1}=H^{(12+4)}_{g,7}$, $H^{(12+4)}_{g,4}=H^{(12+4)}_{g,8}$, and $H^{(12+4)}_{g,5}=H^{(12+4)}_{g,11}$. Thus it suffices to specify the $H^{(12+4)}_{g,1}$, $H^{(12+4)}_{g,4}$ and $H^{(12+4)}_{g,5}$.

Recall (\ref{eqn:exp:specfun-proj}). Also, set $S^{E_6}_1(\tau):=S_{12,1}(\tau)+S_{12,7}(\t)$, and $S^{E_6}_5(\t):=S_{12,5}(\t)+S_{12,11}(\t)$. For $r=1$ define
\begin{gather}
\begin{split}
	H^{(12+4)}_{1A,1}(\tau)
	&:=H^{(12)}_{1A,1}(\tau)+H^{(12)}_{1A,7}(\tau),\\
	H^{(12+4)}_{2B,1}(\tau)	&:=[-\tfrac{1}{48}]\left(H^{(6)}_{8AB,1}(\tau/2)-H^{(6)}_{8AB,5}(\tau/2)\right),\\
	H^{(12+4)}_{4A,1}(\tau)	&:=\frac{1}{S^{E_6}_{1}(\t)^2-S^{E_6}_5(\t)^2}\left(-2\frac{\eta(2\t)^8}{\eta(\t)^4}S^{E_6}_1(\t)+8\frac{\eta(\t)^4\eta(4\t)^4}{\eta(2\t)^4}S^{E_6}_5(\t)\right),\\
	H^{(12+4)}_{3A,1}(\tau)	&:=[-\tfrac{1}{48}]\left(H^{(6)}_{3A,1}(\tau/2)-H^{(6)}_{3A,5}(\tau/2)\right),\\
	H^{(12+4)}_{8AB,1}(\tau)	&:=\frac{1}{S^{E_6}_{1}(\t)^2-S^{E_6}_5(\t)^2}\left(-2F^{(12+4)}_{8AB,1}(\t)S^{E_6}_1(\t)+12F^{(12+4)}_{8AB,5}(\tau/2)S^{E_6}_5(\t)\right).
\end{split}
\end{gather}
In the expression for $g\in 8AB$, we write $F^{(12+4)}_{8AB,1}$ for the unique modular form of weight $2$ for $\Gamma_0(32)$ such that
\begin{gather}
F^{(12+4)}_{8AB,1}(\t)=1 + 12q + 4q^2 - 24q^5 - 16q^6 - 8q^8+O(q^9),
\end{gather}
and we write $F^{(12+4)}_{8AB,5}$ for the unique modular form of weight $2$ for $\Gamma_0(64)$ such that
\begin{gather}
F^{(12+4)}_{8AB,5}(\t)=3q + 4q^3 + 6q^5 - 8q^7 - 9q^9 + 12q^{11} - 18q^{13} - 24q^{15} + O(q^{17}).
\end{gather}

For $r=4$ define
\begin{gather}
\begin{split}
	H^{(12+4)}_{1A,4}(\tau)&:=H^{(12)}_{1A,4}(\tau)+H^{(12)}_{1A,8}(\tau),\\
	H^{(12+4)}_{3A,4}(\tau)	&:=H^{(6)}_{3A,2}(\tau/2)+H^{(6)}_{3A,4}(\tau/2),\\
\end{split}
\end{gather}
and set $H^{(12+4)}_{g,4}(\tau):=0$ for $g\in 2B\cup 4A\cup 8AB$.

For $r=5$ define
\begin{gather}
\begin{split}
	H^{(12+4)}_{1A,5}(\tau)
	&:=H^{(12)}_{1A,5}(\tau)+H^{(12)}_{1A,11}(\tau),\\
	H^{(12+4)}_{2B,5}(\tau)	&:=[-\tfrac{25}{48}]\left(H^{(6)}_{8AB,5}(\tau/2)-H^{(6)}_{8AB,1}(\tau/2)\right),\\
	H^{(12+4)}_{4A,5}(\tau)	&:=\frac{1}{S^{E_6}_{1}(\t)^2-S^{E_6}_5(\t)^2}\left(2\frac{\eta(2\t)^8}{\eta(\t)^4}S^{E_6}_5(\t)-8\frac{\eta(\t)^4\eta(4\t)^4}{\eta(2\t)^4}S^{E_6}_1(\t)\right),\\
	H^{(12+4)}_{3A,5}(\tau)	&:=[-\tfrac{25}{48}]\left(H^{(6)}_{3A,5}(\tau/2)-H^{(6)}_{3A,1}(\tau/2)\right),\\
	H^{(12+4)}_{8AB,5}(\tau)	&:=\frac{1}{S^{E_6}_{1}(\t)^2-S^{E_6}_5(\t)^2}\left(2F^{(12+4)}_{8AB,1}(\t)S^{E_6}_5(\t)-12F^{(12+4)}_{8AB,5}(\tau/2)S^{E_6}_1(\t)\right).
\end{split}
\end{gather}

Finally, define $H^{(12+4)}_{g,r}$ for $g\in 2A\cup 6A$ by setting
\begin{gather}
\begin{split}
	H^{(12+4)}_{2A,r}(\tau)&:=-(-1)^rH^{(12+4)}_{1A,r}(\tau),\\
	H^{(12+4)}_{6A,r}(\tau)&:=-(-1)^rH^{(12+4)}_{3A,r}(\tau).
\end{split}
\end{gather}

\subsubsection{$\ell=13$, $X=A_{12}^2$}\label{sec:exp:2A12}

We have $m^X=13$, so the umbral McKay-Thompson series $H^{(13)}_g=(H^{(13)}_{g,r})$ associated to $g\in G^{(13)}=G^X\simeq \ZZ/4\ZZ$ is a $26$-vector-valued function, with components indexed by $r\in \ZZ/26\ZZ$. 

Define $H^{(13)}_g=(H^{(13)}_{g,r})$ for $g\in G^{(13)}$ by requiring that
\begin{gather}
	\psi^{(13)}_g(\tau,z)=
	-\chi^{(13)}_g\mu_{13,0}^0(\t,z)
	-\bar\chi^{(13)}_g\mu_{13,0}^1(\t,z)
	+\sum_{r\mod 26}H^{(13)}_{g,r}(\t)\theta_{13,r}(\tau,z),
\end{gather}
where $\chi^{(13)}_{g}:=\chi^{X_A}_g$ and $\bar\chi^{(13)}_{g}:=\bar\chi^{X_A}_g$ (cf. Table \ref{tab:chars:eul:13}), and the $\psi^{(13)}_g$ are meromorphic Jacobi forms of weight $1$ and index $13$ given explicitly in Table \ref{tab:2A12-psiXg}.

\begin{table}[ht]
\centering
\caption{Character Values and Meromorphic Jacobi Forms for $\ell=13$, $X=A_{12}^{2}$ \label{tab:2A12-psiXg}}
\begin{tabular}{rrrl}
$[g]$&$\chi^{(13)}_g$&$\bar\chi^{(13)}_g$&$\psi^{(13)}_g(\tau,z)$\\
\noalign{\vskip 1mm}
\hline
\noalign{\vskip 1mm}
1A&2&2&$2i{\th_1(\tau,6z)}{\th_1(\tau,z)^{-1}\th_1(\tau,3z)^{-1}}\eta(\tau)^3$
	\\
2A&$-2$&2&$-2i{\th_1(\tau,6z)}{\th_2(\tau,z)^{-1}\th_2(\tau,3z)^{-1}}\eta(\tau)^3$
	\\
\noalign{\vskip 2pt}\hline\noalign{\vskip 2pt}
4A&$0$&0& cf. (\ref{eqn:exp:2A12-psiXg})
\end{tabular}
\end{table}

\begin{gather}\label{eqn:exp:2A12-psiXg}	
\begin{split}	
	\psi^{(13)}_{4AB}(\tau,z)&:=
	-i\th_2(\tau,6z)\frac{\th_1(\tau,z+\frac14)\th_1(\tau,3z+\frac14)-\th_1(\tau,z-\frac14)\th_1(\tau,3z-\frac14)}{\th_2(2\tau,2z)\th_2(2\tau,6z)}\frac{\eta(2\tau)^2}{\eta(\tau)}
\end{split}
\end{gather}

For use later on we note that $\psi^{(13)}_{1A}=-2\Psi_{1,1}\varphi^{(13)}_1$, where
\begin{gather}\label{eqn:exp:vp13}
	\varphi^{(13)}_1(\tau,z):=\frac{\th_1(\tau,z)\th_1(\tau,6z)}{\th_1(\tau,2z)\th_1(\tau,3z)}.
\end{gather}

\subsubsection{$\ell=14+7$, $X=D_8^3$}\label{sec:exp:3D8}

We have $m^X=14$, so the umbral McKay-Thompson series $H^{(14+7)}_g=(H^{(14+7)}_{g,r})$ associated to $g\in G^{(14+7)}$ is a $28$-vector-valued function with components indexed by $r\in \ZZ/28\ZZ$. We have $H^{(14+7)}_{g,r}=0$ for $r=0\mod 2$, so it suffices to specify the $H^{(14+7)}_{g,r}$ for $r$ odd. Observing that $H^{(14+7)}_{g,r}=-H^{(14+7)}_{g,-r}$ we may determine $H^{(14+7)}_g$ by requiring that
\begin{gather}
	\psi^{(7/2)}_g(\tau,z)
	=-2\bar\chi^{(7/2)}_gi\mu_{7/2,0}(\tau,z)+
	\sum_{\substack{r\in\ZZ+1/2\\  r\mod 7}}
	e(-r/2)H^{(14+7)}_{g,2r}(\tau)\theta_{7/2,r}(\tau,z),
\end{gather}
where $\bar\chi^{(7/2)}_g:=\bar\chi^{X_D}_g$ is the number of fixed points of $g\in G^{(14+7)}\simeq S_3$ in the defining permutation representation on $3$ points. The $\psi^{(7/2)}_g$ are the meromorphic Jacobi forms of weight $1$ and index $7/2$ defined in Table \ref{tab:3D8-psiXg}.

\begin{table}[ht]
\centering
\caption{Character Values and Meromorphic Jacobi Forms for $\ell=14+7$, $X=D_{8}^{3}$ \label{tab:3D8-psiXg}}
\begin{tabular}{rrl}
$[g]$&$\bar\chi^{(7/2)}_g$&$\psi^{(7/2)}_g(\tau,z)$\\
\noalign{\vskip 1mm}
\hline
\noalign{\vskip 1mm}
1A&3&$2i{\theta_1(\tau,3z)}{\theta_1(\tau,z)^{-2}}\eta(\tau)^3$
	\\
2A&1&$2i{\theta_2(\tau,3z)}{\theta_1(\tau,z)^{-1}\theta_2(\tau,z)^{-1}}\eta(\tau)^3$
	\\
3A&$0$& $-2i\th_1(\tau,z)\th_1(\tau,3z)\th_1(3\tau,3z)^{-1}\eta(3\tau)$
\end{tabular}
\end{table}

\subsubsection{$\ell=16$, $X=A_{15}D_9$}\label{sec:exp:A15D9}

We have $m^X=16$, so the umbral McKay-Thompson series $H^{(16)}_g=(H^{(16)}_{g,r})$ associated to $g\in G^{(16)}\simeq \ZZ/2\ZZ$ is a $32$-vector-valued function with components indexed by $r\in \ZZ/32\ZZ$. We have $H^{(16)}_{g,r}=-H^{(16)}_{g,-r}$, so it suffices to specify the $H^{(16)}_{g,r}$ for $1\leq r\leq 15$. 

To define $H^{(16)}_g=(H^{(16)}_{g,r})$ for $g=e$, first define $h(\tau)=(h_r(\tau))$ by requiring that 
\begin{gather}
	-6\Psi_{1,1}(\tau,z)\left(\varphi^{(16)}_1(\tau,z)+\frac12\varphi_2^{(16)}(\tau,z)\right)
	=-24\mu_{16,0}(\tau,z)+\sum_{r\mod 32}h_{r}(\tau)\theta_{16,r}(\tau,z),
\end{gather}
where
\begin{gather}\label{eqn:exp:vp16}
	\begin{split}
	\varphi^{(16)}_1(\tau,z)&:=8\varphi_1^{(4)}(\tau,z)\varphi_1^{(13)}(\tau,z)-\varphi_1^{(5)}(\tau,z)\varphi_1^{(12)}(\tau,z)+\varphi_1^{(7)}(\tau,z)\varphi_1^{(10)}(\tau,z),\\
	\varphi^{(16)}_2(\tau,z)&:=12\varphi_1^{(4)}(\tau,z)\varphi_1^{(13)}(\tau,z)-\varphi_1^{(5)}(\tau,z)\varphi_1^{(12)}(\tau,z)-3\varphi_1^{(16)}(\tau,z).
	\end{split}
\end{gather}
(Cf. (\ref{eqn:exp:vp4}), (\ref{eqn:exp:vp5}), (\ref{eqn:exp:vp7}), (\ref{eqn:exp:vp10}), (\ref{eqn:exp:vp12}), (\ref{eqn:exp:vp13}).) Now define the $H^{(16)}_{1A,r}$ by setting 
\begin{gather}
	H^{(16)}_{1A,r}(\tau):=\frac{1}{12}h_r(\t)
\end{gather}
for $r$ odd. For $r$ even, $2\leq r\leq 14$, use
\begin{gather}
	H^{(16)}_{1A,r}(\tau):=\frac{1}{24}\left(h_r(\t)+h_{16-r}(\t)\right).
\end{gather}
Define $H^{(16)}_{2A,r}$ by requiring
\begin{gather}
	H^{(16)}_{2A,r}(\tau):=-(-1)^rH^{(16)}_{1A,r}(\tau).
\end{gather}

\subsubsection{$\ell=18$, $X=A_{17}E_7$}\label{sec:exp:A17E7}

We have $m^X=18$, so the umbral McKay-Thompson series $H^{(18)}_g=(H^{(18)}_{g,r})$ associated to $g\in G^{(18)}\simeq \ZZ/2\ZZ$ is a $36$-vector-valued function with components indexed by $r\in \ZZ/36\ZZ$. We have $H^{(18)}_{g,r}=-H^{(18)}_{g,-r}$, so it suffices to specify the $H^{(18)}_{g,r}$ for $1\leq r\leq 17$. 

To define $H^{(18)}_g=(H^{(18)}_{g,r})$ for $g=e$, first define $h(\tau)=(h_r(\tau))$ by requiring that 
\begin{gather}
	-24\Psi_{1,1}(\tau,z)\phi^{(18)}(\tau,z)
	=-24\mu_{18,0}(\tau,z)+\sum_{r\mod 36}h_{r}(\tau)\theta_{18,r}(\tau,z),
\end{gather}
where
\begin{gather}\label{eqn:exp:phi18}
	\phi^{(18)}:=\frac{1}{12}\left(\varphi^{(18)}_1+\frac13\varphi_3^{(18)}+4\frac{\th_1^{12}}{\eta^{12}}\left(\varphi_1^{(12)}+2\varphi_2^{(12)}+\frac13\varphi_3^{(12)}\right)\right).
\end{gather}
For the definition of $\phi^{(18)}$ we require
\begin{gather}
	\begin{split}\label{eqn:exp:vp18}
	\varphi^{(9)}_2(\tau,z)&:=\varphi_1^{(4)}(\tau,z)\varphi_1^{(6)}(\tau,z)-4\varphi_1^{(5)}(\tau,z)^2-4\varphi_1^{(9)}(\tau,z),\\
	\varphi^{(11)}_1(\tau,z)&:=3\varphi_1^{(5)}(\tau,z)\varphi_1^{(7)}(\tau,z)+2\varphi_1^{(3)}(\tau,z)\varphi_1^{(9)}(\tau,z)-\varphi_1^{(4)}(\tau,z)\varphi_1^{(8)}(\tau,z),\\
	\varphi^{(12)}_3(\tau,z)&:=\varphi_1^{(4)}(\tau,z)\varphi_2^{(9)}(\tau,z),\\
	\varphi^{(14)}_1(\tau,z)&:=3\varphi_1^{(5)}(\tau,z)\varphi_1^{(10)}(\tau,z)+\varphi_1^{(3)}(\tau,z)\varphi_1^{(12)}(\tau,z)-4\varphi_1^{(4)}(\tau,z)\varphi_1^{(11)}(\tau,z),\\
	\varphi^{(15)}_1(\tau,z)&:=\varphi_1^{(5)}(\tau,z)\varphi_1^{(11)}(\tau,z)+6\varphi_1^{(3)}(\tau,z)\varphi_1^{(13)}(\tau,z)-\varphi_1^{(4)}(\tau,z)\varphi_1^{(12)}(\tau,z),\\
	\varphi^{(15)}_2(\tau,z)&:=\varphi_1^{(4)}(\tau,z)\varphi_1^{(12)}(\tau,z)-2\varphi_1^{(5)}(\tau,z)\varphi_1^{(11)}(\tau,z)-2\varphi_1^{(15)}(\tau,z),\\
	\varphi^{(18)}_1(\tau,z)&:=\varphi_1^{(5)}(\tau,z)\varphi_1^{(14)}(\tau,z)+3\varphi_1^{(3)}(\tau,z)\varphi_1^{(16)}(\tau,z)-4\varphi_1^{(4)}(\tau,z)\varphi_1^{(15)}(\tau,z),\\
	\varphi^{(18)}_3(\tau,z)&:=\varphi_1^{(4)}(\tau,z)\varphi_2^{(15)}(\tau,z),
	\end{split}
\end{gather}
in addition to the other $\varphi^{(m)}_k$ that have appeared already. Now define the $H^{(18)}_{1A,r}$ by setting 
\begin{gather}
	H^{(18)}_{1A,r}(\tau):=\frac{1}{24}h_r(\t)
\end{gather}
for $r$ even. For $r$ odd, use
\begin{gather}
		\begin{split}
	H^{(18)}_{1A,1}(\tau)&:=\frac{1}{24}\left(2h_1(\t)+h_{17}(\t)\right),\\
	H^{(18)}_{1A,3}(\tau)&:=\frac{1}{24}\left(h_3(\t)+h_{9}(\t)\right),\\
	H^{(18)}_{1A,5}(\tau)&:=\frac{1}{24}\left(2h_5(\t)+h_{13}(\t)\right),\\
	H^{(18)}_{1A,7}(\tau)&:=\frac{1}{24}\left(2h_7(\t)+h_{11}(\t)\right),\\
	H^{(18)}_{1A,9}(\tau)&:=\frac{1}{24}\left(h_3(\t)+2h_9(\t)+h_{15}(\t)\right),\\
	H^{(18)}_{1A,11}(\tau)&:=\frac{1}{24}\left(h_7(\t)+2h_{11}(\t)\right),\\
	H^{(18)}_{1A,13}(\tau)&:=\frac{1}{24}\left(h_5(\t)+2h_{13}(\t)\right),\\
	H^{(18)}_{1A,15}(\tau)&:=\frac{1}{24}\left(h_{15}(\t)+h_{9}(\t)\right),\\
	H^{(18)}_{1A,17}(\tau)&:=\frac{1}{24}\left(h_1(\t)+2h_{17}(\t)\right).	
		\end{split}
\end{gather}
Define $H^{(18)}_{2A,r}$ in the usual way for root systems with a type A component, by requiring
\begin{gather}
	H^{(18)}_{2A,r}(\tau):=-(-1)^rH^{(18)}_{1A,r}(\tau).
\end{gather}

\subsubsection{$\ell=18+9$, $X=D_{10}E_7^2$}\label{sec:exp:D102E7}

We have $m^X=18$, so the umbral McKay-Thompson series $H^{(18+9)}_g=(H^{(18+9)}_{g,r})$ associated to $g\in G^{(18+9)}\simeq \ZZ/2\ZZ$ is a $36$-vector-valued function with components indexed by $r\in \ZZ/36\ZZ$. We have $H^{(18+9)}_{g,r}=-H^{(18+9)}_{g,-r}$, $H^{(18+9)}_{g,r}=H^{(18+9)}_{g,18-r}$ for $1\leq r\leq 17$, and $H^{(18+9)}_{g,r}=0$ for $r=0\mod 2$, so it suffices to specify the $H^{(18+9)}_{g,r}$ for $r\in\{1,3,5,7,9\}$. 

Define
\begin{gather}
	\begin{split}
	H^{(18+9)}_{1A,r}(\t)&:=H^{(18)}_{1A,r}(\t)+H^{(18)}_{1A,18-r}(\t),\\
	H^{(18+9)}_{2A,r}(\t)&:=H^{(18)}_{1A,r}(\t)-H^{(18)}_{1A,18-r}(\t),\\
	\end{split}
\end{gather}
for $r\in\{1,3,5,7,9\}$.

\subsubsection{$\ell=22+11$, $X=D_{12}^2$}\label{sec:exp:2D12}

We have $m^X=22$, so the umbral McKay-Thompson series $H^{(22+11)}_g=(H^{(22+11)}_{g,r})$ associated to $g\in G^{(22+11)}\simeq \ZZ/2\ZZ$ is a $44$-vector-valued function with components indexed by $r\in \ZZ/44\ZZ$. We have $H^{(22+11)}_{g,r}=-H^{(22+11)}_{g,-r}$ and $H^{(22+11)}_{g,r}=0$ for $r=0\mod 2$, so it suffices to specify the $H^{(22+11)}_{g,r}$ for $r$ odd. Observing that $H^{(22+11)}_{g,r}=-H^{(22+11)}_{g,-r}$ we may determine $H^{(22+11)}_g$ by requiring that
\begin{gather}
	\psi^{(11/2)}_g(\tau,z)
	=-2\bar\chi^{(11/2)}_gi\mu_{11/2,0}(\tau,z)+
	\sum_{\substack{r\in\ZZ+1/2\\  r\mod 11}}
	e(-r/2)H^{(22+11)}_{g,2r}(\tau)\theta_{11/2,r}(\tau,z),
\end{gather}
where $\bar\chi^{(11/2)}_{1A}:=2$, $\bar\chi^{(11/2)}_{2A}:=0$, and the $\psi^{(11/2)}_g$ are the meromorphic Jacobi forms of weight $1$ and index $11/2$ defined as follows.
\begin{gather}
	\begin{split}
\psi^{(11/2)}_{1A}(\tau,z)&:=2i\frac{\theta_1(\tau,4z)}{\theta_1(\tau,z)\theta_1(\tau,2z)}\eta(\tau)^3\\
\psi^{(11/2)}_{2A}(\tau,z)&:=-2i\frac{\theta_1(\tau,4z)}{\theta_2(\tau,z)\theta_2(\tau,2z)}\eta(\tau)^3
	\end{split}
\end{gather}

\subsubsection{$\ell=25$, $X=A_{24}$}\label{sec:exp:A24}

We have $m^X=25$, so for $g\in G^{(25)}\simeq \ZZ/2\ZZ$, the associated umbral McKay-Thompson series 
$H^{(25)}_g=(H^{(25)}_{g,r})$ is a $50$-vector-valued function, with components indexed by $r\in \ZZ/50\ZZ$, satisfying $H^{(25)}_{g,r}=-H^{(25)}_{g,-r}$, and in particular, $H^{(25)}_{g,r}=0$ for $r=0\mod 25$. So it suffices to specify the $H^{(25)}_{g,r}$ for $1\leq r\leq 24$.

Define $H^{(25)}_g=(H^{(25)}_{g,r})$ for $g=e$ by requiring that 
\begin{gather}
	-\Psi_{1,1}(\tau,z)\varphi_1^{(25)}(\tau,z)
	=-\mu_{25,0}(\tau,z)+\sum_{r\mod 50}H^{(25)}_{e,r}(\tau)\theta_{25,r}(\tau,z),
\end{gather}
where
\begin{gather}\label{eqn:exp:vp25}
	\varphi_1^{(25)}(\tau,z):=\frac12\varphi_1^{(5)}(\tau,z)\varphi_1^{(21)}(\tau,z)-\varphi_1^{(7)}(\tau,z)\varphi_1^{(19)}(\tau,z)+\frac12\varphi_1^{(13)}(\tau,z)^2.
\end{gather}
For the definition of $\varphi_1^{(25)}$ we require
\begin{gather}
	\begin{split}\label{eqn:exp:vplast}
	\varphi^{(17)}_1(\tau,z)&:=4\varphi_1^{(5)}(\tau,z)\varphi_1^{(13)}(\tau,z)-\varphi_1^{(9)}(\tau,z)^2,\\
	\varphi^{(19)}_1(\tau,z)&:=\varphi_1^{(4)}(\tau,z)\varphi_1^{(16)}(\tau,z)+2\varphi_1^{(7)}(\tau,z)\varphi_1^{(13)}(\tau,z)-\varphi_1^{(5)}(\tau,z)\varphi_1^{(15)}(\tau,z),\\
	\varphi^{(21)}_1(\tau,z)&:=\varphi_1^{(5)}(\tau,z)\varphi_1^{(17)}(\tau,z)-2\varphi_1^{(9)}(\tau,z)\varphi_1^{(13)}(\tau,z),
	\end{split}
\end{gather}
in addition to the other $\varphi^{(m)}_k$ that have appeared already. Define $H^{(25)}_{2A,r}$ in the usual way for root systems with a type A component, by requiring
\begin{gather}
	H^{(18)}_{2A,r}(\tau):=-(-1)^rH^{(18)}_{1A,r}(\tau).
\end{gather}

\subsubsection{$\ell=30+15$, $X=D_{16}E_8$}\label{sec:exp:D16E8}

We have $m^X=30$, so the umbral McKay-Thompson series $H^{(30+15)}_g=(H^{(30+15)}_{g,r})$ associated to $g\in G^{(30+15)}=\{e\}$ is a $60$-vector-valued function with components indexed by $r\in \ZZ/60\ZZ$. We have $H^{(30+15)}_{e,r}=-H^{(30+15)}_{e,-r}$, $H^{(30+15)}_{e,r}=H^{(30+15)}_{e,30-r}$ for $1\leq r\leq 29$, and $H^{(30+15)}_{e,r}=0$ for $r=0\mod 2$, so it suffices to specify the $H^{(30+15)}_{e,r}$ for $r\in\{1,3,5,7,9,11,13,15\}$. 

Define
\begin{gather}
	\begin{split}
	H^{(30+15)}_{1A,1}(\t)&:=\frac12\left(H^{(30+6,10,15)}_{1A,1}+[-\tfrac{1}{120}]H^{(10+5)}_{3A,1}(\tau/3)\right),\\
	H^{(30+15)}_{1A,3}(\t)&:=[-\tfrac{9}{120}]H^{(10+5)}_{3A,3}(\tau/3),\\
	H^{(30+15)}_{1A,5}(\t)&:=[-\tfrac{25}{120}]H^{(10+5)}_{3A,5}(\tau/3),\\
	H^{(30+15)}_{1A,7}(\t)&:=\frac12\left(H^{(30+6,10,15)}_{1A,7}+[-\tfrac{49}{120}]H^{(10+5)}_{3A,3}(\tau/3)\right),\\
	H^{(30+15)}_{1A,11}(\t)&:=\frac12\left(H^{(30+6,10,15)}_{1A,1}-[-\tfrac{1}{120}]H^{(10+5)}_{3A,1}(\tau/3)\right),\\
	H^{(30+15)}_{1A,13}(\t)&:=\frac12\left(H^{(30+6,10,15)}_{1A,7}-[-\tfrac{49}{120}]H^{(10+5)}_{3A,3}(\tau/3)\right),\\
	H^{(30+15)}_{1A,15}(\t)&:=-[-\tfrac{105}{120}]H^{(10+5)}_{3A,5}(\tau/3).
		\end{split}
\end{gather}

\subsubsection{$\ell=30+6,10,15$, $X=E_8^3$}\label{sec:exp:3E8}

We have $m^X=30$, and $G^{(30+6,10,15)}=G^X\simeq S_3$. The umbral McKay-Thompson series $H^{(30+6,10,15)}$ 
is a $60$-vector-valued function with components indexed by $r\in \ZZ/60\ZZ$. We have
\begin{gather}
H^{(30+6,10,15)}_{g,r}(\tau) 
=\begin{cases}\label{eqn:intro-rammcktht1}
	\pm H^{(30+6,10,15)}_{g,1}&\text{ if $r=\pm 1,\pm 11,\pm 19,\pm 29\mod 60$,} \\ 
	\pm H^{(30+6,10,15)}_{g,7} &\text{ if $r=\pm 7,\pm 13,\pm 17,\pm 27\mod 60$,}\\
	0&\text{ else,}  
	\end{cases}
\end{gather}
so it suffices to specify the $H^{(30+6,10,15)}_{g,r}$ for $r=1$ and $r=7$. These functions may be defined as follows.
\begin{gather}
\begin{split}
	H^{(30+6,10,15)}_{1A,1}&:=-2\frac{1}{\eta(\tau)^2}
	\left(\sum_{k,l,m\geq 0}+\sum_{k,l,m<0}\right)
	(-1)^{k+l+m}q^{(k^2+l^2+m^2)/2+2(kl+lm+mk)+(k+l+m)/2+3/40}
	\\
	H^{(30+6,10,15)}_{2A,1}&:=-2\frac{1}{\eta(2\tau)}
	\left(\sum_{k,m\geq 0}-\sum_{k,m<0}\right)
	(-1)^{k+m}q^{3k^2+m^2/2+4km+(2k+m)/2+3/40}
	\\
	H^{(30+6,10,15)}_{3A,1}&:=-2\frac{\eta(\tau)}{\eta(3\tau)}
	\sum_{k\in\ZZ}(-1)^kq^{15k^2/2+3k/2+3/40}\\
	H^{(30+6,10,15)}_{1A,7}&=-2\frac{1}{\eta(\tau)^2}
	\left(\sum_{k,l,m\geq 0}+\sum_{k,l,m<0}\right)
	(-1)^{k+l+m}q^{(k^2+l^2+m^2)/2+2(kl+lm+mk)+3(k+l+m)/2+27/40}
	\\
	H^{(30+6,10,15)}_{2A,7}&=2\frac{1}{\eta(2\tau)}
	\left(\sum_{k,m\geq 0}-\sum_{k,m<0}\right)
	(-1)^{k+m}q^{3k^2+m^2/2+4km+3(2k+m)/2+27/40}
	\\
	H^{(30+6,10,15)}_{3A,7}&=-2\frac{\eta(\tau)}{\eta(3\tau)}
	\sum_{k\in\ZZ}(-1)^kq^{15k^2/2+9k/2+27/40}
\end{split}
\end{gather}

\subsubsection{$\ell=46+23$, $X=D_{24}$}\label{sec:exp:D24}

We have $m^X=22$, and $G^{(46+23)}=\{e\}$. The umbral McKay-Thompson series $H^{(46+23)}_e=(H^{(46+23)}_{e,r})$ is a $92$-vector-valued function with components indexed by $r\in \ZZ/92\ZZ$. We have $H^{(46+23)}_{e,r}=-H^{(46+23)}_{e,-r}$ and $H^{(46+23)}_{e,r}=0$ for $r=0\mod 2$, so it suffices to specify the $H^{(46+23)}_{e,r}$ for $r$ odd. Observing that $H^{(46+23)}_{e,r}=-H^{(46+23)}_{e,-r}$ we may determine $H^{(46+23)}_e$ by requiring that
\begin{gather}
	\psi^{(23/2)}_e(\tau,z)
	=-2i\mu_{23/2,0}(\tau,z)+
	\sum_{\substack{r\in\ZZ+1/2\\  r\mod 23}}
	e(-r/2)H^{(46+23)}_{g,2r}(\tau)\theta_{23/2,r}(\tau,z),
\end{gather}
where $\psi^{(23/2)}_e$ is the meromorphic Jacobi forms of weight $1$ and index $23/2$ defined by setting
\begin{gather}
\psi^{(23/2)}_{e}(\tau,z):=2i\frac{\theta_1(\tau,6z)}{\theta_1(\tau,2z)\theta_1(\tau,3z)}\eta(\tau)^3.
\end{gather}

\subsection{Rademacher Sums}\label{sec:mts:rad}

Let $\Gamma_{\infty}$ denote the subgroup of upper-triangular matrices in $\SL_2(\ZZ)$. Given $\alpha\in \RR$ and $\g\in \SL_2(\ZZ)$, define $\rad^{[\alpha]}_{1/2}(\g,\tau):=1$ if $\g\in \G_\infty$. Otherwise, set
\begin{gather}\label{eqn:mts:rad-radreg}
\rad^{[\alpha]}_{1/2}(\g,\tau):=
	e(-\alpha(\g\t-\g\infty))
	\sum_{k\geq 0}\frac{(2\pi i \alpha(\g\t-\g\infty))^{n+1/2}}{\Gamma(n+3/2)},
\end{gather}
where $e(x):=e^{2\pi i x}$.
Let $n$ be a positive integer, and suppose that $\nu$ is a multiplier system for vector-valued modular forms of weight $1/2$ on $\Gamma=\Gamma_0(n)$. Assume that $\nu=(\nu_{ij})$ satisfies $\nu_{11}(T)=e^{\pi i/2m}$, for some basis $\{\gt{e}_i\}$, for some positive integer $m$, where $T=\left(\begin{smallmatrix}1&1\\0&1\end{smallmatrix}\right)$. To this data, attach the Rademacher sum
\begin{gather}\label{eqn:conj:Rad}
	R_{\Gamma,\nu}(\tau)
	:=
	\lim_{K\to \infty}
	\sum_{\gamma\in\G_\infty\backslash\G_{K,K^2}}
	{\nu}(\g)e\left(-\frac{\gamma\tau}{4m}\right)\gt{e}_1
	\jac(\g,\tau)^{1/2}
	\rad^{[-1/4m]}_{1/2}(\g,\tau)
	,
\end{gather}
where $\Gamma_{K,K^2}:=\left\{\left(\begin{smallmatrix}a&b\\c&d\end{smallmatrix}\right)\in\Gamma\mid 0\leq c<K,\,|d|<K^2\right\}$, and $\jac(\gamma,\tau):=(c\tau+d)^{-1}$ for $\gamma=\left(\begin{smallmatrix}a&b\\c&d\end{smallmatrix}\right)$. If the expression (\ref{eqn:conj:Rad}) converges then it defines a mock modular form of weight $1/2$ for $\Gamma$ whose shadow is given by an explicitly identifiable Poincar\'e series. We refer to \cite{revrad} for a review of this, and to \cite{whalen} for a more general and detailed discussion. 

Convergence of (\ref{eqn:conj:Rad}) can be shown by rewriting the Fourier expansion as in \cite[Theorem 2]{whalen} in terms of a sum of Kloostermann sums weighted by Bessel functions. This expression converges at $w=1/2$ by the analysis discussed at the end of \S\ref{Proofs}, following the method of Hooley as adapted by Gannon. That analysis requires not only establishing that the expressions converge, but also explicitly bounding the rates of convergence.

For the special case that $X=A_8^3$ we require $8$-vector-valued functions $\check t^{(9)}_g=(\check{t}^{(9)}_{g,r})$ for $g\in G^X$ with order $3$ or $6$. For such $g$, define $\check{t}^{(9)}_{g,r}$, for $0<r<9$, by setting 
\begin{gather}\label{eqn:jac:t9s}
	\check{t}^{(9)}_{3A,r}(\t):=
	\begin{cases}
		0,&\text{ if $r\neq 0\mod 3$,}\\
		-\theta_{3,3}(\tau,0),&\text{ if $r=3$,}\\
		\theta_{3,0}(\tau,0),&\text{ if $r=6$,}\\
	\end{cases}
\end{gather}
in the case that $g$ has order  $3$, and 
\begin{gather}
	\check{t}^{(9)}_{6A,r}(\t):=
	\begin{cases}
		0,&\text{ if $r\neq 0\mod 3$,}\\
		-\theta_{3,3}(\tau,0),&\text{ if $r=3$,}\\
		-\theta_{3,0}(\tau,0),&\text{ if $r=6$,}\\
	\end{cases}
\end{gather}
when $o(g)=6$. Here $\theta_{m,r}(\tau,z)$ is as defined in (\ref{eqn:specfun-thmr}).

The following result is proved in \cite{mumcor}, using an analysis of representations of the metaplectic double cover of $\SL_2(\ZZ)$.
\begin{thm}[\!\!\!\cite{mumcor}]\label{thm:mts:rad}
Let $X$ be a Niemeier root system and let $g\in G^X$. Assume that the Rademacher sum $R^X_{\Gamma_0(n_g),\check{\nu}^X_g}$ converges. If $X\neq A_8^3$, or if $X=A_8^3$ and $g\in G^X$ does not satisfy $o(g)=0\mod 3$, then we have
\begin{gather}\label{eqn:mts:rad-HR}
	\check{H}^X_{g}(\t)=-2R^X_{\Gamma_0(n_g),\check{\nu}^X_g}.
\end{gather}
If $X=A_8^3$ and $g\in G^X$ satisfies $o(g)=0\mod 3$ then
\begin{gather}\label{eqn:mts:rad-HRt}
	\check{H}^X_{g,r}(\t)=-2R^X_{\Gamma_0(n_g),\check{\nu}^X_g}(\t)+\check{t}^{(9)}_{g}(\t).
\end{gather}
\end{thm}

The $X=A_1^{24}$ case of Theorem \ref{thm:mts:rad} was proven first in \cite{Cheng2011}, via different methods.

\clearpage

\addcontentsline{toc}{section}{References}

\providecommand{\bysame}{\leavevmode\hbox to3em{\hrulefill}\thinspace}
\providecommand{\MR}{\relax\ifhmode\unskip\space\fi MR }

\providecommand{\MRhref}[2]{%
  \href{http://www.ams.org/mathscinet-getitem?mr=#1}{#2}
}
\providecommand{\href}[2]{#2}

\Addresses

\end{document}